\documentclass{amsart}
\usepackage{amsthm}
\usepackage{amsmath}
\usepackage{amstext}
\usepackage{amsfonts}
\usepackage{latexsym}
\usepackage{amscd}
\usepackage{xypic}
\newcommand{\la}{\ensuremath{\rightarrow}}
\newcommand{\sheaf}{\ensuremath{\mathcal{O}}}
\theoremstyle{plain}
\newtheorem{theorem}{Theorem}[section]
\newtheorem{lemma}[theorem]{Lemma}
\newtheorem{proposition}[theorem]{Proposition}
\newtheorem{corollary}[theorem]{Corollary}

\theoremstyle{definition}
\newtheorem{definition}[theorem]{Definition}
\newtheorem{notation}[theorem]{Notation}

\newtheorem{remark}[theorem]{Remark}

\numberwithin{equation}{section}

\begin{document}

\title[Smoothings of schemes with non-isolated singularities]{Smoothings of schemes with non-isolated singularities}
\author{Nikolaos Tziolas}
\address{Department of Mathematics, University of Cyprus, P.O. Box 20537, Nicosia, 1678, Cyprus}
\email{tziolas@ucy.ac.cy}





\begin{abstract}
In this paper we study the deformation and $\mathbb{Q}$-Gorenstein deformation theory of schemes with non-isolated singularities. We obtain obstruction 
spaces for the existence of deformations and also for local deformations to exist globally. Finally we obtain explicit criteria in order for a pure and 
reduced scheme of finite type over a field $k$ to have smoothings and $\mathbb{Q}$-Gorenstein smoothings.  
\end{abstract}

\maketitle

\section{Introduction}

The purpose of this paper is to describe the deformation and $\mathbb{Q}$-Gorenstein deformation theory of schemes defined over a field $k$ with non-isolated singularities and to obtain criteria for the existence of smoothings and $\mathbb{Q}$-Gorenstein smoothings. The motivation for doing so comes from many different problems. Two of the most important ones are the compactification of the moduli space of surfaces of general type (and its higher dimensional analogues) and the minimal model program.

Let $0 \in C$ be the germ of a smooth curve and let $U=C-0$. It is well known~\cite{KoBa88},~\cite{Ale06} that any family 
$f_U \colon \mathcal{X}_U \rightarrow U$ of smooth surfaces of general type over $U$ can be completed in a unique way to a family 
$f \colon \mathcal{X} \rightarrow C$ such that 
$\omega_{\mathcal{X}/C}^{[k]}$ is invertible and ample for some $ k>0$ and the central fiber $X=f^{-1}(0)$ is a stable surface. A stable surface 
is a proper two-dimensional reduced scheme $X$ such that $X$ 
has only semi-log-canonical singularities and $\omega_X^{[k]}$ is locally free and ample for some $k>0$. 
Hence the moduli space of surfaces of general type can be compactified by adding the stable surfaces. 
Therefore it is interesting to know which stable surfaces are smoothable and which are not. For an overview of recent advances in this area, and the 
higher dimensional analogues, we refer the reader to~\cite{Ale06}.

There are two applications from the minimal model program related to the smoothability problem that we would like to mention.

\textit{1.} The outcome of the minimal model program starting with a smooth $n$-dimensional projective variety $X$ is a terminal projective variety $Y$ 
such that either $K_Y$ is nef, or $Y$ has a Mori fiber space structure, which means that there is a projective morphism $f \colon Y \la Z$ with 
$-K_Y$ $f$-ample. Suppose that the second case happens and $\dim Z=1$. Let $z \in Z$ and $Y_z= f^{-1}(z)$. Then $Y_z$ is a Fano variety of dimension 
$n-1$ and $Y$ is 
a $\mathbb{Q}$-Gorenstein smoothing $Y_z$. In general $Y_z$ has non-isolated singularities and may not even be normal. 
Hence the classification of Mori fiber spaces in dimension $n$ is directly related with the classification of 
smoothable Fano varieties of dimension $n-1$. 

\textit{2.} One of the two fundamental maps that appear in the context of the three dimensional minimal model program is an extremal neighborhood. A 3-fold terminal extremal neighborhood~\cite{Ko-Mo92} is a proper birational map $\Delta \subset Y \stackrel{f}{\la} X \ni P$ such that $Y$ is the germ of a 3-fold along a proper curve $\Delta$, $\Delta_{red}=f^{-1}(P)$, $X$ and $Y$ are terminal, and $-K_Y$ is $f$-ample. An extremal neighborhood is the local analogue of a flipping contraction or a divisorial contraction that contracts a divisor onto a curve. In this setting then, $Y$ is a one-parameter $\mathbb{Q}$-Gorenstein smoothing of the general member $H \in |\sheaf_Y|$. The singularities of $H$ are in general difficult to understand and it may even be non-normal. Of course there are natural higher dimensional analogues of the previous construction.

It is therefore of interest to study the deformation theory of schemes with non-isolated singularities and to obtain criteria for a scheme $X$ to be smoothable. The case when $X$ is a reduced scheme with normal crossing singularities has been extensively studied by R. Friedman~\cite{Fr83}. In particular he obtained a condition called $d$-semistability in order for $X$ to be smoothable with a smooth total space and he studied the obstruction theory for a $d$-semistable scheme to be smoothable. As an application of his methods he showed that any $d$-semistable $K3$ surface is smoothable. H. Pinkham and U. Persson have studied the problem of whether a $d$-semistable scheme is smoothable and they have obtained examples that this is not always so~\cite{PiPe83}. Later, Y. Kawamata and Y. Namikawa~\cite{KawNam94} have defined and studied the notion of logarithmic deformations of a normal crossing reduced scheme and they extended Friedman's result on the smoothability of normal crossing $K3$ surfaces, to higher dimensional normal crossing Calabi-Yau varieties.

Typically, one first studies this problem locally and then globally. The local problem is to study which singularities are smoothable and the global
is to find obstructions for the local smoothings to exist globally. If $X$ has isolated singularities only, then it is well known that $H^2(T_X)$
is an obstruction space for the globalization of the local deformations. Hence if $X$ is locally smoothable and $H^2(T_X)=0$ then $X$ itself is
smoothable. However, if the singular locus of $X$ has dimension bigger than one, then there are examples of locally smoothable varieties
whose obstruction in $H^2(T_X)$ is zero which are not globally smoothable~\cite{PiPe83}. The reason behind this is that if the singularities are not isolated, then there are many local automorphisms of deformations that do not lift to higher order. Another major difference between the isolated and non-isolated singularities case is that Schlessinger's cotangent cohomology sheaves $T^i(X)$ do not have finite support anymore but they are sheaves supported on the singular locus of $X$, and are in general very difficult to describe~\cite{Tzi08}.

In this paper we make an effort to present a systematic study of the deformation theory of schemes with positive dimensional singular locus
and write a few smoothability and nonsmoothability criteria. Some of the results that we prove are already known and many others are to our
knowledge new. We have tried to get the most general results with the fewest possible restrictions on the singularities. We hope this paper will be a useful reference to anyone using deformation theory.

This paper is organized as follows.

In section 3 we define the deformation functors $Def(Y,X)$ and $Def^{qG}(Y,X)$ where $Y \subset X$ is a closed subscheme of a scheme $X$ defined over a field $k$. If $Y=X$ then
these are the usual deformation and $\mathbb{Q}$-Gorenstein deformation functors of $X$. 
If $P \in X$ is an affine isolated singularity then $Def(P,X)=Def(P\in X)$ is the functor of 
algebraic deformations of isolated singularities defined by Artin~\cite[Definition 5.1]{Art76}. More generally, if $Y \not = X$ then these are deformation functors of $\widehat{X}$, the formal completion
of $X$ along $Y$ with certain algebraizability conditions that are explained in Definition~\ref{def1}. They are algebraic analogues of deformations of germs of analytic spaces. We also define the local deformation functors $Def_{loc}(Y,X)$ and $Def^{qG}_{loc}(Y,X)$ which parametrize local deformations of $Y \subset X$. In almost all applications, and for the deformation functors to have good properties, we will assume that $Y$ contains the singular locus of $X$.

In section 4 we describe the tangent spaces $\mathbb{T}^1(Y,X)$, $\mathbb{T}^1_{qG}(Y,X)$ of $Def(Y,X)$ and $Def^{qG}(Y,X)$. Moreover, in Proposition~\ref{local-to-global} we obtain the local to global sequence for the functors $Def(Y,X)$ and $Def^{qG}(Y,X)$ which is a generalization of the usual local to global sequence for $Def(X)$~\cite[Theorem 2.4.1]{Ser06}.

In section 5 we study the existence of a pro-representable hull for the deformation functors defined in section 3. It is known that $Def(Y,X)$ has a pro-representable hull if its tangent space $\mathbb{T}^1(Y,X)$ is finite dimensional~\cite{Sch68}. In Theorem~\ref{hull-of-D} we show that this also holds for $Def^{qG}(Y,X)$ and in Theorem~\ref{local-hull} we show that under some strong restrictions on the singularities of $X$, $Def_{loc}^{qG}(Y,X)$ and $Def_{loc}(Y,X)$ have a hull too. Finally in Proposition~\ref{finiteness-of-tangent-spaces} we exhibit some cases when  $\mathbb{T}^1(Y,X)$ and  $\mathbb{T}^1_{qG}(Y,X)$ are finite dimensional over the base field $k$.

In sections 6 and 7 we explain the main technical tool that we use to study the deformation theory of $X$, Kawamata's $T^1$-lifting
property~\cite{Kaw92},~\cite{Kaw97}.

In section 8 we use the $T^1$-lifting property to study the global deformation theory of $Y \subset X$. In particular, in Theorem~\ref{global-ob} we show that if $X$ is a pure and reduced scheme defined over a field of characteristic zero and $X-Y$ is smooth, then $\mathrm{Ext}_{\hat{X}}^2(\widehat{\Omega}_X,\sheaf_{\hat{X}})$ is an obstruction space to lift a deformation $X_n \in Def(Y,X)(A_n)$ to $A_{n+1}$, where $\hat{X}$ is the formal completion of $X$ along $Y$ and $A_{n}=k[t]/(t^{n+1})$. Moreover, we exhibit an explicit obstruction element.

In section 9 we study the problem of when local deformations of $Y \subset X$ exist globally. The main results are,
\begin{enumerate}
\item In Proposition~\ref{smoothness-of-pi} we show that under very strong restrictions on the singularities of $X$, the global to local map
\[
\pi \colon  Def(Y,X) \la Def_{loc}(Y,X)
\]
is smooth if $H^2(\widehat{T}_X)=0$, where $\widehat{T}_X$ is the completion of $T_X$ along $Y$. However, in general $\pi$ may fail to be smooth. This is in contrast to the case of isolated singularities when it is well known that the global to local map is smooth if $H^2(T_X)=0$.
\item To get around the failure of $\pi$ to be smooth, for any small extension
\[
0 \la J \la B \la A \la 0
\]
and any $X_A \in Def(Y,X)(A)$ we define the spaces $Def(X_A/A,B)$ and $Def_{loc}(X_A/A,B)$, parametrizing global and local liftings of $X_A$ to $B$ with certain local compatibility conditions that are explained in Definition~\ref{def-of-local-functors}. In Theorem~\ref{local-to-global-1} we describe them and we show that there is an exact sequence
\[
0 \la H^1(\widehat{T}_X \otimes J) \stackrel{\alpha}{\la} Def(X_A/A,B) \stackrel{\pi}{\la}   Def_{loc}(X_A/A,B)    \stackrel{\partial}{\la} H^2(\widehat{T}_X \otimes J)
\]
generalizing the first order global to local exact sequence. Moreover we show that there are two successive obstructions in $H^0(T^2(X)\otimes J)$ and 
$H^1(T^1(X) \otimes J)$ in order that $\mathrm{Def}_{loc}(X_A/A,B)\not= \emptyset$. If these obstructions vanish, then there is another obstruction in 
$H^2(\widehat{T}_X \otimes J)$ in order that $Def(X_A/A,B)\not = \emptyset$, i.e., for the local deformations to exist globally. 
These obstruction spaces were well known if $X=Y$~\cite{Har04}.
\end{enumerate}

In section 10 we extend all results obtained for the functor $Def(Y,X)$ to $Def^{qG}(Y,X)$. We do this by using the fact that locally 
any $\mathbb{Q}$-Gorenstein deformation of $X$ is induced by a deformation of its index 1 cover~\cite{KoBa88}.

Let $X$ be a scheme of finite type over a field $k$ and $f \colon \mathcal{X} \la \mathcal{S}$ be a deformation of $X$ over the spectrum of a discrete valuation ring $(R,m)$. In section 11 we compare properties of the global deformation $f$ with properties of the associated formal deformation $f_n \colon X_n \la S_n$, where $S_n=\mathrm{Spec}R/m^{n+1}$ and $X_n=\mathcal{X}\times_S S_n$. In particular we obtain criteria on the associated formal deformation in order for the global one to be a smoothing. This is important because the deformations obtained with our methods are only formal and not necessarily algebraic. Then if they are algebraic it is of interest to know which properties of the global deformation can be read from properties of the associated formal deformation.

In section 12  we apply the theory developed in the previous sections to give some smoothing and non-smoothing criteria for a pure and reduced scheme of finite type over a field $k$. The main results are,
\begin{enumerate}
\item Let $D$ be either $Def(X)$ or $Def^{qG}(X)$. In Theorem~\ref{non-smoothing-2} we show that if at any generic point of its
singular locus, $X$ is normal crossings and $H^0(p(T^1_D(X)))=H^1_Z(p(T^1_D(X)))=0$, then $X$ is not smoothable, where $p(T^1_D(X))$
is the quotient of $T^1_D(X)$ by its torsion and $Z$ is the support of the torsion part. As a special case we
get that if $X$ is normal crossings and $H^0(T^1(X))=0$, then $X$ is not smoothable.
\item In Theorem~\ref{smoothing1} we show that if $X$ is a locally smoothable $\mathbb{Q}$-Gorenstein scheme such that the index 1 covers of all its singular points have complete intersection singularities, $T_{qG}^1(X)$ is finitely generated by its global sections and $H^1(T_{qG}^1(X))=H^2(T_X)=0$, then $X$ has a formal $\mathbb{Q}$-Gorenstein smoothing. Various other more specialized smoothing criteria are given as well.
\end{enumerate}

In section 13 we apply the theory developed earlier in order to give examples in the context of the moduli of stable surfaces and the three dimensional 
minimal model program. First we give two examples of non-smoothable stable surfaces. The components of the moduli space of stable surfaces that these 
surfaces belong to, do not contain any smooth surfaces of general type and hence these are extra components that appear by compactifying the moduli space 
of surfaces of general type. Finally, by deforming a particular non-normal surface $H$, we construct a three dimensional divisorial extremal neighborhood 
$f \colon Y \la X$ such that $H$ is the general member of $|\sheaf_Y|$.

\section{Preliminaries}
\begin{enumerate}
\item All schemes in this paper are separated and Noetherian defined over a field $k$. Additional properties will be stated as needed.
\item  We denote by $Art(k)$ the category of Artin local $k$-algebras.
\item  For any coherent sheaf $\mathcal{F}$ on a scheme $X$, we denote $\mathcal{F}^{[n]}=(\mathcal{F}^{\otimes n})^{\ast \ast}$.
\item  Let $F \colon Art(k) \la Sets$  be a deformation functor. Then, following the notation introduced by Schlessinger~\cite{Sch68}, its tangent space is the set $F(k[t]/(t^2))$ and is denoted by $T^1_F$.
\item A small extension of local Artin $k$-algebras is a square zero extension
\[
0 \la J \la B \la A \la 0
\]
of local Artin $k$-algebras $(A,m_A)$, $(B,m_B)$ such that $J$ is a principal ideal of $B$ and $m_B J=0$ (and therefore $J \cong k$ as a $B$-module).
\item Let $X \la Y$ be a morphism of Noetherian separated schemes and $\mathcal{F}$ a coherent sheaf on $X$. Then by $T^i(X/Y,\mathcal{F})$ we denote Schlessingers
cotangent cohomology sheaves~\cite{Li-Sch67}.
\item Let $X$ be a scheme. A formal deformation of $X$, is a flat morphism of formal schemes $\mathfrak{f} \colon \mathfrak{X} \la \mathfrak{S}$, where $\mathfrak{S}=\mathrm{Specf R}$, $(R,m_R)$ is a complete local ring and such that $X \cong \mathfrak{X}\times_{\mathfrak{S}} \mathrm{Specf} (R/m_R)$. Equivalently, a formal deformation of $X$ over $(R,m_R)$, is a collection of compatible deformations $f_n \colon X_n \la \mathrm{Spec}R_n$, for all $n \in \mathbb{Z}_{>0}$, where $R_n=R/m_R^{n+1}$.
    Suppose that $X$ is of finite type over a field $k$. Then the formal deformation is called effective if and only if there is a flat morphism of finite type $f \colon \mathcal{X} \la \mathcal{S}=\mathrm{Spec}R$ of schemes with $X=\mathcal{X}\times_{\mathcal{S}}\mathrm{Spec}(R/m_R) =X$ and such that $\mathfrak{X}=\widehat{\mathcal{X}}$, the formal completion of $\mathcal{X}$ along $X$. In this case, $\mathfrak{f}$ is called the associated formal deformation of $f$. If in addition, $f$ is induced from a deformation $f^{\prime} \colon \mathcal{X}^{\prime} \la \mathrm{Spec} A$, where $(A,m_A)$ is a localization of a finitely generated $k$-algebra such that $\hat{A}\cong R$, then the deformation is called algebraic.
\item A reduced scheme $X$ is called $\mathbb{Q}$-Gorenstein if and only if it is Cohen-Macauley, Gorenstein in codimension 1 and there
is $n \in \mathbb{Z}_{>0}$ such that $\omega_X^{[n]}$ is invertible.
\item A smoothing of a scheme $X$ is a flat morphism $ f \colon \mathcal{X} \la T=\mathrm{Spec} R$, where $(R,m)$ is a discrete valuation ring, such that $\mathcal{X} \times_T \mathrm{Spec}(R/m) \cong X$ and the generic fiber $\mathcal{X} \times_T \mathrm{Spec} K(R) $ is smooth over $K(R)$. If in addition $X$ is $\mathbb{Q}$-Gorenstein and there is $n \in \mathbb{Z}_{>0}$ such that $\omega^{[n]}_{\mathcal{X}/T}$ is invertible, then the smoothing is called $\mathbb{Q}$-Gorenstein. To avoid degenerate situations we will assume that $X$ is either a local scheme and $f$ a morphism of local schemes, or $X$ and $f$ are proper and of finite type.
\end{enumerate}

\section{The Deformation Functors.}\label{definition-of-functors-section}

First we recall the definition of an \'etale neighborhood of a closed subscheme $Y$ of a scheme $X$~\cite{Cox78}.

\begin{definition}
Let $X$ be a Noetherian scheme defined over a field $k$ and $Y \subset X$ a closed subscheme of it. An \'{e}tale neighborhood of $Y$ in $X$
is an \'etale morphism $Z \la X$ such that $ Z\times_X Y \cong Y$.
\end{definition}
Next we define the deformation functors that we are going to study in this paper.
\begin{definition}\label{def1}
Let $X$ be a Noetherian scheme defined over a field $k$ and $Y \subset X$ a closed subscheme of it. Let $\hat{X}$ be the formal completion
of $X$ along $Y$. Then $Def(Y,X) \colon Art(k) \la Sets $ is the functor such that for any finite local Artin $k$-algebra $A$,
$Def(Y,X)(A)$ is the set of isomorphism classes of flat morphisms of formal schemes $ f \colon \mathcal{X} \la \mathrm{Specf} A$
such that
\begin{enumerate}
\item $\mathcal{X} \times_{\mathrm{Specf} A} \mathrm{Specf} k \cong \hat{X}$.
\item There is an open cover $\mathcal{U}_i$ of $\mathcal{X}$ and flat morphisms of schemes $f_i \colon U_i \la \mathrm{Spec} A$ such that
\begin{enumerate}
\item $U_i \times _{\mathrm{Spec} A} \mathrm{Spec} k$ is a local \'etale neighborhood of $Y$ in $X$.
\item $\mathcal{U}_i \la \mathrm{Specf} A$ is the formal completion of $U_i \la \mathrm{Spec} A$ along $Y$.
\end{enumerate}
\end{enumerate}
\end{definition}
Next we define the notion of $\mathbb{Q}$-Gorenstein deformations and the corresponding deformation functor $Def^{qG}(Y,X)$. In order for this to make sense it is necessary to define the notion of relative dualizing sheaves for a formal family as in~\ref{def1}.
\begin{definition}
Let $X$ be a Cohen-Macauley scheme Gorenstein in codimension one defined over a field $k$ and $Y \subset X$ a closed subscheme of it. Let
$ f \colon \mathcal{X} \la \mathcal{S}=\mathrm{Specf} A$ be an element of $Def(Y,X)(A)$, where $A \in Art(k)$. Let $\mathcal{U}_i$ be an open cover
of $\mathcal{X}$ as in Definition~\ref{def1}. Then the sheaves $(\omega_{U_i/A}^{[n]})^{\wedge}$ glue together to a coherent sheaf on $\mathcal{X}$ which
we denote by $\omega_{\mathcal{X}/\mathcal{S}}^{[n]}$. Note that the construction is independent of the cover chosen.
\end{definition}

\begin{definition}\label{def2}
Let $X$ be a $\mathbb{Q}$-Gorenstein scheme defined over a field $k$ and $Y \subset X$ a closed subscheme of it.
The functor of $\mathbb{Q}$-Gorenstein deformations
is the functor $Def^{qG}(Y,X) \colon Art(k) \la Sets$ such that for any finite local Artin $k$-algebra $A$, $Def^{qG}(Y,X)(A)$ is the set of
isomorphism classes of flat morphisms $\mathcal{X} \la \mathcal{S}=\mathrm{Specf} A$ in $Def(Y,X)$, such that the sheaf
$\omega^{[n]}_{\mathcal{X}/\mathcal{S}}$ is invertible for some $n \in \mathbb{Z}_{>0}$.
\end{definition}

It is not immediately clear that $Def^{qG}(Y,X)$ as defined above is a functor. This would be true if the property that
$\omega_{\mathcal{X}/\mathcal{S}}$ is $\mathbb{Q}$-Gorenstein is stable under base extension, which is known to be true~\cite[Lemma 2.6]{Has-Kov04}.

\begin{remark}
\begin{enumerate}
\item If $Y=X$, then the functors $Def(X,X)$ and $Def^{qG}(X,X)$ are just the familiar deformation functors $Def(X)$ and $Def^{qG}(X)$. 
\item Let $P \in X$ be an affine isolated singularity. Then it follows from the definitions and from 
Theorem~\ref{Artin},~\cite[Corollary 2.6]{Art69} that $Def(P,X)$ is the functor of algebraic deformations of an isolated 
singularity~\cite[Definition 5.1]{Art76}. Traditionally this functor is denoted by $Def(P \in X)$ and we will frequently use this notation too. 
More generally, if $X$ has isolated singularities and $Y=X^{sing}=\{P_1,\ldots , P_k\}$, then $Def(Y,X)=\prod_{i=1}^kDef (P_i \in X)$. 
\end{enumerate}
Moreover, as we will see later, in order to obtain 
reasonable results about $Def(Y,X)$ or $Def^{qG}(Y,X)$ (in particular, existence of pro-representable hulls), we will assume that $Y$ is proper and 
$X-Y$ is smooth.
\end{remark}

\begin{remark}
The functors $Def(Y,X)$ and $Def^{qG}(Y,X)$ are an attempt to get an algebraic analog of deformations of germs of analytic spaces. A candidate for an algebraic germ is the formal neighborhood. However, completion along a subscheme is not an algebraic construction. The algebraic analogs of local analytic neighborhoods are \'etale neighborhoods. Ideally we would like to define the notion of an algebraic germ in such a way so that if two are isomorphic then they are at least locally \'etale equivalent and any morphism between two algebraic germs should come, at least locally, from a morphism between \'etale neighborhoods. It is known~\cite[Theorem 4]{Cox78}, that if $Y \subset X_1$, $Y \subset X_2$ is an embedding of a scheme $Y$ into two schemes $X_1$, $X_2$ and $X_1^h \cong X_2^h$, then, under relatively mild hypotheses, the isomorphism is induced by a common \'etale neighborhood of $Y$ in $X_1$, $X_2$. However, it is possible that $\hat{X}_1 \cong \hat{X}_2$ but $X_1^h \not\cong X_2^h$, and hence $X_1$ and $X_2$ are not \'etale equivalent around $Y$~\cite[Example 1]{Cox78}. For these reasons the correct definition of the algebraic germ of $Y \subset X$ would be that of the henselization $X^h$ of $X$ along $Y$ instead of the completion $\hat{X}$. However, due to technical difficulties working with henselization, we work with the formal neighborhood and impose a local algebraizability condition in order not to get too far away from the geometry of $Y \subset X$. Moreover, in many cases the results of Artin~\cite{Art69} allow us to move between the formal and algebraic case.
\end{remark}

\begin{notation}
For the remaining part of this paper, whenever we speak of $Def(Y,X)$ or $Def^{qG}(Y,X)$, $X$ is assumed to satisfy all the relevant properties
stated in Definitions~\ref{def1},~\ref{def2}.
\end{notation}

One of the fundamental problems in deformation theory is to determine when a given scheme $X$ admits a smoothing. The natural approach
to this is first to study the problem locally, i.e, to study which singularities are smoothable and then to globalize the local
smoothings. If $X$ has
isolated singularities only, say $P_1, \ldots , P_k$, then the globalization of the local deformations is achieved by studying the natural
transformation of functors
\begin{equation}\label{isolated-sing}
D(X) \la \prod_{i=1}^k D(P_i, X)
\end{equation}
where $D(X)$ is either $Def(X)$ or $Def^{qG}(X)$ and $D(P_i,X)$ is either $Def(P_i,X)$ or $Def^{qG}(P_i,X)$.
If the singularities of $X$ are not isolated, then the above map does not exist. A kind of sheafification of the local
deformation functors is more appropriate in this case.
\begin{definition}
Let $D(Y,X)$ be either $Def(Y,X)$ or $Def^{qG}(Y,X)$. The functor $\underline{D}(Y,X)$ is the functor \[
\underline{D}(Y,X) \colon Art(k) \la Sh(X)\]
 defined as follows. For any finite local
$k$-algebra $A$,  $\underline{D}(Y,X)(A)$ is the sheaf associated to the presheaf $F$ defined by $F(V) = D(Y\cap V,V)(A)$ for any open set $V$.
\end{definition}
\begin{definition}
Let $D(Y,X)$ be either $Def(Y,X)$ or $Def^{qG}(Y,X)$.
The functor of local deformations of $D(Y,X)$ is the functor $D_{loc}(Y,X) \colon Art(k) \la Sets$ defined by \[
D_{loc}(Y,X)(A)=H^0(\underline{D}(Y,X)(A))\]
\end{definition}

For $D(Y,X)$ as above, there is a natural transformation of functors
\begin{equation}\label{pai}
\pi \colon D(Y,X) \la D_{loc}(Y,X)
\end{equation}
We call this map the local to global map. If $X$ has isolated singularities and $Y=X$, then $\pi$ extends~(\ref{isolated-sing}). 

\begin{remark}
If $X$ has isolated singularities and $H^2(T_X)=0$, then it is well known that $\pi$ is smooth. This is not so in general. This is due to the inability to lift local automorphisms of deformations to higher order. However under some strong conditions on the singularities of $X$, $\pi$ is still smooth (Proposition~\ref{smoothness-of-pi}).
\end{remark}

\section{The tangent space of $Def(Y,X)$ and $Def^{qG}(Y,X)$.}

Let $Y \subset X$ be a closed subscheme of a scheme $X$. In this section we describe the tangent spaces of the functors $Def(Y,X)$, $Def^{qG}(Y,X)$ and the local to global map $\pi$~(\ref{pai}) at the level of tangent spaces.
\begin{definition}
We denote by $ \mathbb{T}^1(Y,X)$, $T^1(Y,X)$, $\mathbb{T}^1_{qG}(Y,X)$ and $T^1_{qG}(Y,X)$ the tangent spaces of the functors $Def(Y,X)$, $\underline{Def}(Y,X)$,  $Def^{qG}(Y,X)$ and  $\underline{Def}^{qG}(Y,X)$, respectively. 
\end{definition}

It easily follows from the definitions of the deformation functors involved that $H^0(T^1(Y,X))$ and $H^0(T_{qG}^1(Y,X))$ are the tangent spaces of 
$Def_{loc}(Y,X)$ and $Def^{qG}_{loc}(Y,X)$, respectively. If $X-Y$ is smooth, then $T^1(Y,X)$ is just Schlessinger's $T^1(X)$ sheaf and  $T^1_{qG}(Y,X)$ is 
the subsheaf $T^1_{qG}(X)$ of $T^1(X)$ defined as follows. For any affine open subset $U \subset X$, $T^1_{qG}(X)(U)$ is the $\sheaf_X(U)$-module of 
isomorphism classes of first order $\mathbb{Q}$-Gorenstein deformations of $U$.

The next proposition describes the global to local map at the level of tangent spaces. If $X=Y$ and $D=Def(X)$ then this is just the familiar global to 
local sequence of the functor $Def(X)$~\cite[Theorem 2.4.1]{Ser06}.
\begin{proposition}\label{local-to-global}
Suppose that $X$ is a reduced scheme and $Y \subset X$ a closed subscheme. Then
\begin{enumerate}
\item Then there is a canonical injection \[
\phi \colon \mathbb{T}^1(Y,X) \la \mathrm{Ext}_{\hat{X}}(\widehat{\Omega}_X,\sheaf_{\hat{X}})
\]
which is an isomorphism if $X-Y$ is smooth.
\item Let $D$ be either $Def(Y,X)$ or $Def^{qG}(Y,X)$. Then there is an exact sequence
\[
0 \la H^1(\hat{T}_X ) \la \mathbb{T}_D^1(Y,X) \la H^0(\widehat{T^1}_D(X))
\]
If in addition $X-Y$ is smooth, then $\mathbb{T}^1_D(Y,X)=\mathbb{T}^1_D(X)$, $T^1_D(Y,X)=T^1_D(X)$ and there is an extended exact sequence
\[
0 \la H^1(\hat{T}_X) \la \mathbb{T}_D^1(X) \la H^0(T_D^1(X)) \la H^2(\hat{T}_X)
\]
\end{enumerate}
where $\hat{X}$ is the formal completion of $X$ along $Y$ and $\widehat{\Omega}_X$, $\hat{T}_X$, $\widehat{T^1}_D(X)$ are the coresponding completions of 
$\Omega_X$, $T_X$ and $T^1_D(X)$ along $Y$.
\end{proposition}
\begin{proof}
We first do the case $D=Def(Y,X)$.
The proof is based on the one for ordinary schemes~\cite[Theorem 2.4.1]{Ser06}. Let $\mathcal{X}_1 \la \mathrm{Specf}A_1 $ be a first order deformation of $\hat{X}$.
Then by definition there is an open cover
$\mathcal{U}_i$ of $\mathcal{X}_1$ such that $\mathcal{U}_i \cong \widehat {U_i}$, where $U_i$ is a first order deformation of a local \'etale
neighborhood $V_i$ of $Y$ in $X$. Then
the extension \[
0 \la k \la A_1 \la k \la 0\]
gives the extension
\[
0 \la \sheaf_{V_i} \la \sheaf_{U_i} \la \sheaf_{V_i} \la 0\]
and since $X$ is assumed to be reduced, there is an exact sequence
\[
0 \la \sheaf_{V_i} \la \Omega_{U_i}\otimes \sheaf_{V_i} \la \Omega_{V_i} \la 0
\]
and consequently
\[
 0 \la \sheaf_{\widehat{V_i}} \la \widehat{\Omega}_{U_i}\otimes \sheaf_{\widehat{V_i}} \la \widehat{\Omega}_{V_i} \la 0
\]
Patching them all together we get the exact sequence
\[
0 \la \sheaf_{\hat{X}} \la \widehat{\Omega}_{\mathcal{X}}\otimes \sheaf_{\hat{X}} \la \widehat{\Omega}_X \la 0
\]
Hence we get a map
\[
\mathbb{T}^1(Y,X) \la \mathrm{Ext}_{\hat{X}}(\widehat{\Omega}_X,\sheaf_{\hat{X}})
\]
which is injective, as in the usual scheme case. Conversely, let
\[
0 \la \sheaf_{\hat{X}} \la \mathcal{E} \la \widehat{\Omega}_X \la 0
\]
be any extension in $\mathrm{Ext}_{\hat{X}}(\widehat{\Omega}_X,\sheaf_{\hat{X}})$. Let $\hat{d} \colon \sheaf_{\hat{X}} \la \widehat{\Omega}_X $ be the
completion of the universal derivation of $X$
(For detailed definitions and properties of $\hat{d}$ and $\widehat{\Omega}_{\mathcal{X}}$, for any formal scheme $\mathcal{X}$ see~\cite{TaLoRo07}).
Then exactly as in the scheme case this gives a first order deformation $\mathcal{X}$ of $\hat{X}$. However in general it may not be locally the
completion of a deformation of a local \'etale
neighborhood of $Y$ in $X$.

The standard local to global spectral sequence gives
\begin{equation}\label{spectral}
0 \la H^1(\widehat{T_X}) \la \mathrm{Ext}^1_{\hat{X}}(\widehat{\Omega}_X,\sheaf_{\hat{X}})
\la H^0(\mathcal{E}xt^1_{\hat{X}}(\widehat{\Omega}_X,\sheaf_{\hat{X}})) \la H^2(\widehat{T_X})
\end{equation}
\textbf{Claim:}
\[
\mathcal{E}xt^1_{\hat{X}}(\widehat{\Omega}_X,\sheaf_{\hat{X}}) \cong \mathcal{E}xt^1_{X}(\Omega_X,\sheaf_{X})^{\wedge}
\]
In fact we will show that
\[
\mathcal{E}xt^1_{\hat{X}}(\widehat{\mathcal{F}},\widehat{\mathcal{P}}) \cong \mathcal{E}xt^1_{X}(\mathcal{F},\mathcal{P})^{\wedge}
\]
where $\mathcal{F}$ and $\mathcal{P}$ are coherent $\sheaf_X$-modules. This is a local result, so we may assume that $X=\mathrm{Spec}A$ and $Y=V(I)$,
where $I \subset A$ is an ideal. Then, since $\mathcal{F}$ is coherent, there is an exact sequence
\[
0 \la \sheaf_X^k \la \sheaf_X^m \la \mathcal{F} \la 0
\]
Applying $\mathcal{H}om_X(\; , \mathcal{P})$ and taking completions, we get the exact sequence
\[
\hat{\mathcal{P}}^m \la \hat{\mathcal{P}}^k \la \mathcal{E}xt^1_{X}(\mathcal{F},\mathcal{P})^{\wedge} \la 0
\]
Taking completions first and then applying $\mathcal{H}om_{\hat{X}}(\; , \hat{P})$, we get the exact sequence
\[
\hat{\mathcal{P}}^m \la \hat{\mathcal{P}}^k \la \mathcal{E}xt^1_{\hat{X}}(\widehat{\Omega}_X,\sheaf_{\hat{X}}) \la 0
\]
The claim now follows immediately from the last two exact sequences.

Since $X$ is reduced it follows that $T^1(X)=\mathcal{E}xt^1_X(\Omega_X,\sheaf_X)$. Hence from~(\ref{spectral}) we get the
exact sequence
\begin{equation}\label{eq-local-to-global}
0 \la H^1(\widehat{T}_X) \la \mathrm{Ext}^1_{\hat{X}}(\widehat{\Omega}_X,\sheaf_{\hat{X}})  \la H^0(\widehat{T^1}(X)) \la H^2(\widehat{T}_X)
\end{equation}
The space $H^1(\widehat{T}_X)$ classifies the first order locally trivial deformations of $\widehat{X}$~\cite{Halp76}, and
$\mathbb{T}^1(Y,X) \subset \mathrm{Ext}^1_{\hat{X}}(\widehat{\Omega}_X,\sheaf_{\hat{X}})$. Therefore there is an exact sequence
\[
0 \la H^1(\widehat{T}_X) \la \mathbb{T}^1(Y,X) \la H^0(\widehat{T^1}(X))
\]
as claimed. If in addition we have that $X-Y$ is smooth, then $T^1(X)$ is supported on $Y$ and hence $\widehat{T^1}(X)=T^1(X)$. Hence every first order
deformation $\mathcal{X}$
of $\widehat{X}$ arising from an element of $\mathrm{Ext}^1_{\widehat{X}}(\widehat{\Omega}_X, \sheaf_{\widehat{X}})$ is locally the completion of a local
deformation of $X$ and hence in this case
$\mathbb{T}^1(Y,X)=\mathrm{Ext}^1_{\widehat{X}}(\widehat{\Omega}_X, \sheaf_{\widehat{X}})$.
This together with the exact sequence~(\ref{eq-local-to-global}) give the exact sequence
claimed in the second part of the proposition.

It remains to consider the $\mathbb{Q}$-Gorenstein functor. Let $\psi \colon \mathbb{T}^1(Y,X) \la H^0(T^1(Y,X))$ be the global to local map
that was defined earlier. Then $H^0(T^1_{qG}(Y,X)) \subset H^0(T^1(Y,X))$ and $\mathbb{T}^1_{qG}(Y,X) = \psi^{-1}(H^0(T^1_{qG}(Y,X)))$. This together
with the results just proven for the usual deformations case give the corresponding ones for the $\mathbb{Q}$-Gorenstein case.
\end{proof}
\begin{remark}
From Proposition~\ref{local-to-global}, it follows that in order to have reasonable results concerning the tangent space of $Def(Y,X)$ or $Def^{qG}(Y,X)$, $X-Y$ must be smooth. From now on we will always assume this.
\end{remark}

\section{Existence of pro-representable hulls.}

In this section we investigate the existence of pro-representable hulls~\cite{Sch68} for all the deformation functors defined in 
section~\ref{definition-of-functors-section}. To do so we use the following result of Schlessinger.

\begin{theorem}[\cite{Sch68}]
Let $F \colon Art(k) \rightarrow Sets$ be a functor such that $F(k)$ is a single point . Let $A^{\prime} \rightarrow A$ and  
$A^{\prime\prime} \rightarrow A$ be morphisms in $Art(k)$ and consider the map
\begin{equation}\label{cond-H}
F(A^{\prime}\times_A A^{\prime\prime}) \rightarrow F( A^{\prime}) \times_{F(A)} F( A^{\prime\prime})
\end{equation}
Then 
\begin{enumerate}
\item $F$ has a pro-representable hull if and only if $F$ has the properties $(H_1)$, $(H_2)$, $(H_3)$ below.
\begin{enumerate}
\item[$(H_1)$] (\ref{cond-H}) is a surjection whenever $A^{\prime\prime} \rightarrow A$ is a small extension.
\item[$(H_2)$] (\ref{cond-H}) is a bijection when $A=k$ and $A^{\prime\prime}=k[t]/(t^2)$.
\item[$(H_3)$] $\dim_k T^1_F < \infty$.
\end{enumerate}
\item $F$ is pro-representable if and only if $F$ has the additional property $(H_4)$:
\[
F(A^{\prime} \times_A A^{\prime}) \rightarrow F(A^{\prime})\times_{F(A)} F(A^{\prime})
\]
is an isomorphism for any small extension $A^{\prime} \rightarrow A$.
\end{enumerate}
\end{theorem}

By using the criterion of the previous theorem, Schlessinger showed the following.

\begin{proposition}[\cite{Sch68}]
Let $X$ be a scheme defined over a field $k$. Then $Def(X)$ has a pro-representable hull if and only if $\dim \mathbb{T}^1(X) < \infty$.
\end{proposition}
The proof given by Schlessinger applies directly to $Def(Y,X)$ and therefore $Def(Y,X)$ has a pro-representable hull if and only 
if $\dim_k \mathbb{T}^1(Y,X) < \infty$.

In what follows first we present some cases when $\mathbb{T}^1(Y,X)$ and $\mathbb{T}^1_{qG}(Y,X)$ have finite dimension over $k$. Then we show that $Def^{qG}(Y,X)$ has a pro-representable hull if and only if $\dim_k \mathbb{T}^1_{qG}(Y,X) < \infty$ and finally we show that under some strong restrictions on the singularities of $X$, $Def_{loc}(Y,X)$ and $Def^{qG}_{loc}(Y,X)$ have a pro-representable hull too.

\begin{proposition}\label{finiteness-of-tangent-spaces}
Let $X$ be a reduced scheme and $Y \subset X$ a proper subscheme of it. Then $\mathbb{T}^1(Y,X)$ and $\mathbb{T}^1_{qG}(Y,X)$ have finite dimension over the base field $k$ in any of the following cases.
\begin{enumerate}
\item $X=Y$.
\item Both $X$ and $Y$ are proper and smooth and the normal bundle $N_{Y/X}$ of $Y$ in $X$ is ample.
\item $Y$ is contractible to an isolated singularity, i.e., there is a proper morphism $f \colon X \la Z$ such that $f(Y)$ is a point,
$X-Y \cong Z-f(Y)$, $Z-f(Y)$ is smooth
and $R^if_{\ast}\sheaf_X =0$, for all $i \geq 1$.
\item $\dim Y=1$, $X-Y$ is smooth and $I_Y/I_Y^{(2)}$ is ample, where $I_Y^{(2)}$ is the second symbolic power of the ideal sheaf $I_Y$ of $Y$ in $X$.
\end{enumerate}
\end{proposition}
\begin{proof}
We use Proposition~\ref{local-to-global}. Then the first part is immediate and the second part was proved by Hartshorne~\cite{Har68}. The third
part is well known in the analytic category
but due to the lack of reference we present a proof here. The result is local around $Y$, so we
may assume that $Z=\mathrm{Spec}A$, where
$(A,m)$ is the localization of a finitely generated $k$-algebra. Let $f\colon X \la Z$ the birational map in the assumption.
Now since $f$ is proper and birational, $H^1(T_X)$ is a finitely generated torsion $A$-module and hence
$H^1(T_X)^{\wedge}=H^1(T_X)$, where $H^1(T_X)^{\wedge}$ is the $m$-adic completion of $H^1(T_X)$. Then according to the formal functions theorem, \[
H^1(T_X) \cong H^1(\widehat{T_X})
\]
Dualizing the standard exact sequence
\[
f^{\ast}\Omega_Z \la \Omega_X \la \Omega_{X/Z} \la 0
\]
and taking into consideration that $f$ is birational, we get the exact sequence
\[
0 \la T_X \la (f^{\ast}\Omega_Z)^{\ast} \la M \la 0
\]
where $M$ is a coherent $\sheaf_X$-module supported on $Y$. Hence $\dim_k H^1(T_X) < \infty $ iff $\dim_k H^1((f^{\ast}\Omega_Z)^{\ast})< \infty$.
Moreover, there is a natural map
$\psi \colon f^{\ast}T_Z \la (f^{\ast}\Omega_Z)^{\ast}$ and the supports of both $\mathrm{Ker}(\psi)$, $\mathrm{CoKer}(\psi)$ are contained in $Y$.
Hence it suffices to show that
$\dim_k H^1(f^{\ast}T_Z)<\infty$. Since $Z$ is affine, there is an exact sequence
\[
0 \la N \la \sheaf_Z^m \la T_Z \la 0
\]
and hence an exact sequence
\[
0 \la Q \la f^{\ast}N \la \sheaf_X^m \la f^{\ast}T_Z \la 0
\]
where $Q$ is supported on $Y$. This breaks into two short exact sequences
\begin{gather*}
0\la Q \la f^{\ast}N \la M \la 0\\
0 \la M \la \sheaf_X^m \la f^{\ast}T_Z \la 0
\end{gather*}
Therefore since $R^1f_{\ast}\sheaf_X =0$ it now follows that $\dim_k H^1(f^{\ast}T_Z) < \infty$ iff $\dim_k H^2(f^{\ast}N)< \infty$. Repeating the above
argument and by induction the result follows.

It remains to show the last part. So, assume that $\dim Y=1$, $I_Y/I_Y^{(2)}$ is ample and that $X-Y$ is smooth.
Then by Proposition~\ref{local-to-global} it suffices to show that
$\dim_k H^1(\widehat{T_X}) < \infty$. The completion $\hat{X}$ of $X$ along $Y$ can be calculated via the
ideal sheaves $I_Y^{(n)}$ and hence
\[
H^1(\widehat{T_X})=\lim_{\longleftarrow}H^1(T_X \otimes \sheaf_X /I_Y^{(n)})
\]
The short exact sequence
\[
0 \la I_Y^{(n)}/I_Y^{(n+1)} \la \sheaf_X /I_Y^{(n+1)} \la \sheaf_X /I_Y^{(n)} \la 0
\]
gives the exact sequence
\[
0 \la K_n \la I_Y^{(n)}/I_Y^{(n+1)} \otimes T_X \la \sheaf_X /I_Y^{(n+1)} \otimes T_X \stackrel{\alpha_n}{\la} \sheaf_X /I_Y^{(n)}\otimes T_X \la 0
\]
We will show that $H^1(\mathrm{Ker}(\alpha_n))=0$, for $n$ sufficiently large and hence, since $Y$ is proper, $\dim_k H^1(\widehat{T_X}) < \infty $.
Since $\dim Y=1$, $H^2(K_n)$=0
and hence it suffices to show
that $H^1( I_Y^{(n)}/I_Y^{(n+1)} \otimes T_X)=0$, for $n$ sufficiently large. The natural map
\[
S^n(I_Y/I_Y^{(2)}) \la I_Y^{(n)}/I_Y^{(n+1)}
\]
is generically surjective along $Y$ and hence there exists an exact sequence
\[
S^n(I_Y/I_Y^{(2)})\otimes T_X \la I_Y^{(n)}/I_Y^{(n+1)}\otimes T_X \la T_n \la 0
\]
where $T_n$ has zero dimensional support. Since $I_Y/I_Y^{(2)}$ is ample, it follows that there is a $n_0 \in \mathbb{Z}$
such that $H^1(S^n(I_Y/I_Y^{(2)})\otimes T_X )=0$, for
all $n \geq n_0$. Therefore $H^1(I_Y^{(n)}/I_Y^{(n+1)}\otimes T_X)=0$ for all $n \geq n_0$ and hence $\dim_k H^1(\widehat{T_X})< \infty $ as claimed.
\end{proof}

\begin{theorem}\label{hull-of-D}
Let $X$ be a $\mathbb{Q}$-Gorenstein scheme and $Y \subset X$ a closed subscheme of it.
Assume furthermore that $\dim_k \mathbb{T}^1_{qG}(Y,X) < \infty$ ( this for example happens if $Y\subset X$ satisfy the conditions of
Proposition~\ref{finiteness-of-tangent-spaces}). Then the functor $Def^{qG}(Y,X)$ has a pro-representable hull.
\end{theorem}
\begin{proof}
We only do the case when $X=Y$. The general case is similar. For convenience set $D=Def^{qG}(Y,X)$.
We follow the general lines of the proof given by Schlessinger for the usual deformation functor
$Def(X)$~\cite[Proposition 3.10]{Sch68}.
It suffices to show that $D$ satisfies Schlessingers conditions $(H_1), (H_2), (H_3)$~\cite{Sch68}. $(H_3)$ is satisfied
by assumption and $(H_2)$ will follow from $(H_1)$ since it is satisfied for the usual deformation
functor $Def(Y,X)$.
Let $A^{\prime\prime} \la A$ and $A^{\prime} \la A$ be homomorphisms between Artin local $k$-algebras such that $A^{\prime\prime} \la A$ is a small
extension, i.e., there
is a square zero extension
\[
0 \la k \la A^{\prime\prime} \la A \la 0
\]
We will show that the natural map
\[
D(A^{\prime\prime} \times_A A^{\prime}) \la D(A^{\prime\prime})\times_{D(A)} D(A^{\prime})
\]
is surjective (this is condition $(H_1)$. Let $X_{A^{\prime\prime}} \in D(A^{\prime\prime})$, $X_{A^{\prime}}\in D(A^{\prime})$ and $X_A\in D(A)$ such that
 $X_{A^{\prime\prime}}\otimes_{A^{\prime\prime}} A = X_{A^{\prime}}\otimes_{A^{\prime}} A = X_A$. Then there are natural maps $\sheaf_{X_{A^{\prime\prime}}}\la \sheaf_{X_A}$ and
 $\sheaf_{X_{A^{\prime}}}\la \sheaf_{X_A}$. Let $R=A^{\prime\prime}\times_A A^{\prime}$ and let $X_R$ be the scheme $(|X|, \sheaf_{X_R})$,
where $|X|$ is the underlying topological space of $X$ and
$\sheaf_{X_R}=\sheaf_{X_{A^{\prime\prime}}}\times_{\sheaf_{X_A}} \sheaf_{X_{A^{\prime}}}$. Then $\sheaf_{X_R}$ is a flat $R$-algebra,
$\sheaf_{X_R} \otimes_R A^{\prime\prime}=
 \sheaf_{X_{A^{\prime\prime}}}$ and $\sheaf_{X_R} \otimes_R A^{\prime}=
 \sheaf_{X_{A^{\prime}}}$~\cite{Sch68}. To conclude the proof we must show that $X_R$ is $\mathbb{Q}$-Gorenstein, i.e., that
it is Cohen-Macauley, Gorenstein in codimension 1 and there is $n \in \mathbb{Z}$ such that $\omega^{[n]}_{{X_R}/R}$ is invertible. Since $X_R$ is a deformation
of $X$ over an Artin local ring $R$, it is Cohen-Macauley and Gorenstein in codimension 1. Let $n$ be the index of $X$. Then there is a natural map
 \[
 \phi \colon \omega^{[n]}_{{X_R}/R} \la \omega^{[n]}_{{X_{A^{\prime\prime}}/{A^{\prime\prime}}}} \times_{\omega^{[n]}_{{X_{A}/A}}} \omega^{[n]}_{{X_{A^{\prime}}/{A^{\prime}}}}
 \]
We will show that it is an isomorphism. First observe that since $X$ is $\mathbb{Q}$-Gorenstein of index $n$ and $X_A$, $X_{A^{\prime}}$, $X_{A^{\prime\prime}}$
are also $\mathbb{Q}$-Gorenstein, they have also index $n$~\cite{KoBa88} and hence the right hand side is invertible.
Since $\omega^{[n]}_{{X_R}/R}$ is reflexive and $X_R$ Cohen-Macauley it suffices to show that $\phi$ is an isomorphism over the Gorenstein locus. Let $X^0\subset X$
be the Gorenstein locus of $X$. Then there is a commutative diagram
\[
\xymatrix{
\omega^{[n]}_{{X^0_R}/R} \ar[r]\ar[d] &   \omega^{[n]}_{{X^0_{A^{\prime}}/{A^{\prime}}}}\ar[d]                             \\
  \omega^{[n]}_{{X^0_{A^{\prime\prime}}/{A^{\prime\prime}}}}\ar[r]   &     \omega^{[n]}_{{X^0_A}/A}
  }
\]
and moreover, since $\omega^{[n]}_{{X^0_R}/R}=\omega^{\otimes n}_{{X^0_R}/R}$,
$\omega^{[n]}_{{X^0_{A^{\prime\prime}}/{A^{\prime\prime}}}}=\omega^{\otimes n}_{{X^0_{A^{\prime\prime}}/{A^{\prime\prime}}}}$ are invertible,
\[
\omega^{[n]}_{{X^0_R}/R} \otimes_R A^{\prime\prime} = \omega^{[n]}_{{X^0_{A^{\prime\prime}}/{A^{\prime\prime}}}}
\]
Hence~\cite[Corollary 3.6]{Sch68}
\[
\omega^{[n]}_{{X^0_R}/R} \cong \omega^{[n]}_{{X^0_{A^{\prime\prime}}/{A^{\prime\prime}}}} \times_{{\omega^{[n]}_{{X^0_{A}/A}}}} \omega^{[n]}_{{X^0_{A^{\prime}}/{A^{\prime}}}}
\]
as claimed and therefore $X_R$ is $\mathbb{Q}$-Gorenstein.
\end{proof}
The next proposition shows that under some strong restrictions on the singularities of $X$, the local deformation functors $Def_{loc}(Y,X)$ and $Def^{qG}_{loc}(Y,X)$
have a hull too. This is useful in the cases when $Def(Y,X)$ and $Def^{qG}(Y,X)$ do not have a hull. The reason of this failure is that they may not have finite
dimensional tangent spaces. However, the tangent spaces of the local functors are $H^0(T^1(Y,X))$ and $H^0(T^1_{qG}(Y,X))$ and since $T^(Y,X)$, $T^1_{qG}(Y,X)$ are coherent sheaves supported on the
singular locus of $X$, $H^0(T^1(Y,X))$ and $H^0(T^1_{qG}(Y,X))$ will be finite dimensional if the singular locus of $X$ is proper and is contained in $Y$.
\begin{theorem}\label{local-hull}
Let $X$ be a scheme and $Y \subset X$ a subscheme of it. Assume that the singular locus $Z$ of $X$ is proper and that $Z \subset Y$. Let $D$ be either $Def(Y,X)$ or $Def^{qG}(Y,X)$. Suppose that one of
the following conditions are satisfied
\begin{enumerate}
\item With the exception of finitely many singular points, $D$ locally satisfies Schlessingers condition $(H_4)$;
\item The codimension of $Z$ in $X$ is at least $3$ and $\mathrm{depth}_P(\sheaf_{X,P}) \geq 3$, for any point $P\in Z$ (closed or not)
\end{enumerate}
Then the local deformation functor $D_{loc}$ has a pro-representable hull.
\end{theorem}

\begin{proof}
We only do the case when $Y=X$. The proof of the general case is similar.

It suffices to verify Schlessinger's conditions $(H_1)$, $(H_2)$ and $(H_3)$. The tangent space of $D_{loc}$ is $H^0(T^1_D(X))$. Since $T^1_D(X)$
is a coherent sheaf supported on the singular locus of $X$, $H^0(T^1_D(X))$ is finite dimensional over the base field $k$.
So $(H_3)$ is satisfied.

Assume now that either one of the conditions in the statement is satisfied. If the second one holds, then $Def(X)=Def(X-Z)$~\cite{Art76} and since $X-Z$
is smooth, it locally satisfies $(H_4)$.
hence we only need to assume that the first condition is satisfied.

Let $A^{\prime\prime} \la A$ and $A^{\prime} \la A$ be homomorphisms between Artin local $k$-algebras such that $A^{\prime\prime} \la A$ is a small
extension. We will show $(H_1)$, i.e.,  that
the natural map
\[
D_{loc}(A^{\prime\prime}\times_A A^{\prime}) \la D_{loc}(A^{\prime\prime})\times_{{D_{loc}(A)}}D_{loc}(A^{\prime})
\]
is surjective. By definition, $D_{loc}(B)=H^0(\underline{D}(B))$, for any local finite $k$-algebra $B$. Let $s^{\prime}\in H^0( \underline{D}(A^{\prime}))$
and $s^{\prime\prime}\in H^0( \underline{D}(A^{\prime\prime}))$ such that they map to $s\in H^0( \underline{D}(A))$ under the natural
maps $\lambda^{\prime}\colon H^0( \underline{D}(A^{\prime}))\la H^0( \underline{D}(A))$ and
$\lambda^{\prime\prime} \colon H^0( \underline{D}(A^{\prime\prime}))\la H^0( \underline{D}(A))$. Let $\{U_i\}$ be an affine open cover of $X$ and
let $U_{ij}=U_i\cap U_j$. Let $\mathcal{X}_i$ be any deformation of $U_i$ over a ring $B$. In what follows we will denote by
$\mathcal{X}_{ij}$ its restriction on $U_{ij}$.

The section $s$ is equivalent to a collection of deformations $\mathcal{X}_i$ of $U_i$ over $A$ and $A$-isomorphisms
$\phi_{ij} \colon \mathcal{X}_{ij} \la \mathcal{X}_{ji}$.
Similarly $s^{\prime}$ is equivalent to a collection
of deformations $\mathcal{X}^{\prime}_i$ of $U_i$ over $A^{\prime}$ and $A^{\prime}$-isomorphisms
$\phi^{\prime}_{ij} \colon \mathcal{X}^{\prime}_{ij} \la \mathcal{X}^{\prime}_{ji}$
and $s^{\prime\prime}$ is equivalent to a collection
of deformations $\mathcal{X}^{\prime\prime}_i$ of $U_i$ over $A^{\prime\prime}$ and $A^{\prime\prime}$- isomorphisms $\phi^{\prime\prime}_{ij}
\colon \mathcal{X}^{\prime\prime}_{ij} \la \mathcal{X}^{\prime\prime}_{ji}$. Since $\lambda^{\prime}(s^{\prime})=\lambda^{\prime\prime}(s^{\prime\prime})=s$,
it follows that there are $A$-isomorphisms $\psi_i^{\prime}\colon \mathcal{X}^{\prime}_i\otimes_{A^{\prime}}A \la \mathcal{X}_i$ and
$\psi_i^{\prime\prime}\colon \mathcal{X}^{\prime\prime}_i\otimes_{A^{\prime\prime}}A \la \mathcal{X}_i$. Then
$ \sheaf_{\mathcal{X}^{\prime}_i} \times_{\sheaf_{\mathcal{X}_i}} \sheaf_{\mathcal{X}^{\prime\prime}_i}$
is a deformation of $U_i$ over $R=A^{\prime\prime}\times_A A^{\prime}$. The collection
$ \{\sheaf_{\mathcal{X}^{\prime}_i} \times_{\sheaf_{\mathcal{X}_i}} \sheaf_{\mathcal{X}^{\prime\prime}_i}\}$
form a section in $H^0(\underline{D}(R))$ iff there are $R$-isomorphisms
\[
\lambda_{ij} \colon \sheaf_{\mathcal{X}^{\prime}_{ij}} \times_{\sheaf_{\mathcal{X}_{ij}}} \sheaf_{\mathcal{X}^{\prime\prime}_{ij}} \la
\sheaf_{\mathcal{X}^{\prime}_{ji}} \times_{\sheaf_{\mathcal{X}_{ji}}} \sheaf_{\mathcal{X}^{\prime\prime}_{ji}}
\]
The natural candidate for such an isomorphism is
\[
\phi_{ij}^{\prime} \times \phi_{ij}^{\prime\prime}\colon \sheaf_{\mathcal{X}^{\prime}_{ij}}\times\sheaf_{\mathcal{X}^{\prime\prime}_{ij}} \la
\sheaf_{\mathcal{X}^{\prime}_{ji}}\times\sheaf_{\mathcal{X}^{\prime\prime}_{ji}}
 \]
This isomorphism induces an isomorphism of
$ \sheaf_{\mathcal{X}^{\prime}_{ij}} \times_{\sheaf_{\mathcal{X}_{ij}}} \sheaf_{\mathcal{X}^{\prime\prime}_{ij}}$ if and only if there is a commutative diagram
\[
\xymatrix{
\sheaf_{\mathcal{X}^{\prime}_{ij}}\otimes_{A^{\prime}}A \ar[dd]_{{\phi_{ij}^{\prime}}}\ar[dr]^{{\psi^{\prime}_{ij}}} &
& \sheaf_{\mathcal{X}^{\prime\prime}_{ij}}\otimes_{A^{\prime\prime}}A \ar[dd]^{{\phi_{ij}^{\prime\prime}}}\ar[dl]_{{\psi^{\prime\prime}_{ij}}} \\
                                        &  \sheaf_{\mathcal{X}_{ij}} \ar[dd]^{\phi_{ij}}                           & \\
\sheaf_{\mathcal{X}^{\prime}_{ji}}\otimes_{A^{\prime}}A \ar[dr]^{{\psi^{\prime}_{ji}}}        &
& \sheaf_{\mathcal{X}^{\prime\prime}_{ji}}\otimes_{A^{\prime\prime}}A \ar[dl]_{{\psi^{\prime}_{ji}}} \\
                                        &  \sheaf_{\mathcal{X}_{ji}}                          &
}
\]
By our assumption, we can refine the open cover in such a way that $U_{ij}$ satisfies $(H_4)$. Then we can modify the $\phi_{ij}$ so
that the left hand side of the diagram commutes and then, since $U_{ij}$ satisfies $(H_4)$, we lift them to $X^{\prime\prime}_{ij}$.
Hence we get a section and therefore $D_{loc}$ satisfies $(H_1)$. Similarly it also satisfies $(H_2)$ (note that $(H_2)$ is satisfied without
any restrictions on the singularities of $X$), and hence $D_{loc}$ has a hull.
\end{proof}
Next we present a simple case when cases when $D_{loc}$ has a hull.
\begin{corollary}
Assume that, with the exception of finitely many singular points, the index 1 cover of any singular point of $X$ is smooth and that the singular locus of $X$ is proper. Then $Def^{qG}_{loc}(Y,X)$ has a hull.
 \end{corollary}
 \begin{proof}
By Theorem~\ref{local-hull} all we need is to show that with the exception of finitely many singular points, property $(H_4)$ is satisfied. This is equivalent to showing that local automorphisms of deformations lift to higher order. Since the result is local, we may assume that $X$ is affine.
Then let $\pi \colon \tilde{X} \la X$ be the index 1 cover of $X$. Let $X_A$ be a $\mathbb{Q}$-Gorenstein deformation of $X$ over $A$. Let $A \la B$
a finite local $A$-algebra and $X_B=X_A \otimes_A B$. Let $\theta$ be a $B$-automorphism of $X_B$.
 Let $\tilde{X}_A \la X_A$ be the index 1 cover of $X_A$. Then $\tilde{X}_A$ is a deformation of $\tilde{X}$~\cite{KoBa88} over $A$
and $\tilde{X}_A\otimes_A B$ is the index 1 cover of $X_B$. From the construction of the index 1 cover, $\theta$ lifts to an automorphism
of $\tilde{X}_B$ which is smooth by assumption. This now lifts to an automorphism of $\tilde{X}_A$ and hence to an automorphism of $X_A$.
 \end{proof}
\begin{remark}
From the proof of Theorem~\ref{local-hull}, it is clear that the obstruction for the local deformation functors to have a hull is the presence
of automorphisms. The only cases that we were able to show existence of a hull is in fact when there are no automorphisms. In view of this,
perhaps it would be better to consider the stack of deformations instead.
\end{remark}

\section{The $T^1$-lifting property.}\label{T1-section}

The main technical tool that we will use in order to study the deformation theory of a scheme $X$ is Kawamata's $T^1$-lifting property~\cite{Kaw92},~\cite{Kaw97}. We recall
the basic definitions and properties.

Let $D \colon Art(k) \la Sets$ be a deformation functor of some scheme $X$ defined over a field of characteristic zero, i.e., a covariant functor
that satisfies
Schlessinger's conditions $(H_1)$ and $(H_2)$. Assume moreover that $D$ has an obstruction space $T_D^2$.
For $A\in Art(k)$, $D(A)$ is the set of isomorphism classes of pairs $(X_A,\phi_0)$ consisting of deformations $X_A$ of $X$ and marking
isomorphisms $\phi_0 \colon X_A\otimes_A k \la X$.
The class of $(X_A,\phi_0)$ will be denoted by $[X_A,\phi_0]$.

Let $B_n=k[x,y]/(x^{n+1},y^{2})$ and $C_n=k[x,y]/(x^{n+1},y^{2},x^ny)$. There are natural maps
$\alpha_n \colon A_{n+1}\la A_n$, $\beta_n \colon B_n \la A_n$, $\gamma_n \colon B_n \la C_n$, $\delta_n \colon C_n \la B_{n-1}$, $\zeta_n \colon A_n \la C_n$ and $\varepsilon_n \colon A_{n+1} \la B_n$ with $\beta_n(x)=t$, $\beta_n(y)=0$, $\varepsilon_n(t)=x+y$, $\zeta_n(t)=x+y$.

\begin{definition}
Let $[X_n, \phi_0] \in D(A_n)$. Then we define
\begin{enumerate}
\item $\mathbb{T}_D^1(X_n/A_n)$ to be the set of isomorphism classes of pairs $(Y_n, \psi_n)$ consisting of deformations $Y_n$ of $X$ over $B_n$ and marking isomorphisms
$\psi_n \colon Y_n\otimes_{B_n}A_n \la X_n$.
\item $T_D^1(X_n/A_n)$ to be the sheaf of sets on $X$ associated to the presheaf $\mathcal{F}$ such that for any open $U \subset X$, $\mathcal{F}(U)=\mathbb{T}^1_D(U_n/A_n)$, where $U_n=\mathcal{X}_n|_U$.
\end{enumerate}
If $D$ is $Def(Y,X)$ or $Def^{qG}(Y,X)$ then we use the notation $\mathbb{T}^1(X_n/A_n)$, $T^1(X_n/A_n)$, $\mathbb{T}_{qG}^1(X_n/A_n)$ and $T_{qG}^1(X_n/A_n)$, respectively.
\end{definition}

\begin{definition}[~\cite{Kaw92},~\cite{Kaw97}]
We say that the deformation functor $D$ satisfies the $T^1$-lifting property if and only if for any $X_n \in D(A_n)$ the natural map
\[
\phi_n \colon \mathbb{T}^1_D(X_n/A_n)\la \mathbb{T}^1_D(X_{n-1}/A_{n-1})
\]
is surjective, where $X_{n-1}=D(\alpha_{n-1})(X_n)$.
\end{definition}

\begin{theorem}[{~\cite[Theorem 1]{Kaw92}}]\label{T1-lifting-property}
Let $D$ be a deformation functor that satisfies the $T^1$-lifting property. Then $D$ is smooth. In particular,
if $D$ has a hull, then its hull is smooth.
\end{theorem}

In fact the proof of the previous theorem shows the following.

\begin{theorem}\label{T1}
Let $D$ be a deformation functor, $X_n\in D(A_n)$, $X_{n-1}=D(\alpha_n)(X_n)$ and
$Y_{n-1}=D(\varepsilon_{n-1})(X_n)\in \mathbb{T}_D^1(X_{n-1}/A_{n-1})$.
Then $X_n$ lifts to $A_{n+1}$, i.e., is in the image of $D(A_{n+1}) \la D(A_n)$,
if and only if $Y_{n-1}$ is in the image of the natural map
\[
\phi_n \colon \mathbb{T}^1_D(X_n/A_n)\la \mathbb{T}^1_D(X_{n-1}/A_{n-1}).
\]
\end{theorem}
The advantage of the last theorem is that it allows us to exhibit in the next section a very explicit obstruction element to lift $X_n$ to $A_{n+1}$. The following result is also useful.

\begin{proposition}\label{T1-1}
With assumptions as in Theorem~\ref{T1}, let $Y_n \in \mathbb{T}^1(X_n/A_n)$ be a lifting of $Y_{n-1}$, i.e., $\phi_n(Y_n)=Y_{n-1}$. Then there is a lifting $X_{n+1}$ of $X_n$ over $A_{n+1}$ such that $Y_{n}=D(\varepsilon_{n})(X_{n+1})$.
\end{proposition}

The proof of the proposition depends on the following result of Schlessinger.

\begin{theorem}[{~\cite{Sch68}}]\label{action}
Let $D \colon Art(k) \la Sets$ be a functor that satisfies $(H_2)$. Let
\[
0 \la J \la B \stackrel{\alpha}{\la} A \la 0
\]
be a small extension of local Artin $k$-algebras and let $D(\alpha) \colon D(B) \la D(A)$ be the natural map. Then for any $\xi_A \in D(A)$, there is a natural action of the tangent space $t_D$ of $D$ on the set $D(\alpha)^{-1}(\xi_A)$. Moreover, if $D$ satisfies $(H_1)$, then the action is transitive.
\end{theorem}
A careful look at the proof of the previous theorem reveals that the action described satisfies the following functorial property.
\begin{corollary}\label{functoriality-of-action}
With assumptions as in Theorem~\ref{action}, let
\[
\xymatrix{
0 \ar[r]  & J  \ar[d]^{f}\ar[r] & B \ar[d]^{g} \ar[r]^{\alpha} & A \ar[d]^{h} \ar[r] & 0 \\
0 \ar[r]  & J^{\prime}  \ar[r] & B^{\prime} \ar[r]^{\alpha^{\prime}} & A^{\prime} \ar[r] & 0 \\
}
\]
be a commutative diagram of small extensions of local Artin $k$-algebras such that $f$ is a $k$-isomorphism. Let $\xi_A \in D(A)$ and $\xi_{A^{\prime}}=D(h)(\xi_A) \in D(A^{\prime})$. Then the natural map $D(\alpha)^{-1}(\xi_A) \la D(\alpha^{\prime})^{-1}(\xi_{A^{\prime}})$ is $t_D$-equivariant.
\end{corollary}
If $f$ is not an isomorphism, then the previous result is not true.

\begin{proof}[Proof of Proposition~\ref{T1-1}]
Let $\zeta_n \colon A_{n} \la C_n$ be defined by $\zeta_n(t)=x+y$ and $\delta_n \colon C_n \la B_{n-1}$ be the natural map. Then $\delta_n \zeta_n = \varepsilon_{n-1}$.
Consider the commutative diagram of small extensions
\[
\xymatrix{
0 \ar[r]  & J  \ar[d]^{f}\ar[r] & A_{n+1} \ar[d]^{\varepsilon_n} \ar[r] & A_n \ar[d]^{\zeta_n} \ar[r] & 0 \\
0 \ar[r]  & J^{\prime}  \ar[r] & B_n \ar[r]  & C_n \ar[r] & 0 \\
}
\]
where $J=(t^{n+1})$, $J^{\prime}=(xy^n)$ and $f$ is the isomorphism given by sending $t^{n+1}$ to $xy^n$. The above diagram induces the following commutative diagram
\[
\xymatrix{
D(A_{n+1}) \ar[r]^{D(\alpha_n)} \ar[d]^{D(\varepsilon_n)} & D(A_n)  \ar[d]^{D(\zeta_n)} \ar[r] & T^2_D \otimes J \ar@{=}[d] \\
D(B_n) \ar[r]^{D(\gamma_n)} & D(C_n) \ar[r] & T^2_D \otimes J^{\prime}
}
\]
where $T^2_D$ is an obstruction space for $D$. Let $Z_n= D(\zeta_n)(X_n)$. Then the $T^1$-lifting property implies that $D(\gamma_n)(Y_n)=Z_n$~\cite{Kaw92},~\cite{Kaw97}. Let $X_{n+1}^{\prime}$ be a lifting of $X_n$, which exists by the $T^1$-lifting property, and $Y_n^{\prime}=D(\varepsilon)(X_{n+1}^{\prime})$. Then $Y_n, Y^{\prime}_n \in D(\gamma_n)^{-1}(Z_n)$ which is a homogeneous $t_D$-space by Theorem~\ref{action}. Hence there is $ \theta \in t_D$ such that $\theta \cdot Y_n^{\prime}=Y_n$. Moreover, by Corollary~\ref{functoriality-of-action}, the natural map $D(\alpha_n)^{-1}(X_n) \la D(\zeta_n)^{-1}(Z_n)$ is $t_D$-equivariant. Hence $D(\varepsilon_n)(X_{n+1}) = Y_n$, where $X_{n+1}= \theta \cdot X^{\prime}_{n+1}$.

\end{proof}
\begin{remark}
The $T^1$-lifting property was originally introduced by Ran~\cite{Ran92} in order to study infinitesimal deformations of a complex manifold
and was later generalized by Kawamata~\cite{Kaw92},~\cite{Kaw97} to the case of an arbitrary deformation functor $D$. Later, a stronger
version of the $T^1$-lifting property was introduced by Fantechi and Manetti~\cite{FaMa99}. According to their definition, a deformation functor
$D$ has the $T^1$-lifting property  if, for any $n \in \mathbb{N}$, the natural map
\[
D(B_{n+1}) \la D(B_n) \times_{D(A_n)} D(A_{n+1})
\]
is surjective and they show that if $D$ has the $T^1$-lifting property and $k$ has characteristic zero, then $D$ is smooth.
Then naturally, for any $X_n \in D(A_n)$ one can define $T_D^1(X_n/A_n)= \{ Y_n \in D(B_n) , \; D(\beta_n)(Y_n)=X_n\}$ and then $D$
has the new $T^1$-lifting property if and only if the natural map $T^1(X_n/A_n) \la T^1(X_{n-1}/A_{n-1})$ is surjective for any $X_n \in D(A_n)$.
This is a stronger condition since it depends only on $D$ and does not take into consideration any automorphisms of $X_n$. However, $T^1(X_n/A_n)$ does not have any natural $k$-vector space structure even in the case when $D=Def(X)$. For this reason we consider the weaker definition
given by Kawamata but which has the advantage that $\mathbb{T}^1_D(X_n/A_n)$ has a natural $k$-vector space structure if $D$ is either $Def(Y,X)$
or $Def^{qG}(Y,X)$, which are the cases of interest in this paper.
\end{remark}

\section{Description of $\mathbb{T}^1(X_n/A_n)$ and $T^1(X_n/A_n)$.}
Let $X$ be a pure and reduced scheme defined over a field of characteristic zero and $Y \subset X$ a closed subscheme of it such that $X-Y$ is smooth. Let $\mathcal{X}_n \in Def(Y,X)(A_n)$.
In this section we describe the spaces $\mathbb{T}^1(\mathcal{X}_n/A_n)$ and the sheaves $T^1(\mathcal{X}_n/A_n)$.

First we state a simple technical result that will be needed later.
\begin{lemma}\label{pure}
Let $X$ be a pure scheme and $X_R$ a deformation of it over a local Artin $k$-algebra $R$. Then $X_R$ is also pure.
\end{lemma}
\begin{proof}
The proof will be by induction on the length $l(R)$ of $R$. If $l(R)=1$ then $X_R=X$ which by assumption is pure. Now for any Artin ring $R$, the
maximal ideal $m$ has a composition sequence
$(0)=I_0\subset I_1 \subset \cdots I_{k-1}\subset I_k=m$ such that $I_k/I_{k+1} \cong R/m$. Since $I_1 = A/m$ and $I_1 \la I_1/I_1^2$ is surjective,
it follows that $I_1^2=0$. Hence there
is a square zero extension
\[
0\la k \la R \la B \la 0
\]
which gives the square zero extension
\[
0 \la \sheaf_X \la \sheaf_{X_R} \stackrel{p}{\la} \sheaf_{X_B} \la 0
\]
Let $J \subset \sheaf_{X_R}$ an ideal sheaf such that $\dim \mathrm{Supp}(J) < \dim X$. Then by induction $p(J)=0$ and hence $J \subset \sheaf_X$ and
hence $J=0$ since $X$ is pure.
\end{proof}
\begin{proposition}\label{T1-formula}
Suppose that $X$ is a pure and reduced scheme, $Y\subset X$ a closed subscheme and $X-Y$ is smooth. Let
$\mathcal{X}_n \in Def(Y,X)(A_n)$.
Then
\[
\mathbb{T}^1(\mathcal{X}_n/A_n)\cong \mathrm{Ext}^1_{\mathcal{X}_n}(\widehat{\Omega}_{\mathcal{X}_n/A_n},\sheaf_{\mathcal{X}_n})
\]
and
\[
T^1(\mathcal{X}_n/A_n)\cong \mathcal{E}xt^1_{{\mathcal{X}}_n}(\widehat{\Omega}_{\mathcal{X}_n/A_n},\sheaf_{\mathcal{X}_n})
\]
\end{proposition}
\begin{proof}
The proof is along the lines of Proposition~\ref{local-to-global}. We will only show the first isomorphism, the proof of second is identical.
Let $\{\mathcal{U}^i_n\}$ be an open cover of $\mathcal{X}_n$ such that $ \mathcal{U}^i_n = \widehat{U^i_n}$,
where $U^i_n$ is a deformation over $A_n$ of a local \'etale neighborhood $V^i$ of $Y$ in $X$. Let also $\mathcal{Y}_n\in D(B_n)$
and $\{\mathcal{W}_n^i\}$ the corresponding
open cover such that $\mathcal{W}_n^i=\widehat{W^i_n}$, where $W^i_n$ is a deformation over $B_n$ of a local \'etale neighborhood
$Z^i$ of $Y $ in $X$.
By Lemma~\ref{pure}, $U^i_n$, $V^i$, $W^i_n$ and $Z^i$ are also pure.

$B_n$ is the trivial square zero extension of $A_n$ by $A_n$. Therefore the trivial extension
\[
0\la A_n \la B_n \la A_n \la 0
\]
gives the extension (not necessarily trivial) of $A_n$-algebras
\[
0 \la \sheaf_{W^i_n}\otimes_{B_n} A_n \la \sheaf_{W^i_n} \la \sheaf_{W^i_n}\otimes_{B_n} A_n \la 0
\]
There is a right exact sequence
\[
\sheaf_{W^i_n}\otimes_{B_n} A_n \stackrel{\alpha_n}{\la} \Omega_{{W^i_n}/A_n} \otimes_{B_n} A_n \la \Omega_{{W^i_n\otimes_{B_n} A_n}/A_n} \la 0
\]
Since $X$ ir pure and reduced, it follows that $W^i_n\otimes_{B_n}A_n$ is pure and hence $\alpha_n$ is injective. Now taking completions
we get the exact sequence
\[
0 \la \sheaf_{\mathcal{U}_i} \la (\Omega_{{W^i_n}/A_n } \otimes_{B_n} A_n)^{\wedge} \la (\Omega_{{W^i_n\otimes_{B_n} A_n}/A_n})^{\wedge} \la 0
\]
Now if $(A,m)$ is a local $k$-algebra, then $\widehat{\Omega}_{A/k}\cong \widehat{\Omega}_{\hat{A}/k}$, where $\hat{A}$ is the $m$-adic completion
of $A$~\cite{TaLoRo07}. Therefore, and
patching the above sequences together, it follows that there is an exact sequence
\[
0 \la \sheaf_{\mathcal{X}_n} \la \widehat{\Omega}_{\mathcal{Y}_n/A_n} \otimes_{B_n}A_n \la \widehat{\Omega}_{\mathcal{X}_n} \la 0
\]
and hence we get a map
\[
\mathbb{T}^1_D(\mathcal{X}_n/A_n)\la \mathrm{Ext}^1_{\mathcal{X}_n}(\widehat{\Omega}_{\mathcal{X}_n/A_n},\sheaf_{\mathcal{X}_n})
\]
which, as in the usual scheme case, is injective. We will show that it is also surjective. Let
\[
0 \la \sheaf_{\mathcal{X}_n} \la \mathcal{E} \la \widehat{\Omega}_{X_n} \la 0
\]
be an element of $\mathrm{Ext}^1_{\mathcal{X}_n}(\widehat{\Omega}_{\mathcal{X}_n/A_n},\sheaf_{\mathcal{X}_n})$. Let
\[
\hat{d} \colon \sheaf_{\mathcal{X}_n} \la \widehat{\Omega}_{\mathcal{X}_n/A_n}
\]
be the completion of the universal derivation~\cite{TaLoRo07}. Then again as in the usual scheme case we get a square zero extension of $A_n$-algebras
\begin{equation}\label{extension}
0 \la \sheaf_{\mathcal{X}_n} \stackrel{\sigma}{\la} \sheaf_{\mathcal{Y}_n} \la \sheaf_{\mathcal{X}_n} \la 0
\end{equation}
Moreover, arguing in exactly the same way as in Proposition~\ref{local-to-global}, it follows that the above extension is locally the completion
of an extension of $U^i_n$ by $U^i_n$.
To complete the proof we need to show that $\sheaf_{\mathcal{Y}_n}$ admits the structure of a flat $B_n$-algebra and that
$\mathcal{Y}_n \otimes_{B_n}A_n =\mathcal{X}_n$.
$\sheaf_{\mathcal{Y}_n}$ is already an $A_n$-algebra and it can be made into an $A_1$-algebra via  $\lambda \colon k[t]/(t^2) \la \sheaf_{\mathcal{Y}_n}$
by setting $\lambda (t)=\sigma(1)$.
Therefore $\sheaf_{\mathcal{Y}_n}$ becomes a $B_n=A_1 \otimes A_n$-algebra. The flatness follows from the next straightforward generalization
of~\cite[Lemma A.9]{Ser06}.
\begin{lemma}
Let $(B,m_B)$ be a local ring, $A$ a $B$-algebra and $M$ a finitely generated $A$-module. Let
\begin{equation}\label{E}
0\la M \la A^{\prime} \la A \la 0
\end{equation}
be a square zero extension of $A$ by $M$.
Let $R$ be an $A^{\prime}$-algebra. Then $R$ is a flat $A^{\prime}$-algebra if and only if the sequence $(\ref{E})\otimes_{A^{\prime}}R$ is exact
and $R\otimes_{A^{\prime}}A$ is a flat $A$-algebra.
\end{lemma}
From the construction of the $B_n$-algebra structure on $\sheaf_{\mathcal{Y}_n}$ it follows that
$\sheaf_{\mathcal{Y}_n}\otimes_{B_n}A_n = \sheaf_{\mathcal{X}_n}$.
Moreover, since $X-Y$ is smooth, $(\ref{extension})\otimes_{B_n}A_n$ is exact on $X-Y$ and since $X$ is pure it follows that
$(\ref{extension})\otimes_{B_n}A_n$ is in fact exact and hence
$\sheaf_{\mathcal{Y}_n}$ is flat over $B_n$.
\end{proof}
\begin{remark}
If $X=Y$ then the above proposition says simply that
\[
\mathbb{T}^1(X_n/A_n)\cong \mathrm{Ext}^1_{X_n}({\Omega}_{X_n/A_n},\sheaf_{X_n})
\]
and
\[
T^1(X_n/A_n)\cong \mathcal{E}xt^1_{{X}_n}({\Omega}_{X_n/A_n},\sheaf_{X_n})
\]
where $X_n \in Def(X)(A_n)$.
\end{remark}
\begin{remark}
In the $X=Y$, Proposition~\ref{T1-formula} was proved by Namikawa~\cite{Nam06}.
\end{remark}
As a corollary of Proposition~\ref{T1-formula}, the spectral sequence relating the functors $\mathrm{Ext}$ and $\mathcal{E}xt$ gives the local to global
sequence for $T^1$.
\begin{corollary}
With assumptions as in Proposition~\ref{T1-formula}, there is an exact sequence
\[
0 \la H^1(\widehat{T}_{\mathcal{X}_n/A_n}) \la \mathbb{T}^1(\mathcal{X}_n/A_n) \la H^0(T^1(\mathcal{X}_n/A_n)) \la H^2(\widehat{T}_{\mathcal{X}_n/A_n})
\]
\end{corollary}
The next technical lemma will be needed.
\begin{lemma}\label{base-change}
Let $X$ be a pure and reduced scheme and let $X_A$ be a deformation of it over a local Artin $k$-algebra $A$. Let $F_A$ be a coherent sheaf on $X_A$
such that
there is a nonempty open subset $U_A \subset X_A$ such that the restriction $F_A|_{U_A}$ is flat over $A$. Let $A\la B$ be a homomorphism of finite
Artin local $k$-algebras
and $X_B=X_A\otimes_A B$. Let $i \colon X_B \la X_A$ the inclusion and $G_B$ a coherent $\sheaf_{X_B}$-module. Then for all $k \geq 0$,
\[
\mathrm{Ext}^k_{X_A}(F_A,i_{\ast}G_B)\cong \mathrm{Ext}^k_{X_B}(i^{\ast} F_A  ,G_B)
\]
and
\[
\mathcal{E}xt^k_{X_A}(F_A,i_{\ast}G_B)\cong \mathcal{E}xt^k_{X_B}(i^{\ast}F_A,G_B)
\]
\end{lemma}
\begin{proof}
For any $k$ there are natural maps
\begin{gather*}
\phi^k_F \colon \mathrm{Ext}^k_{X_B}(i^{\ast}F_A,G_B) \la \mathrm{Ext}^k_{X_A}(F_A,i_{\ast}G_B)\\
\psi^k_F \colon \mathcal{E}xt^k_{X_B}(i^{\ast}F_A,G_B) \la \mathcal{E}xt^k_{X_A}(F_A,i_{\ast}G_B)
\end{gather*}
which is defined as follows. Let $[E_B] $ be an element of $\mathrm{Ext}^k_{X_B}(i^{\ast}F_A ,G_B)$. This is represented by an extension
\[
0 \la G_B \la E_1 \la E_2 \la \cdots \la E_k \la i^{\ast}F_A  \la 0
\]
Moreover, there is a natural map $\lambda_A \colon F_A \la i_{\ast}i^{\ast}F_A $. Then we define
$\phi^k_F([E_A])\in \mathrm{Ext}_{X_A}^k(F_A, i_{\ast}G_B)$ to be the extension
obtained by pulling back $[E_A]$ with $\lambda$. Similarly for $\psi^k_F$.

Let $i_{\ast}\colon \mathbf{Coh(X_B)} \la \mathbf{Coh(X_A)}$ be the induced map between the corresponding categories of coherent sheaves.
Let $G$ be either the $\mathcal{H}om_{X_A}(F_A, \cdot)$ or $\mathrm{Hom}_{X_A}(F_A, \cdot )$ functor. Since $i_{\ast}$ is exact, to prove the
lemma it suffices to show that $i_{\ast}$ sends injectives to $G$-acyclics. First we show this in the case when $G=\mathcal{H}om_{X_A}(F_A, \cdot )$.
Let $I_B$ be an injective $\sheaf_{X_B}$-module. We will show that
\[
\mathcal{E}xt^k_{X_A}(F_A,I_B)=0
\]
this is local and so we may assume that $X$, and hence $X_A$, is affine. Then $X_A$ has enough locally free sheaves. So we may write
\[
0 \la P_A \la E_A \la F_A \la 0
\]
where $E_A$ is locally free. Hence
\[
\mathcal{E}xt^k_{X_A}(F_A,I_B)=\mathcal{E}xt^{k-1}_{X_A}(P_A,I_B)
\]
Moreover, since $X$ is pure it follows that $X_A$ is pure as well. Hence $E_A$ is pure and hence $P_A$ is also pure and its restriction on $U_A$ is
flat over $A$. Continuing similarly we find that
\[
\mathcal{E}xt^k_{X_A}(F_A,I_B)=\mathcal{E}xt^{1}_{X_A}(N_A,I_B)
\]
where $N_A$ is also pure and its restriction on $U_A$ is flat over $A$. Now consider the exact sequence
\[
0 \la Q_A \la M_A \la N_A \la 0
\]
where $M_A$ is locally free. Then, as before, $Q_A$ is pure and hence since $N_A$ is flat over $U_A$, it follows that
\[
 0 \la i^{\ast}Q_A \la i^{\ast}M_A \la i^{\ast}N_A \la 0
\]
is exact too. Therefore there is a commutative diagram
\[
\xymatrix{
\mathcal{H}om_{X_A}(N_A, i_{\ast}I_B) \ar[r]\ar[d]^{f_1}  & \mathcal{H}om_{X_A}(M_A, i_{\ast}I_B) \ar[r]\ar[d]^{f_2}  &
\mathcal{H}om_{X_A}(Q_A , i_{\ast}I_B) \ar[r]\ar[d]^{f_3}
 & \mathcal{E}xt^1_{X_A}(N_A, i_{\ast}I_B) \ar[r]\ar[d]^{f_4}  & 0\\
\mathcal{H}om_{X_B}(i^{\ast}N_A, I_B)\ar[r]   & \mathcal{H}om_{X_B}(i^{\ast}M_A, I_B)\ar[r] & \mathcal{H}om_{X_B}(i^{\ast}Q_A , I_B)\ar[r]
& \mathcal{E}xt^1_{X_A}(i^{\ast}N_A, I_B) \ar[r]  & 0
}
\]
where $f_1$, $f_2$ and $f_3$ are clearly isomorphisms. Therefore $f_4$ is an isomorphism too. But since $I_B$ is an injective $\sheaf_{X_B}$-module,
\[
\mathcal{E}xt^1_{X_A}(G_A,i_{\ast}I_B)=\mathcal{E}xt^1_{X_B}(i^{\ast}G_A,I_B)=0
\]
and hence
\[
\mathcal{E}xt^k_{X_A}(F_A,i_{\ast}I_B)=0
\]
for all $k \geq 1$ as claimed. Next we show the corresponding statement for the global $\mathrm{Ext}$.
The spectral sequence relating the local and global Ext functors show that
\[
\mathrm{Ext}_{X_A}^k(F_A, i_{\ast}I_B) = H^k(\mathcal{H}om_{X_A}(F_A,i_{\ast}I_B))=H^k(\mathcal{H}om_{X_B}(i^{\ast}F_A, I_B))=
\mathrm{Ext}^k_{X_B}(i^{\ast}F_A,I_B) =0
\]
The argument about the $\mathcal{E}xt$ sheaves cannot be directly applied to the global Ext functor because of the possible absence of enough
locally free sheaves on $X_A$.
\end{proof}
Next we give a version of the previous results in the case of formal schemes.
\begin{corollary}
With assumptions as in Lemma~\ref{base-change}, let $\mathcal{X}_A\in Def(Y,X)(A)$ and $\mathcal{X}_B=\mathcal{X_A}\otimes_A B$. Let $\mathcal{F}_A$
be a coherent sheaf on $\mathcal{X}_A$ such that there is an open $\mathcal{U}_A\subset \mathcal{X}_A$ such that
$\mathcal{F}_A|_{\mathcal{U}_A}$ is flat over $A$. Let $\mathcal{G}_B$ be a coherent sheaf on $\mathcal{X}_B$ and
$i \colon \mathcal{X}_B \la \mathcal{X}_A$ the inclusion. Then
\[
\mathcal{E}xt^i_{\mathcal{X}_A}(\mathcal{F}_A ,i_{\ast}\mathcal{G}_B)\cong \mathcal{E}xt^i_{\mathcal{X}_B}(i^{\ast}\mathcal{F}_A ,\mathcal{G}_B)
\]
and
\[
\mathrm{Ext}^i_{\mathcal{X}_A}(\mathcal{F}_A ,i_{\ast}\mathcal{G}_B)\cong \mathrm{Ext}^i_{\mathcal{X}_B}(i^{\ast}\mathcal{F}_A ,\mathcal{G}_B)
\]
\end{corollary}
\begin{proof}
The natural map $\phi^i_{\mathcal{F}}$ defined in Lemma~\ref{base-change} exists in this case too. Then the proof proceeds similarly and it is local.
Locally $\mathcal{X}_A \cong \widehat{V}_A$, where $V_A$ is a deformation over $A$ of a local \'etale neighborhood $V$ of $Y$ in $X$.
So we may assume that $\mathcal{F}_A=\widehat{F}_A$, $\mathcal{G}_B =\widehat{G}_B$where $F_A$, $G_B$ are coherent sheaves on $V_A$, $V_B$.
But then as we have already seen in Proposition~\ref{local-to-global},
\[
\mathcal{E}xt^i_{\widehat{V_A}}(\widehat{F}_A, \widehat{G}_B) =(\mathcal{E}xt^i_{V_A}(F_A,G_B))^{\wedge}
\]
Moreover, if $\mathcal{I}_B$ is an injective $\sheaf_{\mathcal{X}_B}$-module, then $\mathcal{I}_B=\widehat I_B$, where $I_B$ is an
injective $\sheaf_{V_B}$-module. Now the proof
proceeds exactly the same as the one in Lemma~\ref{base-change}.
\end{proof}

We now state the key result that will enable us to obtain obstructions to lift a deformation $X_n \in Def(Y,X)(A_n)$ to $A_{n+1}$.

\begin{proposition}\label{the-exact-sequence}
Let $X$ be a pure and reduced scheme defined over a field $k$ of characteristic zero
and $Y\subset X$ a closed subscheme of it such that $X-Y$ is smooth. Let $\mathcal{X}_n\in Def(Y,X)(A_n)$. Then there are exact sequences
\begin{gather*}
0\la \widehat{T}_{X}\la \widehat{T}_{\mathcal{X}_n/A_n} \la \widehat{T}_{\mathcal{X}_{n-1}/A_{n-1}} \la T^1(Y,X) \la T^1(\mathcal{X}_{n}/A_{n}) \la\\
 \la T^1(\mathcal{X}_{n-1}/A_{n-1}) \stackrel{\theta}{\la} \mathcal{E}xt_{\hat{X}}^2(\widehat{\Omega}_X,\sheaf_{\hat{X}})
\end{gather*}
and
\begin{gather*}
0\la H^0(\widehat{T}_{X})\la H^0(\widehat{T}_{\mathcal{X}_n/A_n}) \la H^0(\widehat{T}_{\mathcal{X}_{n-1}/A_{n-1}}) \la \mathbb{T}^1(Y,X) \la
\mathbb{T}^1(\mathcal{X}_{n}/A_{n}) \la\\
 \la \mathbb{T}^1(\mathcal{X}_{n-1}/A_{n-1}) \stackrel{\Theta}{\la} \mathrm{Ext}_{\hat{X}}^2(\widehat{\Omega}_X,\sheaf_{\hat{X}})
\end{gather*}
\end{proposition}
Note that since $X-Y$ is assumed to be smooth, then $T^1(Y,X)=T^1(X)$ by Proposition~\ref{local-to-global}.
\begin{proof}
Apply $\mathrm{Hom}_{\mathcal{X}_n}(\widehat{\Omega}_{\mathcal{X}_n/A_n}, \cdot)$ and
$\mathcal{H}om_{\mathcal{X}_n}(\widehat{\Omega}_{\mathcal{X}_n/A_n}, \cdot)$
on the square zero extension
\[
0 \la \sheaf_{\widehat{X}} \la \sheaf_{\mathcal{X}_n} \la \sheaf_{\mathcal{X}_{n-1}} \la 0
\]
and then use Proposition~\ref{T1-formula} and Lemma~\ref{base-change}.
\end{proof}

\section{Global lifting of deformations.}

Let $\mathcal{X}_n \in Def(Y,X)(A_n)$. In this section we obtain obstructions to lift $\mathcal{X}_n$ to $A_{n+1}$.
Let $\mathcal{Y}_{n-1}=Def(Y,X)(\varepsilon_{n-1})(\mathcal{X}_n) \in \mathbb{T}^1(\mathcal{X}_{n-1}/A_{n-1})$. According to Theorem~\ref{T1},
$\mathcal{X}_n$ lifts to $A_{n+1}$ if and only if $\mathcal{Y}_{n-1}$ is in the image of the natural map
\[
\mathbb{T}^1(\mathcal{X}_n/A_n) \stackrel{\tau_n}{\la} \mathbb{T}^1(\mathcal{X}_{n-1}/A_{n-1})
\]
Now according to Proposition~\ref{the-exact-sequence} there is an exact sequence
\[
\mathbb{T}^1(\mathcal{X}_n/A_n) \stackrel{\tau_n}{\la} \mathbb{T}^1(\mathcal{X}_{n-1}/A_{n-1}) \stackrel{\Theta}{\la} \mathrm{Ext}_{\hat{X}}^2(\widehat{\Omega}_X,\sheaf_{\hat{X}})
\]
Identifying $\mathbb{T}^1(\mathcal{X}_{n-1}/A_{n-1})$ with
$\mathrm{Ext}_{\mathcal{X}_{n}}^1(\widehat{\Omega}_{\mathcal{X}_{n}/A_{n}},\sheaf_{\mathcal{X}_{n-1}})=
\mathrm{Ext}_{\mathcal{X}_{n-1}}^1(\widehat{\Omega}_{\mathcal{X}_{n-1}/A_{n-1}},\sheaf_{\mathcal{X}_{n-1}})$ and
 $\mathbb{T}^1(\mathcal{X}_{n}/A_{n})$ with
$\mathrm{Ext}_{\mathcal{X}_{n}}^1(\widehat{\Omega}_{\mathcal{X}_{n}/A_{n}},\sheaf_{\mathcal{X}_{n}})$  we see that $\mathcal{Y}_{n-1}$ is represented
by the extension
\[
0 \la \sheaf_{\mathcal{X}_{n-1}} \la E \la \widehat{\Omega}_{\mathcal{X}_{n}/A_{n}} \la 0
\]
which is the pullback of the extension
\[
0 \la \sheaf_{\mathcal{X}_{n-1}} \la \widehat{\Omega}_{\mathcal{Y}_{n-1}/A_{n-1}}\otimes_{B_{n-1}}A_{n-1} \la \widehat{\Omega}_{\mathcal{X}_{n-1}/A_{n-1}} \la 0
\]
under the natural map $ \widehat{\Omega}_{\mathcal{X}_{n}/A_{n}} \la \widehat{\Omega}_{\mathcal{X}_{n-1}/A_{n-1}}$. Hence
\[
E=(\widehat{\Omega}_{\mathcal{Y}_{n-1}/A_{n-1}}
\otimes_{B_{n-1}}A_{n-1}) \times_{\widehat{\Omega}_{\mathcal{X}_{n-1}/A_{n-1}}}\widehat{\Omega}_{\mathcal{X}_{n}/A_{n}}.
\]
Then $\Theta( \mathcal{Y}_{n-1})\in
\mathrm{Ext}_{\mathcal{X}_n}^2(\widehat{\Omega}_{\mathcal{X}_n/A_n},\sheaf_{\widehat{X}})=
\mathrm{Ext}_{\hat{X}}^2(\widehat{\Omega}_X,\sheaf_{\hat{X}})$ is represented by the two-term extension
\[
0 \la \sheaf_{\hat{X}} \la \sheaf_{\mathcal{X}_n} \la E \la \widehat{\Omega}_{\mathcal{X}_n/A_n} \la 0,
\]
Hence, using Theorem~\ref{T1} we get that

\begin{theorem}\label{global-ob}
With assumptions as in Proposition~\ref{the-exact-sequence}, let $\mathcal{Y}_{n-1}=Def(Y,X)(\varepsilon_{n-1})(\mathcal{X}_n)$. Then
the obstruction to lift $\mathcal{X}_n$ to a deformation $\mathcal{X}_{n+1}$
over $A_{n+1}$ is the element
$ob(\mathcal{X}_n) \in \mathrm{Ext}_{\mathcal{X}_n}^2(\widehat{\Omega}_{\mathcal{X}_n/A_n},\sheaf_{\widehat{X}})=
\mathrm{Ext}_{\hat{X}}^2(\widehat{\Omega}_X,\sheaf_{\hat{X}})$ represented by the extension
\[
0 \la \sheaf_{\hat{X}} \la \sheaf_{\mathcal{X}_n} \la E \la \widehat{\Omega}_{\mathcal{X}_n/A_n} \la 0.
\]
where
\[
E=(\widehat{\Omega}_{\mathcal{Y}_{n-1}/A_{n-1}}
\otimes_{B_{n-1}}A_{n-1}) \times_{\widehat{\Omega}_{\mathcal{X}_{n-1}/A_{n-1}}}\widehat{\Omega}_{\mathcal{X}_{n}/A_{n}}.
\]
Therefore if $\mathrm{Ext}_{\hat{X}}^2(\widehat{\Omega}_X,\sheaf_{\hat{X}})=0$
and $Y$, $X$ satisfy the conditions of Proposition~\ref{finiteness-of-tangent-spaces}, then the hull of $Def(Y,X)$ is smooth.
\end{theorem}
In practice it is easier to verify vanishing for cohomology than for the Ext groups. Next we give some cohomological conditions for the vanishing
of $\mathrm{Ext}_{\hat{X}}^2(\widehat{\Omega}_X,\sheaf_{\hat{X}})$. First we make a definition.
\begin{definition}
Let $X$ be a pure scheme and $Y \subset X$ a closed subscheme such that $X-Y$ is smooth. Then we denote by $Ob^3(X)$ the cokernel of the local to
global obstruction map
$H^0(T^1(X)) \la H^2(\widehat{T}_X)$ of Proposition~\ref{local-to-global}.
\end{definition}
\begin{corollary}\label{obstructions-1}
There are three succesive obstructions in $H^0(\mathcal{E}xt^2_{\widehat{X}}(\widehat{\Omega}_{X}, \sheaf_{\widehat{X}}))$, $H^1(T^1(X))$ and
$Ob^3(X)$ for $X_n$ to lift to $A_{n+1}$. Therefore, if
\[
H^0(\mathcal{E}xt^2_{\widehat{X}}(\widehat{\Omega}_{X}, \sheaf_{\widehat{X}}))=H^1(T^1(X))=Ob^3(X)=0
\]
 and $Def_Y(X)$ has a hull, then its hull is smooth of dimension
\[
h^1(\widehat{T}_X)+h^0(T^1(X))-h^2(\widehat{T}_X)
\]
\end{corollary}
\begin{proof}
Consider the Leray spectral sequence
\[
E_2^{p,q}=H^p(\mathcal{E}xt^q_{\hat{X}}(\widehat{\Omega}_X,\sheaf_{\hat{X}})) \Rightarrow E^{p+q}=
\mathrm{Ext}^{p+q}_{\hat{X}}(\widehat{\Omega}_X,\sheaf_{\hat{X}}))
\]
Then there are exact sequences
\begin{gather*}
0 \la E_1^2 \la E^2 \la E_2^{0,2} \\
0 \la E_2^{1,0}\la E^1 \la E_2^{0,1} \la E_2^{2,0} \la E_1^2 \la E_2^{1,1} \la E_2^{3,0}
\end{gather*}
Now considering that $E^2=\mathrm{Ext}^{p+q}_{\hat{X}}(\widehat{\Omega}_X,\sheaf_{\hat{X}}))$,
$E_2^{0,2}=H^0(\mathcal{E}xt^2_{\widehat{X}}(\widehat{\Omega}_{X}, \sheaf_{\widehat{X}}))$ and
$E_2^{2,0} = H^2(\widehat{T}_X)$, we get the claim.
\end{proof}

\begin{corollary}
Suppose that $Def(Y,X)$ has a hull and that
\[
H^0(\mathcal{E}xt^2_{\widehat{X}}(\widehat{\Omega}_{X}, \sheaf_{\widehat{X}}))=H^1(T^1(X))=H^2(\widehat{T}_X)=H^0(T^1(X))=0
\]
Then every deformation of $X$ is formally locally trivial.
\end{corollary}
This happens because from the previous corollary the hull of $Def(Y,X)$ is smooth and is the same as the hull of the locally trivial
deformations $Def^{\prime}(Y,X)$.

\begin{remark}
The simplest case when $Ob^3(X)=0$ is when $H^2(\widehat{T}_X)=0$. This happens in particular when there is a morphism $f \colon X \la S$, where
$S$ is affine, $f$ is proper with fibers of dimension $\leq \; 1$ and $Y=f^{-1}(s)$, for some $s \in S$. Then from the formal functions theorem it follows that $H^2(\widehat{T}_X)=0$. 
This is the  case of 3-fold flips and divisorial contractions with at most one-dimensional fibers.
\end{remark}

\section{Local to global.}\label{local-to-global-section}

Let $X$ be a scheme and $Y \subset X$ a closed subscheme such that $X-Y$ is smooth. In the previous section we obtained obstructions in order for a deformation $X_n \in Def(Y,X)(A_n)$ to lift to a deformation  $X_{n+1} \in Def(Y,X)(A_{n+1})$ in the case that $X$ was pure and reduced. However our methods were global and do not give us any information about the local structure of $X_{n+1}$. In this section we will study the problem of when local liftings of $X_n$ globalize to give a deformation $X_{n+1}$ of $X$ over $A_{n+1}$, or more generally, when local deformations of $X$ exist globally..

Ideally one should study the local to global map $\pi \colon Def(Y,X) \la Def_{loc}(Y,X)$. If $X=Y$, $X$ has isolated singularities and $H^2(T_X)=0$ then $\pi$ is known to be smooth. This is not necessarily true anymore if $X$ has positive dimensional singular locus. The reason is the same as the one for the failure of $Def_{loc}(Y,X)$ to have a hull. It is the presence of local automorphisms that do not lift to higher order. However under strong restrictions on the singularities of $X$, $\pi$ is smooth.

\begin{proposition}\label{smoothness-of-pi}
With assumptions as in Theorem~\ref{local-hull}, suppose also that $H^2(\widehat{T}_X)=0$. Then $\pi$ is smooth.
\end{proposition}

\begin{proof}
We only do the case when $X=Y$. The general case is proved similarly. For convenience, set $D=Def_{loc}(Y,X)$ and $D_{loc}=Def_{loc}(Y,X)$. Then it suffices to show that for any small extension
\[
0 \la J \la B \stackrel{g}{\la} A \la 0
\]
the natural map
\[
D(B) \la D(A)\times_{D_{loc}(A)} D_{loc}(B)
\]
is surjective.

Let $X_A \in D(A)$, $s_A =\pi (X_A) \in D_{loc}(A)$ and $s_B \in D_{loc}(B)$ such that $D_{loc}(g)(s_B)=s_A$. By the definition of $D_{loc}$, $s_B$ and $s_A$ are equivalent to an open cover $\{ U_i\}$ of $X$, a collection of deformations $U_i^B$ and $U_i^A$ of $U_i$ over $B$ and $A$ respectively, such that $U_i^B \otimes_B A \cong U^A_i$, $B$-isomorphisms $\phi^B_{ij} \colon U_i^B|_{U_i \cap U_j} \la U_j^B|_{U_j \cap U_i}$ and $A$-isomorphisms $\phi^A_{ij} \colon U_i^A|_{U_i \cap U_j} \la U_j^A|_{U_j \cap U_i}$ such that for any $i,j,k$, $\phi_{ij}^A\phi_{jk}^A\phi_{ki}^A$ is the identity automorphism of $U^A_{ijk}=U_i^A \cap U_j^A \cap U_k^A$.

By assumption, we may take $U_i$ in such a way that $U_i \cap U_j$ satisfies $(H_4)$. Hence we may take the $\phi^B_{ij}$ such that on $U_{ijk}=U_i \cap U_j \cap U_k$, the restriction of $\phi^B_{ijk}=\phi^B_{ij} \phi^B_{jk}\phi^B_{ki}$ on $U^A_{ijk}$ is the identity automorphism of $U^A_{ijk}$. Hence $\phi^B_{ijk}$ corresponds to a $B$-derivation $d_{ijk} \in \mathrm{Hom}_{U_i^B}(\Omega_{U_i^B/B}, \sheaf_{U_i})=\mathrm{Hom}_{U_i}(\Omega_{U_i},\sheaf_{U_i})$. On the fourfold intersections $U_{ijks}=U_i\cap U_j \cap U_k \cap U_s$ they satisfy a cocycle condition and hence we get an element of $H^2(\mathcal{H}om_X(\Omega_X, \sheaf_X))=H^2(T_X)$. If this element vanishes then the $\phi^B_{ij}$ can be modified in such a way that $\phi^B_{ij}\phi^B_{jk}\phi^B_{ki}$ is the identity automorphism of $U_i^B \cap U_j^B \cap U_k^B$ and hence the $U_i^B$ glue to a global deformation $X_B$.
\end{proof}

In order to get around the failure of the local to global map $\pi \colon Def(Y,X) \la Def_{loc}(Y,X)$ to be smooth, we must gain some control of the automorphisms of deformations. Having this in mind, and following the ideas of Lichtenbaum-Schlessinger~\cite{Li-Sch67}, we make the following definitions.

\begin{definition}\label{def-of-local-functors}
Let
\begin{equation}\label{small-ext}
0 \la J \la B \la A \la 0
\end{equation}
be a small extension of Artin rings and $X_A \in Def(Y,X)(A)$. Let $(X^i_B,\phi_i)$, $i=1,2$ be pairs where $X^i_B\in Def(Y,X)(B)$ and $\phi_i \colon X_A \la X^i_B \otimes_B A$ isomorphisms. We say that the pair $(X^1_B, \phi_1) $ is isomorphic to the pair $(X^2_B, \phi_2)$ if and only if there is a $B$-isomorphism $\psi \colon X^1_B \la X^2_B$ such that $\psi \phi_1 = \phi_2$.
\begin{enumerate}
\item We define by $Def(X_A/A,B)$ to be the set of isomorphism classes of pairs $[X_B, \phi]$ of deformations $X_B\in Def(Y,X)(B)$ and marking isomorphisms $\phi \colon X_A \la X_B \otimes_B A$.
\item Let $\underline{Def}(X_A/A,B)$ be the sheaf of sets associated to the presheaf $F$ on $X$ such that $F(U)=Def(U_A/A,B)$, where $U_A=X_A|_U$. Then we define
\[
Def_{loc}(X_A/A,B)=H^0(\underline{Def}(X_A/A,B))
\]
\end{enumerate}
Note that there is a natural map
\[
\pi \colon Def_(X_A/A,B) \la Def_{loc}(X_A/A,B).
\]
\end{definition}
Note also that since any square zero extension of local Artin $k$-algebras can be obtained by a sequence of successive small extensions, we do not lose anything by working only with small extensions.
\begin{remark}
Let $X_n \in Def(Y,X)(A_n)$. Then in the notation of section~\ref{T1-section}, $\mathbb{T}^1(X_n/A_n)=Def(X_n/A_n,B_{n+1})$ and $T^1(X_n/A_n)=Def_{loc}(X_n/A_n,B_{n+1})$.
\end{remark}

\begin{theorem}\label{local-to-global-1}
Let $X$ be a scheme defined over a field $k$ and $Y \subset X$ a closed subscheme such that $X-Y$ is smooth. Let
\[
0 \la J \la B \la A \la 0
\]
be a small extension of local Artin $k$-algebras and $X_A \in Def(Y,X)(A)$. Then
\begin{enumerate}
\item $Def(X_A/A,B)$ and $Def_{loc}(X_A/A,B)$ are $\mathbb{T}^1(Y,X) \otimes J$ and $H^0(T^1(X)\otimes J)$ homogeneous spaces, respectively.
\item Let $s_B \in Def_{loc}(X_A/A,B)$. Then the set $\pi^{-1}(s_B)$ is a homogeneous space over $H^1(\widehat{T}_X \otimes J)$.
\item There is a sequence
\[
0 \la H^1(\widehat{T}_X \otimes J) \stackrel{\alpha}{\la} Def(X_A/A,B) \stackrel{\pi}{\la}   Def_{loc}(X_A/A,B)    \stackrel{\partial}{\la} H^2(\widehat{T}_X \otimes J)
\]
which is exact in the following sense. Let $s_B \in Def_{loc}(X_A/A,B)$. Then $s_B$ is in the image of $\pi$ if and only if $\partial (s_B)=0$. Moreover, let $X_B, X^{\prime}_B \in Def(X_A/A,B)$ such that $\pi(X_A)=\pi(X^{\prime}_A)$. Then there is $\gamma \in H^1(\widehat{T}_X \otimes J)$ such that $X^{\prime}_A = \gamma \cdot  X_A $, where by $"\cdot"$ we denote the action of
$H^1(\widehat{T}_X \otimes J)$ on $\pi^{-1}(s_B)$.
\end{enumerate}
\end{theorem}

\begin{proof}
We will only prove the theorem in the case when $X=Y$. The local algebraizability conditions embedded in the definition of $Def(Y,X)$ ensure that, with some effort, all steps of the proof can be carried out in the case when $Y \not= X$ and $X-Y$ is smooth. The proof of the theorem consists of two steps.

\textbf{Step 1.} In this step we will obtain a description of $Def(X_A/A,B)$ and $Def_{loc}(X_A/A,B)$ using cotangent sheaf cohomology and spaces of infinitesimal extensions that we describe next. Let $X$ be an $S$-scheme and $\mathcal{F}$ an $\sheaf_X$-module. We denote by $\mathrm{Ex}(X/S, \mathcal{F})$ the space of square zero extensions
\[
0 \la \mathcal{F} \la \sheaf_{X^{\prime}} \la \sheaf_X \la 0
\]
of $S$-schemes~\cite{Gr64}. Note that there is always a natural map
\[
\mathrm{Ex}(X/S, \mathcal{F}) \la H^0(T^1(X/S,\mathcal{F}))
\]
where $T^1(X/S,\mathcal{F})$ is the first cotangent cohomology sheaf of
Schlessinger~\cite{Li-Sch67}. This map is an isomorphism if $X$ and $S$ are affine.

The sequence $B\la A \la \sheaf_{X_A}$ gives the exact sequences~\cite{Li-Sch67},~\cite{Gr64}
\begin{equation}\label{T1-sequence}
0 \la T^1(X_A/A, J \otimes_A \sheaf_{X_A}) \la T^1(X_A/B, J \otimes_A \sheaf_{X_A}) \stackrel{\nu}{\la} T^1(A/B, J \otimes_A \sheaf_{X_A}) \la T^2(X_A/A, J \otimes_A \sheaf_{X_A})
\end{equation}
and
\begin{equation}\label{Ex-eq}
0 \la   \mathrm{Ex}(X_A/A,J \otimes_A \sheaf_{X_A})    \la \mathrm{Ex}(X_A/B,J \otimes_A \sheaf_{X_A}) \stackrel{\mu}{\la} \mathrm{Ex}(A/B,J \otimes_A \sheaf_{X_A})
\end{equation}
Taking global sections on the first one we get
\begin{equation}\label{H0}
0 \la H^0(T^1(X_A/A, J \otimes_A \sheaf_{X_A})) \la H^0(T^1(X_A/B, J \otimes_A \sheaf_{X_A})) \stackrel{\lambda}{\la} H^0(T^1(A/B, J \otimes_A \sheaf_{X_A}))
\end{equation}
By a slight abuse of notation we denote by $[\mathcal{I}]$ both the elements of $\mathrm{Ex}(A/B,J \otimes_A \sheaf_{X_A})$ and $H^0(T^1(A/B, J \otimes_A \sheaf_{X_A}))$ corresponding to the square zero extension
\[
0 \la J \otimes_A \sheaf_{X_A} \la B \la A \la 0
\]
\textit{Claim:}

\begin{enumerate}
\item \[
Def_{loc}(X_A/A,B)=\lambda^{-1}([\mathcal{I}])
\]
\item \[
Def(X_A/A,B)= \mu^{-1}([I])
\]
\end{enumerate}

Indeed, an element of $Def_{loc}(X_A/A,B)$ is equivalent to an open cover $\{ U_i\}$ of $X$ and pairs $[U^i_B, \phi^i_A] \in Def(U^i_A/A,B)$, where $U^i_A=X_A|_{U_i}$, such that for any $i,j$, $[U^i_B|_{U_i \cap U_j}, \phi^i_A|_{U_i \cap U_j}]=[U^J_B|_{U_i \cap U_j}, \phi^J_A|_{U_i \cap U_j}]$. These
give square zero extensions $[e_i]\in T^1(U^i_A/B,J \otimes_A \sheaf_{U^i_A})$
\[
0 \la J \otimes_A \sheaf_{U^i_A} \la \sheaf_{U^i_B} \la \sheaf_{U^i_A} \la 0
\]
which are isomorphic on the overlaps $U_i \cap U_j$ and hence glue to an element $[e] \in H^0(T^1(X_A/B, J \otimes_A \sheaf_{X_A}))$. Moreover, the fact that $U^i_B$ is flat over $B$ and $U^i_B \otimes_B  A=U^i_A$ imply that $\lambda ([e])=[\mathcal{I}]$~\cite{Li-Sch67}. Therefore $Def_{loc}(X_A/A,B)=\lambda^{-1}([\mathcal{I}])$. A similar argument shows also that $Def(X_A/A,B)= \mu^{-1}([I])$.

\textbf{Step 2.} This is the main part of the proof of the theorem. Combining the results of the claim and the exact sequences~(\ref{Ex-eq}) and~(\ref{H0}), it follows that $Def_{loc}(X_A/A,B)$ and $Def(X_A/A,B)$ are
$H^0(T^1(X_A/A, J \otimes_A \sheaf_{X_A}))=H^0(T^1(X)\otimes J)$~\cite{Li-Sch67} and $\mathrm{Ex}(X_A/A,J \otimes_A \sheaf_{X_A})=\mathbb{T}^1(X)\otimes J$~\cite{Gr64} homogeneous spaces.
This shows~\ref{local-to-global-1}.1.

We proceed to show~\ref{local-to-global-1}.2. In what follows we use the following notation. Let $\{U_i\}_{i\in I}$ be an open cover of $X$. Then for any choice of indices $i_1, \ldots , i_k$, we set $U_{i_1i_2\cdots i_k}=U_{i_1} \cap \cdots \cap U_{i_k}$. Also if $X_R$ is a deformation of $X$ over an Artin ring $R$, we set $X^{i_1\cdots i_k}_R=X_R|_{U_{i_1} \cap \cdots \cap U_{i_k}}$.

Let $s_B \in Def_{loc}(X_A/A,B)$. First we exhibit the action of $H^1(T_X \otimes J)=H^1(\mathcal{H}om_{X_A}(\Omega_{X_A/A}, J \otimes_A \sheaf_{X_A}))$ on $\pi^{-1}(s_B)$. Let $[X_B, \phi] \in \pi^{-1}(s_B)$ and $\gamma \in H^1(\mathcal{H}om_{X_A}(\Omega_{X_A/A}, J \otimes_A \sheaf_{X_A}))$. The element $s_B$ is equivalent to give an open cover $\{U_i\}_{i\in I}$ of $X$, elements $[U^i_B, \phi^i] \in Def(U^i_A/A,B)$ and for all $i, j$ isomorphisms $\phi^{ij} \colon U^i_B|_{U_i\cap U_j} \la U^j_B|_{U_i\cap U_j}$ such that on $U_i \cap U_j$, $\phi^{ij} \phi^i = \phi^j$. The element $[X_B, \phi]\in Def(X_A/A,B)$ is equivalent to give elements $[U^i_B, \psi^i] \in Def(U^i_A/A,B)$, for all $i$, and for all $i, j$ isomorphisms $\psi^{ij} \colon U^i_B|_{U_i\cap U_j} \la U^j_B|_{U_i\cap U_j}$ such that on $U_i \cap U_j$, $\psi^{ij} \psi^i = \psi^j$ and on the triple intersections $U_i \cap U_j \cap U_k$, $\psi^{jk}\psi^{ij}=\psi^{ik}$. The cohomology class $\gamma$ is equivalent to a collection $\gamma_{ij} \in \mathrm{Hom}_{X^{ij}_A}(\Omega_{X^{ij}_A/A}, J \otimes \sheaf_{X^{ij}_A})=\mathrm{Hom}_{U^{ij}_B}(\Omega_{U^{ij}_B/B}, J \otimes \sheaf_{X^{ij}_A})$, where $U^{ij}_B=U^i_B|_{U_i \cap U_j}$, that satisfy the cocycle condition on the triple intersections. Therefore, $\gamma$ is equivalent to a collection of $B$-derivations $d_{ij} \colon \sheaf_{U^{ij}_B} \la J \otimes \sheaf_{X^{ij}_A}$, satisfying the cocycle condition on the triple intersections. Then we define $ \gamma \cdot [X_B, \phi]$ to be the element of $\pi^{-1}(s_B)$ that is defined by the data $[U^i_B,\psi^i]$ and glueing isomorphisms $\psi^{ij}+d_{ij} \colon U^i_B|_{U_i\cap U_j} \la U^j_B|_{U_i\cap U_j}$.

It remains to show that $\pi^{-1}(s_B)$ is a $H^1(T_X \otimes J)$-homogeneous space, i.e., that $H^1(T_X \otimes J)$ acts transitively on $\pi^{-1}(s_B)$. Let $[X_B, \psi], [X^{\prime}_B, \psi^{\prime}] \in \pi^{-1}(s_B)$. Then there is an open cover $\{U_i\}_{i\in I}$ of $X$ and  isomorphisms $\lambda_i \colon X_B|_{U_i} \la X^{\prime}_B|_{U_i}$, for all $i \in I$, such that on $U_i$, $\lambda_i \psi = \psi^{\prime}$. Then on $U_{ij}$, $\lambda_{ij}=\lambda_j^{-1}\lambda_i $ is an automorphism of $X^{ij}_B$ over $X^{ij}_A$. Therefore, $\lambda_{ij}$ corresponds to a $B$-derivation $d_{ij} \in \mathrm{Der}_B(\sheaf_{X^{ij}_B} , J \otimes \sheaf_{X^{ij}_A})= \mathrm{Hom}_{X^{ij}_B}(\Omega_{X^{ij}_B/B}, J \otimes \sheaf_{X^{ij}_A})= \mathrm{Hom}_{X^{ij}_A}(\Omega_{X^{ij}_A/A}, J \otimes \sheaf_{X^{ij}_A})$. These satisfy the cocycle condition on triple intersections and hence give an element $\gamma \in H^1(\mathcal{H}om_{X_A}(\Omega_{X_A/A}, J \otimes_A \sheaf_{X_A}))=H^1(T_X \otimes J)$. Now from the definition of the action of $H^1(T_X \otimes J)$ on $\pi^{-1}(s_B)$, it is clear that $\gamma \cdot [X_B, \psi]=[X^{\prime}_B,\psi^{\prime}]$, and therefore the action is transitive.

Next we show~\ref{local-to-global-1}.3. Taking into consideration the previous two parts, it suffices to construct the map $\partial$ and to show that $\mathrm{Ker}(\partial) \subset \mathrm{Im}(\pi)$. Let $s_B \in Def_{loc}(X_A/A,B)$ as above. Then for any $i,j,k \in I$, $\phi_{ijk}=\phi_{ki}\phi_{jk}\phi_{ij}$ is a $B$-automorphism of $U^i_B|_{U_{ijk}} $ over $X^{ijk}_A$. Therefore $\phi_{ijk}$ corresponds to a $B$-derivation $d_{ijk} \in \mathrm{Der}_B (\sheaf_{U^i_B|_{U_{ijk}}}, J \otimes \sheaf_{X^{ijk}_A})= \mathrm{Hom}_{X^{ijk}_A}(\Omega_{X^{ijk}_A/A}, J \otimes \sheaf_{X^{ijk}_A})$. These satisfy the cocycle condition on the fourfold intersections and therefore give an element of $H^2(\mathcal{H}om_{X_A}(\Omega_{X_A/A}, J \otimes_A \sheaf_{X_A}))=H^2(T_X \otimes J)$. This defines the map $\partial$. If $\partial(s_B)=0$, then the isomorphisms $\phi_{ij}$ can be modified so that $\phi_{ijk}$ is the identity automorphism of $U^i_B|_{U_{ijk}} $ and therefore the $U^i_B$ and $\phi^i$ glue to a global deformation $X_B$ and and isomorphism $\phi \colon X_A \la X_B \otimes_B A$. Hence $s_B = \pi ([X_B, \phi])$, as claimed.
\end{proof}

\begin{corollary}\label{obstructions-3}
With assumptions as in Theorem~\ref{local-to-global-1}, there are two successive obstructions in $H^0(T^2(X)\otimes J)$ and $H^1(T^1(X) \otimes J)$ in order that $\mathrm{Def}_{loc}(X_A/A,B)\not= \emptyset$, i.e., for $X_A$ to lift locally to $B$. If these obstructions vanish then there is another obstruction in $H^2(\widehat{T}_X \otimes J)$ in order that $\mathrm{Def}(X_A/A,B)\not= \emptyset$, i.e., for the local deformations to globalize.
\end{corollary}
\begin{proof}
We only do the case $X=Y$. The general case is similar. Let $Q = \mathrm{Im} (\nu)$, where $\nu$ is the map in the long exact sequence~(\ref{T1-sequence}). Then there are two exact sequences,
\begin{gather}
0 \la H^0(T^1(X_A/A,J \otimes \sheaf_{X_A})) \la H^0(T^1(X_A/B,J \otimes \sheaf_{X_A})) \stackrel{\alpha}{\la} H^0(Q) \stackrel{\beta}{\la} \\
H^1(T^1(X_A/A,J \otimes \sheaf_{X_A})) \notag\\
0 \la H^0(Q) \la H^0(T^1(A/B,J \otimes \sheaf_{X_A})) \la H^0(T^2(X_A/A,J \otimes \sheaf_{X_A}))
\end{gather}
By step 1. of the proof of Theorem~\ref{local-to-global-1}, $\mathrm{Def}_{loc}(X_A/A,B)=\lambda^{-1}([I])$, where $\lambda=\beta \alpha$. It is now clear from the above exact sequences that there are two successive obstructions in $H^0(T^2(X_A/A,J \otimes \sheaf_{X_A}))=H^0(T^2(X)\otimes J)$ and $H^1(T^1(X_A/A,J \otimes \sheaf_{X_A}))=H^1(T^1(X)\otimes J)$ so that $\lambda^{-1}([I])\not= \emptyset$. If these obstructions vanish then from Theorem~\ref{local-to-global-1}.3, it follows that there is another obstruction in $H^2(T_X\otimes J)$ so that $\mathrm{Def}(X_A/A,B) \not= \emptyset$.
\end{proof}
The spaces $\mathrm{Def}(X_A/A,B)$ and $\mathrm{Def}_{loc}(X_A/A,B)$ do not have in general any vector space structure over the ground field $k$. This complicates any calculation involving them. However, if $B$ is the trivial extension of $A$ by $J$ then these spaces do have natural $k$-vector space structures.

\begin{remark}
A variant of Theorem~\ref{local-to-global-1} is already known in the case $X=Y$ and the obstructions in Corollary~\ref{obstructions-3} are also well known~\cite{Har04},~\cite{Li-Sch67}. However, to our knowledge, the $Def_{loc}$ space and the global to local sequence~\ref{local-to-global-1}.3 have not been considered earlier and this separates our statement from the ones that can already be found in the literature.
\end{remark}
\begin{remark}
Theorem~\ref{local-to-global-1} obtains a relation between the local and global deformation spaces $\mathrm{Def}(X_A/A,B)$ and $\mathrm{Def}_{loc}(X_A/A,B)$. However, the obstructions obtained in Corollary~\ref{obstructions-3} are not satisfactory in many ways. We explain why. Recall quickly how the obstructions work. In the notation of the previous corollary, given a deformation $X_A$ of $X$ over $A$, then if the obstruction in $H^0(T^2(X))$ vanishes, we can lift $X_A$ locally to $B$, i.e., there is an open cover $\{U^i\}$ of $X$ and liftings $U^i_{B}$ of $X_A|_{U^i}$ over $B$. Then if the second obstruction in $H^1(T^1(X))$ vanishes, the local liftings can be modified in order to agree on overlaps. By doing this we do find obstructions in order that $\mathrm{Def}_{loc}(X_A/A,B)\not= \emptyset$ but we lose all local information about the liftings. In order to have some control over the singularities of a lifting of $X_A$ we would like to choose a particular lifting $U^i_B$ of $X_A|_{U^i}$ and then find obstructions to globalize it. This requires more careful study and additional obstructions will appear. For general choice of the rings $A$ and $B$ this is probably quite tricky but for the purposes of this paper (where mainly one parameter deformations are studied) we will only consider deformations over the rings $A_n$. Our main tool is again the $T^1$-lifting property.
\end{remark}

\subsection{Local to global and the $T^1$- lifting property.}\label{local-subsection}

Let $X_n$ be a deformation of $X$ over $A_n$. Next we present a method to lift $X_n$ to a deformation $X_{n+1}$ of $X$ over $A_{n+1}$ that allows us to control the singularities of $X_{n+1}$.

Let $X_{n-1}=X_n \otimes_{A_n}A_{n-1}$ and $Y_{n-1}=X_n \otimes_{A_n}B_{n-1} \in \mathbb{T}^1(X_{n-1}/A_{n-1})$, where $B_{n-1}$ is an $A_n$-algebra via the map $\varepsilon_{n-1} \colon A_n \la B_{n-1}$ defined in section~\ref{T1-section}. Then according to the $T^1$-lifting property (Theorem~\ref{T1}), $X_n$ lifts to $A_{n+1}$ if and only if $Y_{n-1}$ is in the image of the natural map $\tau_n \colon \mathbb{T}_D^1(X_n/A_n) \la \mathbb{T}_D^1(X_{n-1}/A_{n-1})$. Theorem~\ref{global-ob} obtained an explicit obstruction element for this to happen. However as mentioned earlier, it does not offer any local information about the possible liftings. Local information is carried by the sheaves $T^1(X_n/A_n)$. These are related to $\mathbb{T}^1(X_n/A_n)$ by the following natural commutative diagram.
\begin{equation}\label{local-to-global-diagram-1}
\xymatrix{
\mathbb{T}_D^1(X_n/A_n) \ar[d]_{\tau_n}\ar[r]^{\phi_n} & H^0(T_D^1(X_n/A_n)) \ar[d]^{\sigma_{n}} \\
\mathbb{T}_D^1(X_{n-1}/A_{n-1}) \ar[r]^{\phi_{n-1}} & H^0(T_D^1(X_{n-1}/A_{n-1}))
}
\end{equation}
The idea is the following. Let $s_{n-1}=\phi_{n-1}(Y_{n-1})$. Instead of lifting $Y_{n-1}$ directly through $\tau_n$, we will obtain obstructions in order for $s_{n-1}$ to be in the image of $\sigma_n$. If these vanish, then we choose a particular element $s_n \in H^0(T^1(X_n/A_n))$ such that $\sigma_n(s_n)=s_{n-1}$ and we will obtain obstructions for the existence of a global $Y_n \in \mathbb{T}^1(X_n/A_n)$ such that $\phi_n(Y_n)=s_n$. This way we can control the local structure of $Y_n$. Then, according to Proposition~\ref{T1-1}, there is a lifting $X_{n+1}$ of $X_n$ over $A_{n+1}$ such that $X_{n+1} \otimes_{A_{n+1}}B_n=Y_n$, where again $B_n$ is an $A_{n+1}$-algebra via $\varepsilon_n \colon A_{n+1} \la B_n$. Now suppose that by this process we have obtained a formal deformation $f_n \colon X_n \la \mathrm{Spec}(A_n)$, for $n$. Suppose that it is induced by an algebraic deformation $f \colon \mathcal{X} \la \mathrm{Spec} A$. We will see next that the sections $s_n$ carry a lot of information about the singularities of $\mathcal{X}$. In particular, smoothings can be detected by them, as is exhibited by the next two propositions.

\begin{proposition}\label{sn}
Let $f \colon \mathcal{X} \la \Delta$ is a deformation of a pure and reduced scheme $X$ over the spectrum of a discrete valuation ring $(A,m_A)$. Let $f_n \colon X_n \la \mathrm{Spec} A_n$ be the associated formal deformation and  $Y_n = X_{n+1}\otimes_{A_{n+1}}B_n \in \mathbb{T}^1(X_n/A_n)$. Moreover, let $e \in \mathbb{T}^1(\mathcal{X}/\Delta)$ be the element that is represented by the extension
\begin{equation}\label{pullback-sequence}
0 \la \sheaf_{\mathcal{X}}=f^{\ast}\omega_{\Delta} \la \Omega_{\mathcal{X}} \la \Omega_{\mathcal{X}/\Delta} \la 0
\end{equation}
Then $e_n =Y_n$ in $\mathbb{T}^1(X_n/A_n)$, where $e_n =e \otimes_A A_n$.
\end{proposition}
\begin{proof}
By Proposition~\ref{T1-formula}, $\mathbb{T}^1(X_n/A_n)=\mathrm{Ext}^1_{X_n}(\Omega_{X_n/A_n},\sheaf_{X_n})$ and $\mathbb{T}^1(\mathcal{X}/\Delta)=\mathrm{Ext}^1_{\mathcal{X}}(\Omega_{\mathcal{X}/\Delta},\sheaf_{\mathcal{X}})$. It follows from their definition that $Y_n$  and $e_n$ are represented by the extensions
\[
0 \la \sheaf_{X_n} \stackrel{\alpha}{\la} (\Omega_{X_{n+1}\otimes_{{A_{n+1}}}B_n/A_n})\otimes_{B_n}A_n \la \Omega_{X_n/A_n} \la 0
\]
and
\[
0 \la \sheaf_{X_n} \stackrel{\beta}{\la} \Omega_{\mathcal{X}}\otimes_A A_n \la \Omega_{X_n/A_n} \la 0,
\]
respectively, where $\alpha(1)=d(1\otimes x) \otimes 1$ and $\beta(1)=dt \otimes 1$, where $t$ is a generator of the maximal ideal of $m_R$. It is now easy to see that the two extensions are isomorphic via
\[
\Phi \colon \Omega_{\mathcal{X}}\otimes_A A_n  \la (\Omega_{X_{n+1}\otimes_{{A_{n+1}}}B_n/A_n})\otimes_{B_n}A_n
\]
defined by $\Phi(dz \otimes a )= d(\overline{z}\otimes 1) \otimes a$, where $z \in \sheaf_{\mathcal{X}}$, $a\in A$ and $\overline{z}$ is the class of $z$ in $\sheaf_{X_n}$.
\end{proof}

\begin{proposition}
With assumptions as in Proposition~\ref{sn}, assume moreover that $X$ has complete intersection singularities. Then $f$ is a smoothing of $X$ if and only if there are $k, n \in \mathbb{Z}_{>0}$, $k < n$,  such that
\[
t^kT^1(X_n/A_n)\subset \sheaf_{X_n} \cdot s_n
\]
\end{proposition}
\begin{proof}
Dualizing the exact sequence~(\ref{pullback-sequence}) we get the exact sequence
\[
\sheaf_{\mathcal{X}} \stackrel{\alpha}{\la} T^1(\mathcal{X}/\Delta) \la T^1(\mathcal{X}) \la 0
\]
where $\alpha(1)=e$. Therefore $T^1(\mathcal{X})=\mathrm{CoKer}(\alpha) = T^1(\mathcal{X}/\Delta)/\sheaf_{\mathcal{X}}\cdot e$.

Suppose that $f$ is a smoothing. Then $T^1(\mathcal{X})$ is supported over $m_A$ and hence there is $k \in \mathbb{Z}_{>0}$ such that $t^k (T^1(\mathcal{X}/\Delta)/\sheaf_{\mathcal{X}}\cdot e)=0$. Reducing it modulo $m_A^n$ and using Proposition~\ref{sn}, we get the claim.

Conversely, suppose there are a $k,n \in \mathbb{Z}_{>0}$ such that $t^kT^1(X_n/A_n)\subset \sheaf_{X_n} \cdot s_n$. Let $\mathcal{F}= T^1(\mathcal{X}/\Delta)/\sheaf_{\mathcal{X}}\cdot e$ and $\mathcal{F}_n=T^1(X_n/A_n)/ \sheaf_{X_n} \cdot s_n$. Then by by Lemma~\ref{Fn} and Proposition~\ref{sn}, it follows that $\mathcal{F}/t^{n+1}\mathcal{F}=\mathcal{F}_n$, where $t$ is a generator of the maximal ideal $m_A$ of $A$. Then by assumption, $t^k(\mathcal{F}/t^{n+1}\mathcal{F})=0$ and hence $t^k\mathcal{F}=t^{n+1}\mathcal{F}=\mathcal{F}_n$ and therefore by Nakayama's lemma, $t^k\mathcal{F}=0$. Hence $T^1(\mathcal{X})$ is supported over $m_A$, and hence by Lemma~\ref{ci}, $f$ is a smoothing.
\end{proof}

Even though our previous discussion was for the case when $X=Y$, it is also valid in the general case.

The remaining part of this section is devoted to the study of the maps $\sigma_n$ and $\phi_n$ in diagram~(\ref{local-to-global-diagram-1}). In particular we obtain conditions in order for them to be surjective.

\begin{proposition}\label{local-sequence}
With assumptions as in Proposition~\ref{the-exact-sequence}, there are canonical exact sequences
\begin{gather*}
0 \la H^0(T^1(X)/\mathcal{F}_n) \la H^0(T^1(\mathcal{X}_n/A_n) )\stackrel{\sigma_n}{\la} H^0 (T^1(\mathcal{X}_{n-1}/A_{n-1}) ) \la Q_n \la 0\\
0 \la L_n \la Q_n \la H^0(\mathcal{E}xt^2_{\hat{X}}(\widehat{\Omega}_X, \sheaf_{\hat{X}}))\\
0 \la L_n \la H^1(T^1(X)/\mathcal{F}_n) \la H^1(T^1(\mathcal{X}_{n}/A_{n}))
\end{gather*}
and a noncanonical one
\begin{gather*}
0 \la H^0(T^1(X)/\mathcal{F}_n) \la H^0(T^1(\mathcal{X}_n/A_n) )\stackrel{\sigma_n}{\la} H^0 (T^1(\mathcal{X}_{n-1}/A_{n-1}) ) \la\\
\la H^1(T^1(X)/\mathcal{F}_n) \oplus H^0(\mathcal{E}xt_{\hat{X}}(\widehat{\Omega}_X, \sheaf_{\hat{X}}))
\end{gather*}
where $\mathcal{F}_n \subset T^1(X)$ is the cokernel of the map
\[
\widehat{T}_{\mathcal{X}_n/A_n}\la \widehat{T}_{\mathcal{X}_{n-1}/A_{n-1}}
\]
\end{proposition}
\begin{proof}
From Proposition~\ref{the-exact-sequence} there is an exact sequence
\[
0 \la T^1(X)/\mathcal{F}_n \la T^1(\mathcal{X}_{n}/A_{n}) \stackrel{h_n}{\la} T^1(\mathcal{X}_{n-1}/A_{n-1}) \stackrel{\mu_n}{\la} T^2(Y,X)
\]
Let $M_n=\text{Ker}(\mu_n)$. Then the above sequence breaks into two short exact sequences
\begin{gather*}
0 \la T^1(X)/\mathcal{F}_n \la T^1(\mathcal{X}_{n}/A_{n})\stackrel{h_n}{\la} M_n \la 0\\
0 \la M_n \la T^1(\mathcal{X}_{n-1}/A_{n-1}) \stackrel{\mu_n}{\la} T^2(Y,X)
\end{gather*}
Then we get the following exact sequences in cohomology
\begin{gather*}
0 \la H^0(T^1(X)/\mathcal{F}_n) \stackrel{f_1}{\la} H^0(T^1(\mathcal{X}_{n}/A_{n})) \stackrel{f_2}{\la} H^0(M_n) \stackrel{f_3}{\la}
H^1(T^1(X)/\mathcal{F}_n)
\stackrel{f_4}{\la} H^1(T^1(\mathcal{X}_{n}/A_{n})) \\
0 \la H^0(M_n) \stackrel{g_1}{\la} H^0(T^1(\mathcal{X}_{n-1}/A_{n-1})) \stackrel{g_2}{\la} H^0(T^2(Y,X))
\end{gather*}
We wish to understand the kernel and Cokernel of the map $\sigma_n=g_1\circ f_2$. Consider the following commutative diagram
\[
\xymatrix{
0 \ar[r]      &       H^0(T^1(X)/\mathcal{F}_n) \ar[r]^{f_1}\ar[d] &  H^0(T^1(\mathcal{X}_{n}/A_{n})) \ar[r]^{f_2}\ar[d]^{\phi_n} &
\text{Im}(f_2) \ar[r]\ar[d]^{\beta} &        0 \\
0  \ar[r] & 0 \ar[r] & H^0(T^1(\mathcal{X}_{n-1}/A_{n-1})) \ar@{=}[r] & H^0(T^1(\mathcal{X}_{n-1}/A_{n-1})) \ar[r] & 0
}
\]
where $\beta$ is the restriction of $g_1$ on $\text{Im}(f_2)$. The snake lemma now gives that $\text{ker}(\mu_n)=H^0(T^1(X)/\mathcal{F}_n)$ and
$\text{Coker}(\beta) = \text{CoKer}(h)$. Therefore there is an exact sequence
\begin{equation}\label{eq1}
0 \la H^0(T^1(X)/\mathcal{F}_n) \la H^0(T^1(\mathcal{X}_n/A_n) )\la H^0 (T^1(\mathcal{X}_{n-1}/A_{n-1}) ) \la Q_n \la 0
\end{equation}
where $Q_n = \text{Coker}(\beta)$. Now from the diagram
\[
\xymatrix{
0 \ar[r]      &     \text{Im}(f_2)\ar@{=}[r] \ar[d]&  \text{Im}(f_2)\ar[r]\ar[d]^{\beta}  & 0 \ar[r]\ar[d] &        0 \\
0  \ar[r] &    H^0(M_n) \ar[r]                     & H^0(T^1(\mathcal{X}_{n-1}/A_{n-1})) \ar[r] & H^0(T^2(Y,X)) \ar[r] & 0
}
\]
gives that there is an exact sequence
\begin{equation}\label{eq2}
0 \la L_n \la Q_n \la H^0(T^2(Y,X))
\end{equation}
where $L_n = \text{CoKer}[\text{Im}(f_2) \la H^0(M_n)$ and therefore there is another exact sequence
\begin{equation}\label{eq3}
0 \la L_n \la H^1(T^1(X)/\mathcal{F}_n) \la H^1(T^1(\mathcal{X}_{n}/A_{n}))
\end{equation}
Now the proposition follows from~(\ref{eq1}),~(\ref{eq2}), and~(\ref{eq3}).
\end{proof}

\begin{corollary}\label{local-obstructions-1}
There are two succesive obstructions in  $H^0( \mathcal{E}xt^2_{\hat{X}}(\widehat{\Omega}_X, \sheaf_{\hat{X}}))$ and $H^1(T^1(X)/\mathcal{F}_n)$ in order for an element $s_{n-1}$ of $H^0(T^1(X_{n-1}/A_{n-1}))$ to be in the image of $\sigma_n$.
\end{corollary}
The exact sequences in the previous proposition are not very enlightning in general. However, if $X$ has local complete intersection singularities, then they are greatly simplified.

\begin{corollary}\label{Cor-local-ob-1}
Suppose that $X$ has local complete intersection singularities, or more generally that $H^0(\mathcal{E}xt^2_X(\Omega_X,\sheaf_X))=0$. Then there is an exact sequence
\[
0 \la H^0(T^1(X)/\mathcal{F}_n) \la H^0(T^1(\mathcal{X}_n/A_n) )\stackrel{\sigma_n}{\la} H^0 (T^1(\mathcal{X}_{n-1}/A_{n-1}) ) \stackrel{\partial}{\la}  H^1(T^1(X)/\mathcal{F}_n)
\]
\end{corollary}

Next we study the local to global map $\phi_n$. If $X$ is pure and reduced, then the diagram~(\ref{local-to-global-diagram-1}) is part of the commutative diagram with exact rows
\begin{equation}\label{local-to-global-diagram-2}
\xymatrix{
H^1(\widehat{T}_{X_n/A_n}) \ar[r]^{\psi_n}\ar[d]_{\mu_n} & \mathbb{T}^1(X_n/A_n) \ar[d]_{\tau_n}\ar[r]^{\phi_n} & H^0(T^1(X_n/A_n)) \ar[d]^{\sigma_{n}} \ar[r]^{\partial_n} & H^2(\widehat{T}_{X_n/A_n}) \ar[d]_{\lambda_n}\\
H^1(\widehat{T}_{X_{n-1}/A_{n-1}}) \ar[r]^{\psi_{n-1}}   &  \mathbb{T}^1(X_{n-1}/A_{n-1}) \ar[r]^{\phi_{n-1}} & H^0(T^1(X_{n-1}/A_{n-1})) \ar[r]^{\partial_{n-1}} & H^2(\widehat{T}_{X_{n-1}/A_{n-1}})
}
\end{equation}
where $\psi_n$ and $\psi_{n-1}$ are injective. Hence the obstruction for an element $s_n \in  H^0(T^1(X_n/A_n))$ to be in the image of $\phi_n$ is the element $\partial_n(s_n) \in H^2(\widehat{T}_{X_n/A_n})$. If $X$ has isolated singularities then it is well known that there are successive obstructions in $H^2(\widehat{T}_X)$ in order for $\partial_n(s_n)$ to be zero. However, in the general case this is not so and the reason is once more the inability to lift local automorphisms. The best that we can do in this case is to find conditions for the map $\phi_n$ to be surjective.
\begin{proposition}\label{local-obstructions-2}
Let $X$ be a pure and reduced scheme over a field $k$ and $Y \subset X$ a closed subscheme of it such that $X-Y$ is smooth. Let $X_n \in Def(Y,X)(A_n)$. Let $\mathcal{F}_k= \mathrm{CoKer}[\widehat{T}_{X_k/A_k} \la \widehat{T}_{X_{k-1}/A_{k-1}}] \subset T^1(X)$. Then if $H^2(\widehat{T}_X)=H^1(\mathcal{F}_k)=0$, for all $k \leq n$, then $\phi_n$ is surjective.
\end{proposition}
Note that if the singularities of $X$ are isolated, then $H^1(\mathcal{F}_k)=0$ and the proposition is the familiar result about isolated singularities. Admittedly it is not very easy to check the conditions of the proposition but at least the sheaves $\mathcal{F}_k$ are all subsheaves of $T^1(X)$ which depends only on $X$.
\begin{proof}
The long exact sequence described in Proposition~\ref{the-exact-sequence}, gives the following short exact sequences
\begin{gather*}
0 \la \widehat{T}_X \la \widehat{T}_{X_n/A_n} \la Q_n \la 0 \\
0 \la Q_n \la \widehat{T}_{X_{n-1}/A_{n-1}} \la \mathcal{F}_n \la 0 \\
\end{gather*}
These give the exact sequences
\begin{gather*}
\cdots \la H^2(\widehat{T}_X) \la H^2(\widehat{T}_{X_n/A_n}) \la H^2(Q_n) \la \cdots \\
\cdots \la H^1(\mathcal{F}_n) \la H^2(Q_n) \la H^2(\widehat{T}_{X_{n-1}/A_{n-1}}) \la \cdots
\end{gather*}
Now the claim follows by induction on $n$.
\end{proof}
So far we found conditions in order $\phi_n$ and $\sigma_n$ to be surjective. Going back to our original problem, starting with a deformation $X_n$ of $X$ over $A_n$, we want to lift $Y_{n-1}=X_n\otimes_{A_n}B_{n-1}$ to a $Y_n$ in $\mathbb{T}^1(X_n/A_n)$. Let $s_{n-1}=\phi_{n-1}(Y_{n-1})$. If the obstructions in Corollary~\ref{Cor-local-ob-1} and Proposition~\ref{local-obstructions-2} vanish, then there is a $Y_n^{\prime} \in \mathbb{T}^1(X_n/A_n)$ such that $\phi_{n-1}(\tau_n(Y^{\prime}_n)-Y_{n-1})=0$. Hence in order to obtain a lifting $Y_n$ of $Y_{n-1}$, we want to lift the locally trivial deformation $Z_{n-1}=\tau_n(Y^{\prime}_n)-Y_{n-1}$. If $X$ has isolated singularities then it is well known that the obstruction to lift $Z_{n-1}$ to a locally trivial deformation $Z_n$ over $A_n$ is in $H^2(T_X)$ (this also follows immediately from the next proposition). In general though, this is not true. Again the best that we can do is to find conditions for $\tau_n$ to be surjective.
\begin{proposition}\label{loc-trivial-ob}
With asumptions as in Proposition~\ref{local-obstructions-2}, if $H^1(\mathcal{F}_n)=H^2(\widehat{T}_X)=0$, then every locally trivial lifting $Z_{n-1}$ of $X_{n-1}$ over $B_{n-1}$ lifts to a locally trivial lifting $Z_n$ of $X_n$ over $B_n$.
\end{proposition}
\begin{proof}
From diagram~\ref{local-to-global-diagram-1} it follows that the isomorphism classes of locally trivial liftings of $X_{k}$ over $B_{k}$ are in one to one correspondence with $H^1(\widehat{T}_{X_{k}/A_{k}})$. Then the statement of the proposition is equivalent to say that if $H^1(\mathcal{F}_n)=H^2(\widehat{T}_X)=0$, then the natural map
\[
\mu_n \colon H^1(\widehat{T}_{X_n/A_n}) \la H^1(\widehat{T}_{X_{n-1}/A_{n-1}})
\]
is surjective. This follows with similar arguments as in the proof of Proposition~\ref{local-obstructions-2}.
\end{proof}

The previous discussion suggests that we must study the sheaves $\mathcal{F}_n$ and the quotients $T^1(X)/\mathcal{F}_n$. There are two main cases.

The first is when $T^1(X)/\mathcal{F}_n$ has finite support for all $n$ and hence no higher cohomology and $\sigma_n$ is surjective for all $n$. In this case the only obstruction for $X_n$ to lift to $A_{n+1}$ is in $H^2(\widehat{T}_X)$. This case is treated in Lemma~\ref{Fn}.

The second case is when for some reason we know that $H^2(\mathcal{F}_n)=0$ for all $n$. The simplest cases when this happens is when the singular
locus of $X$ is one dimensional or if there is a proper morphism $f \colon X \la Z$ with one-dimensional fibers and $Z$
affine (these are for example the cases of flipping, flopping and divisorial contractions with one dimensional fibers). In this case we will show that $H^1(T^1(X)/\mathcal{F}_n)$ is a quotient of $H^1(T^1(X))$ and hence we can at least find a uniform bound for its dimension, which is finite if $X$ has proper singular locus. Indeed, there is an exact sequence
\[
0 \la \mathcal{F}_n \la T^1(X) \la T^1(X)/\mathcal{F}_n \la 0
\]
which induces the exact sequence
\[
H^1(\mathcal{F}_n) \la H^1(T^1(X)) \la H^1(T^1(X)/\mathcal{F}_n) \la H^2(\mathcal{F}_n)
\]
Since $H^2(\mathcal{F}_n)=0$, it follows that $H^1(\mathcal{F}_n)$ is a quotient of $H^1(T^1(X))$.

So we have shown that
\begin{corollary}
Suppose that the singular
locus of $X$ is one dimensional or there is a proper morphism $f \colon X \la Z$ with one-dimensional fibers and $Z$
affine (these are for example the cases of flipping, flopping and divisorial contractions with one dimensional fibers). Then if $H^1(T^1(X))=0$, the map $\sigma_n \colon H^0(T^1(X_n/A_n)) \la H^0(T^1(X_{{n-1}/A_{n-1}})$ is surjetive, for all $n$.
\end{corollary}

\section{$\mathbb{Q}$-Gorenstein deformations}\label{Q-section}

Let $X$ be a $\mathbb{Q}$-Gorenstein scheme and $Y \subset X$ a closed subscheme such that $X-Y$ is smooth. In this section we extend the results obtained in the previous sections regarding the usual deformation functor $Def(Y,X)$ to the case of the $\mathbb{Q}$-Gorenstein deformation functor $Def^{qG}(Y,X)$. To do so we will locally compare the $\mathbb{Q}$-Gorenstein deformations of $X$, with the deformations of its index 1 cover $\tilde{X}$. The key property that enables us to do so is that locally every $\mathbb{Q}$-Gorenstein deformation of $X$ lifts to a deformation of $\tilde{X}$~\cite{KoBa88}.

Let
\[
0 \la J \la B \la A \la 0
\]
be a small extension of Artin rings and $X_A \in Def^{qG}(Y,X)(A)$. Then in complete analogy with the case of $Def(Y,X)$ (Definition~\ref{def-of-local-functors}) we define $Def^{qG}(X_A/A,B)$, $\underline{Def}^{qG}(X_A/A,B)$ and $Def^{qG}_{loc}(X_A/A,B)=H^0(\underline{Def}(X_A/A,B))$

We need the following technical result.

\begin{lemma}\label{action}
Let $B$ be an $A$-algebra, $M$ a $B$-module and $G$ a group acting on them compatibly with the algebra structure,
i.e, for any $g \in G$, the map $\phi_g \colon B \la B$
defined by $\phi_g(b)=g \cdot b$ is an $A$-algebra isomorphism and $g \cdot (bm)=(g\cdot b)(g \cdot m)$,
for any $b \in B, m \in M$. Then there is an action of $G$
on $T^i(B/A,M)$, $i=0,1,2$. If $A=k$ a field, then $G$ also acts on $\cup_{C \in Art(k)} Def(B)(C) $, where $Def(B)(C)$ is the set
of all deformations of $B$ over $C$.
\end{lemma}

\begin{proof}

For any $g \in G$, there is an induced isomorphism $\phi_g \colon B \la B$ of $B$ given by $\phi_g(b)=g^{-1} \cdot b$, for any $b \in B$.
This gives an isomorphism
\[
\phi^{\ast}_g \colon T^i(B/A,M) \la T^i(B/A,M^{\ast})
\]
where $M^{\ast}$ is $M$ as an abelian group but the $B$ module structure is given by $b \cdot m=(g^{-1}\cdot b)m$. The map
$\psi_g \colon M^{\ast} \la M$ given by $\psi_g(m)=g \cdot m$ is a $B$-module homomorphism inducing an isomorphism
\[
\psi^{\ast}_g \colon T^i(B/A,M^{\ast}) \la T^i(B/A,M)
\]
Now the map $f_g = \psi_g^{\ast} \circ \phi_g^{\ast} \colon T^i(B/A,M) \la T^i(B/A,M)$ gives the $G$-action on $T^i(B/A,M)$.

$T^i(B/A,M)$, $i=1,2$, can be described as the spaces of infinitesimal one and two term extensions of $B$ by $M$, respectively.
It is usefull to describe the action of $G$ on $T^i(B/A,M)$ when the latter is viewed this way.

Let $(E)$ be a one term infinitesimal extension
\[
0 \la M \la C \la B \la 0
\]
of $B$ by $M$. Then for any $g \in G$ let $(E^{\prime})$ be the pull back extension
\[
\xymatrix
{
0 \ar[r] & M \ar@{=}[d]\ar[r] & C^{\prime} \ar[d]\ar[r] & B \ar[d]^{\phi_g} \ar[r] & 0 \\
0 \ar[r] & M \ar[r] & C \ar[r] & B \ar[r] & 0\\
}
\]
and $(E^{\prime\prime})$ be the pushout extension
\[
\xymatrix
{
0 \ar[r] & M \ar[d]_{\psi_g}\ar[r] & C^{\prime} \ar[d]\ar[r] & B \ar@{=}[d] \ar[r] & 0 \\
0 \ar[r] & M \ar[r] & C^{\prime\prime} \ar[r] & B \ar[r] & 0\\
}
\]
Then $g \cdot [E] = [E^{\prime\prime}]$ in $T^1(B/A,M)$. The action on two term extensions is defined exactly analogously. Next we will
show that $G$ acts on $\cup_{C \in Art(k)} Def(B)(C) $. So let $C \in Art(k)$ be a finite local Artin $k$-algebra and $R_C$ a deformation of
$B$ over $C$. We proceed by induction on the length of $C$. If $\text{length}(C)=1$, then $R_C \in T^1(B/k,B)$ and the action is already
defined. Now any $C$ appears as an extension
\[
0 \la k \la C \la C^{\prime} \la 0
\]
Let $R_{C^{\prime}} = R_C \otimes_C C^{\prime}$. Then by induction, $g \cdot R_{C^{\prime}}$ is defined and there is an isomorphism
$g \cdot R_{C^{\prime}} \la R_{C^{\prime}}$ (not over $C^{\prime}$ in general). Define $g \cdot R_C$
to be the extension obtained by pulling back
\[
0 \la B \la R_C \la R_{C^{\prime}} \la 0
\]
via the map $ g \cdot R_{C^{\prime}} \la R_{C^{\prime}}$.

\end{proof}

\begin{center}
\textbf{Construction of the sheaves} $T^i_{qG}(X_A/B,\mathcal{F})$.
\end{center}

Let $X_A \la \text{Spec}A \la \text{Spec}B$ be morphisms of schemes such that $X_A$ is a Cohen-Macauley, relatively Gorenstein in codimension 1 and such that there is
$n \in \mathbb{Z}$ with $\omega^{[n]}_{X_A/A}$ invertible. Let $\mathcal{F}$ be a coherent sheaf on $X_A$. Next we will define coherent sheaves
$T^i_{qG}(X_A/B, \mathcal{F})$.

Let $X_A = \cup_i U_i $ be an affine cover of $X_A$ and let $\pi_i \colon \tilde{U}_i \la U_i $ be the index 1 cover of $U_i$.
Let $r_i$ be the index of $U_i$ and $\mathcal{F}_i=\mathcal{F}|_{U_i}$.
Then $\pi_i$ is Galois with Galois group the group of $r_i$ roots of unity $\mu_{r_i}$. Then by Lemma~\ref{action},
$\mu_i$ acts on $T^k(\tilde{U}_i, \pi_i^{\ast} \mathcal{F}_i)$, $ k \geq 0$.
Let $T^k_{qG}(U_i/B, \mathcal{F}_i)=(T^k(\tilde{U}_i/B, \pi_i^{\ast}\mathcal{F}_i))^{\mu_i}$.
This is a coherent sheaf on $U_i$. We will show that
these sheaves glue to a coherent sheaf $T^k_{qG}(X_A/B,\mathcal{F})$. It suffices to show that there are isomorphisms
$\phi_{ij} \colon T^k_{qG}(U_i/B,\mathcal{F}_i)|_{U_{ij}} \la T^k_{qG}(U_j/B,\mathcal{F}_i)|_{U_{ij}}$, where $U_{ij}=U_i \cap U_j$.
Let $r_{ij}$ be the index of $U_{ij}$. Then
$r_{ij} | r_j$ and $r_{ij}|r_i$. Let $\pi_{ij} \colon \tilde{U}_{ij} \la U_{ij}$ be the index 1 cover of $U_{ij}$.
Then from the uniqueness and the construction of the index 1 cover it follows that there are factorizations
\[
\xymatrix{
\pi^{-1}_i(U_{ij}) \ar[ddr]_{\pi_i}\ar[dr]^{\phi_{ij}} &           &          \pi^{-1}_j(U_{ij}) \ar[ddl]^{\pi_j}\ar[dl]_{\phi_{ji}}  \\
                                       & \tilde{U}_{ij} \ar[d]^{\pi_{ij}}                     &           \\
                                       & U_{ij}                &             \\
}
\]
where $\phi_{ij}$, $\phi_{ji}$ are \'etale of degrees $s_{ij} =r_i /r_{ij}$ and $s_{ji}=r_j/r_{ij}$. Then
\[
T^k(\pi_i^{-1}(U_{ij})/Y,\pi_i^{\ast}\mathcal{F}_i) = \phi_{ij}^{\ast} T^k(\tilde{U}_{ij}/Y,\pi_{ij}^{\ast}\mathcal{F}_i)
\]
and therefore
\[
(T^k(\pi_i^{-1}(U_{ij})/Y,\pi_i^{\ast}\mathcal{F}_i))^{\mu_{s_{ij}}}=T^k(\tilde{U}_{ij}/Y,\pi_{ij}^{\ast}\mathcal{F}_i)
\]
Hence
\[
(T^k(\pi_i^{-1}(U_{ij})/Y,\pi_i^{\ast}\mathcal{F}_i))^{\mu_{r_i}} = ((T^k(\pi_i^{-1}(U_{ij}),\pi_{i}^{\ast}\mathcal{F}_i))^{\mu_{s_{ij}}})^{\mu_{r_{ij}}}=
(T^k(\tilde{U}_{ij})/Y, \pi_{ij}^{\ast}\mathcal{F}_{ij})^{\mu_{r_{ij}}}
\]
Exactly similarly it follows that
\[
(T^k(\pi_j^{-1}(U_{ij})/Y,\pi_j^{\ast}\mathcal{F}_j))^{\mu_{r_j}} = (T^k(\tilde{U}_{ij})/Y,\pi_{ij}^{\ast}\mathcal{F}_{ij})^{\mu_{r_{ij}}}
\]
and hence
\[
T^k_{qG}(U_i/B, \mathcal{F}_i)|_{U_{ij}} = T^k_{qG}(U_j/B,\mathcal{F}_j)|_{U_{ij}}
\]
and hence the sheaves $T^k_{qG}(U_i/B,\mathcal{F}_i)$ glue to a global sheaf $T^k_{qG}(X_A/B,\mathcal{F})$.

The next proposition shows that $T^0_{qG}$ and $T^0$ agree under certain conditions.
\begin{proposition}\label{T0}
Suppose that $\mathcal{F}$ is a locally free coherent sheaf on $X_A$. Then
\[
T^0_{qG}(X_A/B,\mathcal{F}) \cong T^0(X_A/B, \mathcal{F})
\]
\end{proposition}
\begin{proof}
Let $\{U_i \}$ be an affine cover of $X_A$ and $\mathcal{F}_i=\mathcal{F}|_{U_i}$.
Let $\pi_i \colon \tilde{U}_i \la U_i$ be the index 1 cover and $G_i$ the corresponding Galois group.
Then $T^0_{qG}(U_i/B, \mathcal{F}_i)= T^0(\tilde{U}_i/B, \mathcal{F}_i)^{G_i}$, and moreover,
$T^0(\tilde{U}_i/B, \mathcal{F}_i)=\mathrm{Hom}_{\tilde{U}_i}(\Omega_{\tilde{U}_i/B}, \pi_i^{\ast}\mathcal{F}_i)^{\sim}$.
The $G_i$-action is given as follows. Let $g \in G_i$ and $f \in
\mathrm{Hom}_{\tilde{U}_i}(\Omega_{\tilde{U}_i/B}, \pi_i^{\ast}\mathcal{F}_i)$. Then $g \cdot f$ is the $\sheaf_{\tilde{U}_i}$-sheaf homomorphism defined by
$(g \cdot f)(x) = g^{-1} \cdot f(g \cdot x)$. The natural map $\pi_i^{\ast}\Omega_{U_i/B} \la \Omega_{\tilde{U}_i/B}$ induces a homomorphism
\begin{equation}\label{T0-eq-1}
\phi \colon \mathrm{Hom}_{\tilde{U}_i}(\Omega_{\tilde{U}_i/B}, \pi_i^{\ast}\mathcal{F}_i) \la
\mathrm{Hom}_{\tilde{U}_i}(\pi^{\ast}\Omega_{U_i/B}, \pi_i^{\ast}\mathcal{F}_i) = \mathrm{Hom}_{U_i}(\Omega_{U_i/B}, \pi_i^{\ast}\mathcal{F}_i)
\end{equation}
Now since $\mathcal{F}$ is assumed to be locally free, it follows that both modules in the above sequence are reflexive. Moreover,
since $X_A$ is Gorenstein in codimension 1, $\pi_i$ is \'etale in codimension 1 and therefore $\phi$ is an isomorphism.
Hence taking $G_i$-invariants we get an isomorphism
\[
T^0_{qG}(U_i/B, \mathcal{F}_i) \la (\mathrm{Hom}_{U_i}(\Omega_{U_i/B}, \pi_i^{\ast}\mathcal{F}_i)^{G_i})^{\sim}
\]

I now claim that $\mathrm{Hom}_{U_i}(\Omega_{U_i/B}, \pi_i^{\ast}\mathcal{F}_i)^{G_i}= \mathrm{Hom}_{U_i}(\Omega_{U_i/B}, \mathcal{F}_i)$.
The natural injection $ \mathcal{F}_i \la \pi_i^{\ast}\mathcal{F}_i$ gives a natural injection
\[
\psi \colon \mathrm{Hom}_{U_i}(\Omega_{U_i/B}, \mathcal{F}_i) \la \mathrm{Hom}_{U_i}(\Omega_{U_i/B}, \pi_i^{\ast}\mathcal{F}_i)^{G_i}
\]
Now let $f \in \mathrm{Hom}_{U_i}(\Omega_{U_i/B}, \pi_i^{\ast}\mathcal{F}_i)^{G_i}$. The definition of the $G_i$-action shows that
$\mathrm{Im}(f) \subset (\pi_i^{\ast}\mathcal{F}_i)^{G_i}=\mathcal{F}_i$. Hence $ \psi$ is surjective and hence an isomorphism. Hence for any $U_i$
we get an isomorphism
\[
g_i \colon T^0_{qG}(U_i/B, \mathcal{F}_i) \la T^0(U_i/B/ \mathcal{F}_i)
\]
Following the construction of the sheaves $T^i_{qG}$ we see that these isomorphisms glue to a global isomorphism
\[
g \colon T^0_{qG}(X_A/B,\mathcal{F}) \la T^0(X_A/B,\mathcal{F})
\]
as claimed.
\end{proof}

\begin{proposition}
Let $X$ be a $\mathbb{Q}$-Gorenstein scheme defined over a field $k$.
Let $X_A \la \text{Spec} A$ be a $\mathbb{Q}$-Gorenstein deformation of $X$ over a finite local Artin $k$-algebra $A$. Let
\[
0 \la J \la B \la A \la 0
\]
be an extension of finite local Artin $k$-algebras with $J^2=0$.
Then there is a $k$-isomorphism
\[
 T^1_{qG}(X_A/B, J \otimes \sheaf_{X_A})  \la \text{Def}^{qG}(X_A/A,B)
\]
where $\text{Def}^{qG}(X_A/A,B)$ is the space of isomorphism classes of $\mathbb{Q}$-Gorenstein
liftings $X_B$ of $X_A$ over $B$.
\end{proposition}

\begin{proof}
Let $r$ be the index of $X$, $\pi_A \colon \tilde{X}_A \la X_A$ the index 1 cover of $X_A$ and $\pi \colon \tilde{X} \la X$
the index 1 cover of $X$. Then $\tilde{X}_A$ is a deformation of $\tilde{X}$ over $A$.
An element of $T^1_{qG}(X_A/B, J \otimes \sheaf_{X_A})$ corresponds to a $\mu_r$-invariant square zero extension
\begin{equation}\label{eq-T1-1}
0 \la J \otimes \sheaf_{\tilde{X}_A} \la \sheaf_{\tilde{X}_B} \la \sheaf_{\tilde{X}_A} \la 0
\end{equation}
Taking invariants we get an extension
\begin{equation}\label{eq-T1-2}
0 \la J \otimes \sheaf_{X_A} \la \sheaf_{X_B} \la \sheaf_{X_A} \la 0
\end{equation}
and hence a $\mathbb{Q}$-Gorenstein lifting $X_B$ of $X_A$ over $B$. This defines a map
\[
\phi \colon  T^1_{qG}(X_A/B, J \otimes \sheaf_{X_A})  \la \text{Def}^{qG}(X_A/A,B)
\]
Next we show that $\phi$ is surjective. Indeed, let $X_B$ be a $\mathbb{Q}$-Gorenstein lifting of $X_A$ over $B$. Then there is a square zero extension
as in~(\ref{eq-T1-2}). Let $\pi_B \colon \tilde{X}_B \la X_B$ be the index 1 cover of $X_B$. As before, this is a lifiting of $\tilde{X}_A$ over $B$.
Hence there is a $\mu_r$-invariant extension as in~(\ref{eq-T1-1}). Therefore $\phi$ is surjective.

It remains to show that $\phi$ is injective. Since $X$ is Gorenstein in codimension 1 it follows that $\pi \colon \tilde{X} \la X$
is \'etale in codimension 1. Let $U \subset X$ be the Gorenstein locus. Then $\pi^{-1}(U) \la U$ is \'etale and
$\text{codim}(\tilde{X}-\pi^{-1}(U),\tilde{X}) \geq 2$. Therefore the natural map
\[
Def(\tilde{X}) \la Def(\pi^{-1}(U))
\]
is injective~\cite[Lemma 9.1]{Art76} and hence $\phi$ is injective too.
\end{proof}
It now immediately follows that
\begin{corollary}
Let $X$ be a $\mathbb{Q}$-Gorenstein scheme defined over a field $k$ and $Y \subset X$ a closed subscheme such that
$X-Y$ is smooth. Let $X_n \in Def_Y^{qG}(X)(A_n)$. Then
\begin{enumerate}
\item
\[
T^1_{qG}(Y,X)=T^1_{qG}(X/k,\sheaf_X)
\]
\item
\[
T^1_{qG}(X_n/A_n)=T^1_{qG}(X_n/A_n, \sheaf_{X_n})
\]
\end{enumerate}
\end{corollary}

Most of the functorial properties of the usual $T^i$ sheaves hold for the $T^i_{qG}$ as well. Next we present a few that are of interest to us.

\begin{proposition}\label{properties}
Let $X$ be a $\mathbb{Q}$-Gorenstein scheme defined over a field $k$.
Let $ A \in Art(k)$ and $X_A \la \text{Spec}(A) $ be a $\mathbb{Q}$-Gorenstein deformation of $X$ over $A$. Then
\begin{enumerate}
\item Let $A \la B$ be a morphism of Artin local $k$-algebras, $X_B = X_A \times_{\text{Spec}(A)} \text{Spec}(B)$ the fiber product and $\mathcal{F}_B$
an $\sheaf_{X_B}$-module. Then there are natural isomorphisms
\[
T^i_{qG}(X_B/B,\mathcal{F}_B) \cong   T^i_{qG}(X_A/A, j_{\ast}\mathcal{F}_B)
\]
where $j \colon X_B \la X_A$ is the projection map.
\item Let $C \la B \la A$ be a sequence of ring homomorphisms and $\mathcal{F}$ an $\sheaf_{X_A}$-module. Then there is an exact sequence
\[
\la T^i_{qG}(X_A/B, \mathcal{F}) \la  T^i_{qG}(X_A/C, \mathcal{F}) \la T^i(B/C,\mathcal{F}) \la  T^{i+1}_{qG}(X_A/B, \mathcal{F}) \la
\]
\item Let \[
0 \la J \la B \la A \la 0
\]
be a square zero extension of Artin local $k$-algebras and $X_B$ a $\mathbb{Q}$-Gorenstein lifting of $X_A$ over $B$. Then there is and an
exact sequence
\[
\la T^i_{qG}(X_A/A,J \otimes_A \sheaf_{X_A}) \la  T^i_{qG}(X_B/B,\sheaf_{X_B}) \la T^i_{qG}(X_A/A,\sheaf_{X_A}) \la
T^{i+1}_{qG}(X_A/A,J \otimes_A \sheaf_{X_A}) \la
\]
\end{enumerate}

\end{proposition}
The proof of the proposition follows immediately from the corresponding ones for the usual $T^i$ after passing to the index 1 covers as before.

Next we show that $T^2_{qG}(X)=T^2_{qG}(X/k,\sheaf_X)$ is an obstruction sheaf for $Def_Y^{qG}(X)$ if $X-Y$ is smooth. For the sake of simplicity
we only present the case $X=Y$.

\begin{proposition}\label{T2qG}
Let $X$ be a $\mathbb{Q}$-Gorenstein scheme defined over a field $k$.
Let
\[
0 \la J \la B \la A \la 0
\]
be a square zero extension of finite Artin local $k$-algebras such that $m_BJ=0$ and $X_A \in Def^{qG}(X)(A)$.
Then there is a section
$ob(X_A) \in H^0(T^2_{qG}(X))\otimes_k J$ such that for any affine open
$U_A \subset X_A$, $ob(X_A)|_U \in T^2_{qG}(U)\otimes_k J$
is the obstruction to lift $U_A$
to a $\mathbb{Q}$-Gorenstein deformation $U_B$ of $U$ over $B$, where $U=U_A \otimes_A k$.
\end{proposition}
\begin{proof}
This a local result
and so we may assume that $X$ is affine of index $r$. Let $\pi \colon \tilde{X} \la X$ be the index 1 cover of $X$. Then
$T^2_{qG}(X)=(T^2(\tilde{X}))^{\mu_r}$. Let
\[
0 \la J \la B \la A \la 0
\]
be an extension of finite local Artin algebras and $X_A$ a $\mathbb{Q}$-Gorenstein deformation of $X$ over $A$. Let
$\pi_A \colon \tilde{X_A} \la X_A$ be the index 1 cover. Then $\tilde{X_A}$ is a deformation of $\tilde{X}$ over $A$~\cite{KoBa88} and
by Lemma~\ref{action}, the obstruction $ob(\tilde{X}_A) \in T^2(\tilde{X}) \otimes_k J$ is $\mu_r$-invariant and hence it is in fact
in $T^2_{qG}(X) \otimes_k J$. Hence if $ob(\tilde{X_A})=0$, then there is a deformation $\tilde{X}^{\prime}_B$ of $\tilde{X}$ over $B$, lifting
$\tilde{X}_A$. This may not be $\mu_r$-invariant but $\tilde{X}_B=\frac{1}{r}\sum_{i=0}^{r-1} \zeta^i \cdot \tilde{X}^{\prime}_B$ is,
where $\zeta$ is a primitive $r$-root of unity. Then
$X_B = \tilde{X}_B / \mu_r$ is a lifting of $X_A$ over $B$.
\end{proof}

Now that we have developed the theory of $\mathbb{Q}$-Gorenstein cotangent sheaves $T^i_{qG}(X)$, we can repeat word by word most of the arguments for the usual deformation functor $Def(Y,X)$ in section~\ref{local-to-global-section}. In particular we have the following.

\begin{theorem}\label{Q-obstructions}
Let $X$ be a $\mathbb{Q}$-Gorenstein scheme defined over a field $k$ and $Y \subset X$ a closed subscheme such that $X-Y$ is smooth. Let
\[
0 \la J \la B \la A \la 0
\]
be a small extension of local Artin $k$-algebras and $X_A \in Def^{qG}(Y,X)(A)$. Then
\begin{enumerate}
\item $Def^{qG}(X_A/A,B)$ and $Def^{qG}_{loc}(X_A/A,B)$ are $\mathbb{T}_{qG}^1(Y,X) \otimes J$ and $H^0(T_{qG}^1(X)\otimes J)$ homogeneous spaces, respectively.
\item Let $s_B \in Def^{qG}_{loc}(X_A/A,B)$. Then the set $\pi^{-1}(s_B)$ is a homogeneous space over $H^1(\widehat{T}_X \otimes J)$.
\item There is a sequence
\[
0 \la H^1(\widehat{T}_X \otimes J) \stackrel{\alpha}{\la} Def^{qG}(X_A/A,B) \stackrel{\pi}{\la}   Def^{qG}_{loc}(X_A/A,B)    \stackrel{\partial}{\la} H^2(\widehat{T}_X \otimes J)
\]
which is exact in the following sense. Let $s_B \in Def^{qG}_{loc}(X_A/A,B)$. Then $s_B$ is in the image of $\pi$ if and only if $\partial (s_B)=0$. Moreover, let $X_B, X^{\prime}_B \in Def^{qG}(X_A/A,B)$ such that $\pi(X_A)=\pi(X^{\prime}_A)$. Then there is $\gamma \in H^1(\widehat{T}_X \otimes J)$ such that $X^{\prime}_A = \gamma \cdot  X_A $, where by $"\cdot"$ we denote the action of
$H^1(\widehat{T}_X \otimes J)$ on $\pi^{-1}(s_B)$.
\end{enumerate}
\end{theorem}

\begin{corollary}\label{obstructions-4}
With assumptions as in Theorem~\ref{Q-obstructions}, there are two successive obstructions in $H^0(T_{qG}^2(X)\otimes J)$ and $H^1(T_{qG}^1(X) \otimes J)$ in order that $\mathrm{Def}^{qG}_{loc}(X_A/A,B)\not= \emptyset$, i.e., for $X_A$ to lift locally to $B$. If these obstructions vanish then there is another obstruction in $H^2(\widehat{T}_X \otimes J)$ in order that $\mathrm{Def}(X_A/A,B)\not= \emptyset$, i.e., for the local obstructions to globalize.
\end{corollary}

The local lifting method and the results that were described in subsection~\ref{local-subsection} apply immediately to the $\mathbb{Q}$-Gorenstein case as well. For the convenience of the reader, we state the main technical tools needed in order to apply it.

\begin{proposition}\label{Q-exact-sequence}
Let $X$ be a $\mathbb{Q}$-Gorenstein scheme defined over a field $k$ and $Y \subset X$ a closed subscheme such that $X-Y$ is smooth.
Let $\mathcal{X}_n \in Def_Y^{qG}(X)(A_n)$ and $\mathcal{X}_{n-1}=\mathcal{X}_n \otimes_{A_n} A_{n-1}$.
Then there is an exact sequence
\begin{gather*}
0 \la \widehat{T}_X \la \widehat{T}_{X_n/A_n} \la \widehat{T}_{X_{n-1}/A_{n-1}} \la T^1_{qG}(X) \la T^1_{qG}(X_n/A_n) \la \\
\la T^1_{qG}(X_{n-1}/A_{n-1}) \la T^2_{qG}(X)
\end{gather*}
\end{proposition}
\begin{proof}
Use Propositions~\ref{properties},~\ref{T0} on the extension
\[
0 \la k \la A_n \la A_{n-1} \la 0
\]
\end{proof}

\begin{proposition}\label{Q-local-sequence}
With assumptions as in Proposition~\ref{Q-exact-sequence}, there are canonical exact sequences
\begin{gather*}
0 \la H^0(T^1_{qG}(X)/\mathcal{F}_n) \la H^0(T^1_{qG}(\mathcal{X}_n/A_n) )\la H^0 (T^1_{qG}(\mathcal{X}_{n-1}/A_{n-1}) ) \la Q_n \la 0\\
0 \la L_n \la Q_n \la H^0(T^2_{qG}(X))\\
0 \la L_n \la H^1(T^1_{qG}(X)/\mathcal{F}_n) \la H^1(T_{qG}^1(\mathcal{X}_{n}/A_{n}))
\end{gather*}
and a noncanonical one
\begin{gather*}
0 \la H^0(T^1_{qG}(X)/\mathcal{F}_n) \la H^0(T^1_{qG}(\mathcal{X}_n/A_n) )\stackrel{\phi_n}{\la} H^0 (T^1_{qG}(\mathcal{X}_{n-1}/A_{n-1}) ) \la\\
\la H^1(T^1_{qG}(X)/\mathcal{F}_n) \oplus H^0(T^2_{qG}(X))
\end{gather*}
where $\mathcal{F}_n \subset T^1_{qG}(X)$ is the cokernel of the map $\widehat{T}_{\mathcal{X}_n/A_n}\la \widehat{T}_{\mathcal{X}_{n-1}/A_{n-1}}$.
\end{proposition}

\begin{corollary}
Suppose that the index 1 cover of every singular point of has local complete intersection singularities. Then there is an exact sequence
\[
0 \la H^0(T_{qG}^1(X)/\mathcal{F}_n) \la H^0(T_{qG}^1(\mathcal{X}_n/A_n) )\stackrel{\sigma_n}{\la} H^0 (T_{qG}^1(\mathcal{X}_{n-1}/A_{n-1}) ) \stackrel{\partial}{\la}  H^1(T_{qG}^1(X)/\mathcal{F}_n)
\]
\end{corollary}

\section{From formal to algebraic.}

For geometric applications we are interested in algebraic deformations $f \colon \mathcal{X} \la S$ of a scheme $X$ of finite type over a field $k$. However, the methods of this paper are formal and they produce only formal deformations of $X$. It is therefore of interest to know under what conditions a formal deformation is algebraic and which properties of an algebraic deformation can be read from the associated formal deformation.

The problem of whether a formal deformation is algebraic is a very difficult one. An affirmative answer is known in the case when $X$ is affine with  isolated singularities~\cite[Theorem 5.1]{Art76} and when it is projective with $H^2(X,\sheaf_X)=0$~\cite[Theorem 2.5.13]{Ser06}~\cite{Gr59}. This problem is extensively studied in~\cite{Art69}.

In general it is difficult to compare the properties of an algebraic deformation and its associated formal deformation. For example, it is possible that the formal deformation is trivial but the global one is not~\cite[Example 1.2.5]{Ser06}. In this section we state criteria in order to recognize the properties of being locally trivial and smoothing from certain properties of the corresponding formal deformation. Then we define the notion of formal smoothing which we will use in section~\ref{smoothings-section}.

The next theorem by Artin is the key to the relation between locally formally trivial and locally trivial deformations.
\begin{theorem}[Corollary 2.6,~\cite{Art69}]\label{Artin}
Let $X_1$, $X_2$ be $S$-schemes of finite type and let $x_i \in X_i$ be points, $i=1,2$. If the complete local rings $\widehat{\sheaf}_{X_i,x_i}$
are $\sheaf_S$-isomorphic, then $X_1$ and $X_2$ are locally isomorphic for the \'etale topology.
\end{theorem}

\begin{corollary}\label{formal-to-global}
Let $f \colon \mathcal{X} \la S$ be a flat morphism of schemes of finite type. Moreover, assume that either $f$ is proper or that
it is a morphism of local schemes. Let $s \in S$ and suppose that the corresponding formal deformation $X_n \la S_n$, where
$X_n= \mathcal{X} \times_S S_n$, $S_n=\mathrm{Spec} ( \sheaf_{S,s}/m_s^{n+1})$, is locally trivial. Then there is a neighborhood $s \in U \subset S$
and an \'etale cover $\{ V_i\}$ of $f^{-1}U$ such that $V_i \la U$ is trivial.
\end{corollary}
In particular, with assumptions as in the previous theorem, if the fiber over $s$, $\mathcal{X}_s$, is singular then the general fiber is
singular too and hence $f$ is not a smoothing.
\begin{proof}
If $f$ is a flat family of local schemes, then the corollary follows immediately from~\ref{Artin}. Now suppose that $f$ is proper. Let $X_s=\mathcal{X} \times_S \mathrm{Spec} k(s)$ and $\widehat{\mathcal{X}}$ be the formal completion of $\mathcal{X}$ along $X_s$. Then the assumptions imply that $\widehat{\mathcal{X}}$ is locally trivial. In particular, it follows that
$\widehat{\sheaf}_{\mathcal{X},P} \cong \widehat{\sheaf}_{\mathcal{Y},P}$, where $\mathcal{Y}=X_s \times S$,  $P \in X_s$ and $\widehat{\sheaf}_{\mathcal{X},P}$, $\widehat{\sheaf}_{\mathcal{Y},P}$ are the completions of $\sheaf_{\mathcal{X},P}$, $\sheaf_{\mathcal{Y},P}$, at the maximal ideals $m_{\mathcal{X},P}$, $m_{\mathcal{Y},P}$ of $\sheaf_{\mathcal{X},P}$ and $\sheaf_{\mathcal{Y},P}$. Hence by Theorem~\ref{Artin} it follows that there is an \'etale
cover $\{ V_i \}$ of $X_s$ in $\mathcal{X}$ such that $V_i \la S$ is trivial. Let $Z=\mathcal{X}-\cup_i V_i$. Then since $f$ is proper, $Y=f(Z)$ is closed in $S$ and $U=S-Y$ has the required properties.
\end{proof}

Let $f \colon \mathcal{X} \la \mathcal{S}$ be a deformation of a scheme $X$ over the spectrum of a discrete valuation ring. Next we will obtain criteria on the corresponding formal deformation $f_n \colon X_n \la \mathrm{Spec}A_n$ in order for $f$ to be a smoothing. First we define the relative differentials of a morphism of formal schemes.

\begin{definition}[~\cite{Lip-Na-Sa05}]
Let $\mathfrak{f} \colon \mathfrak{X} \la \mathfrak{S} $ be a morphism of formal schemes. Let $\mathfrak{I}$, $\mathfrak{J}$ be ideals of definition of $\mathfrak{X}$, $\mathfrak{S}$ respectively, such that $\mathfrak{f}^{\ast} \mathfrak{J}\cdot \sheaf_{\mathfrak{X}} \subset \mathfrak{I}$. Let $X_n=(\mathfrak{X},\sheaf_{\mathfrak{X}}/\mathfrak{I}^{n+1})$, $S_n=(\mathfrak{S},\sheaf_{\mathfrak{S}}/\mathfrak{J}^{n+1})$ the corresponding schemes and $f_n \colon X_n \la S_n$ the corresponding morphism. Then $\underset{\leftarrow}{\lim} \Omega_{X_n/S_n}$ and $\underset{\leftarrow}{\lim} \omega_{X_n/S_n}$ are sheaves of $\sheaf_{\mathfrak{X}}=\underset{\leftarrow}{\lim}\sheaf_{X_n}$-modules and we define
\[
\Omega_{\mathfrak{X}/\mathfrak{S}}=\underset{\leftarrow}{\lim} \Omega_{X_n/S_n}
\]
the sheaf of formal relative differentials and
\[
\omega_{\mathfrak{X}/\mathfrak{S}} = \underset{\leftarrow}{\lim} \omega_{X_n/S_n}
\]
the formal dulalizing sheaf. If $\mathfrak{f}$ is of pseudo-finite type, then both are coherent. In this case we also define
\[
T^1(\mathfrak{X}/\mathfrak{S})=\mathcal{E}xt^1_{\mathfrak{X}}(\Omega_{\mathfrak{X}/\mathfrak{S}}, \sheaf_{\mathfrak{X}})
\]
the first order formal relative cotangent sheaf. For the basic properties of $\Omega_{\mathfrak{X}/\mathfrak{S}}$ we refer the reader to~\cite{TaLoRo07}.
\end{definition}
Next we define the notion of a formal $\mathbb{Q}$-Gorenstein deformation $\mathfrak{f} \colon \mathfrak{X} \la \mathfrak{S}$ and the corresponding sheaf $T^1_{qG}(\mathfrak{X}/\mathfrak{S})$.
\begin{definition}
Let $\mathfrak{f} \colon \mathfrak{X} \la \mathfrak{S} $ be a flat morphism of formal schemes.
\begin{enumerate}
\item We say that $\mathfrak{f}$ is a formal $\mathbb{Q}$-Gorenstein deformation, if there are ideals of definition $\mathfrak{I}$, $\mathfrak{J}$ of $\mathfrak{X}$, $\mathfrak{S}$ respectively, such that $\mathfrak{f}^{\ast} \mathfrak{J}\cdot \sheaf_{\mathfrak{X}} \subset \mathfrak{I}$ and the corresponding deformations of schemes $f_n \colon X_n \la S_n$, where $X_n=(\mathfrak{X},\sheaf_{\mathfrak{X}}/\mathfrak{I}^{n+1})$, $S_n=(\mathfrak{S},\sheaf_{\mathfrak{S}}/\mathfrak{J}^{n+1})$, are $\mathbb{Q}$-Gorenstein.
\item Suppose that $\mathfrak{f}$ is a formal $\mathbb{Q}$-Gorenstein deformation. Then, with notation as in (1), let $\{U_i\}$ be an affine open cover of $X$ and let $X_{i,n}=X_n|_{U_i}$. Then the deformation $X_{i,n} \la S_n$ is induced by a deformation $\tilde{X}_{i,n} \la S_n$, where $\pi_{i,n} \colon \tilde{X}_{i,n} \la X_{i,n}$ is the index 1 cover~\cite{KoBa88}. These form an inverse system and setting $\tilde{\mathfrak{X}}_{i} = \underset{\leftarrow}{\lim} \tilde{X}_{i,n}$ we get a map of formal schemes $\pi_i \colon \tilde{\mathfrak{X}}_{i} \la \mathfrak{X}|_{U_i}$ which we call the formal index 1 cover. Then as in the usual scheme case, the covering groups $G_i$ act on $T^1(\mathfrak{X}_i/\mathfrak{S})$ and we define $T^1_{qG}(\mathfrak{X}_i,\mathfrak{S})= T^1(\mathfrak{X}_i/\mathfrak{S})^{G_i}$. These glue together to a coherent sheaf $T^1_{qG}(\mathfrak{X}/\mathfrak{S})$ on $\mathfrak{X}$.
\end{enumerate}
\end{definition}
\begin{notation}
Let $\mathfrak{F}$ be a coherent sheaf on a formal scheme $\mathfrak{X}$. We denote by $\mathrm{Fitt}_k(\mathfrak{F}) \subset \sheaf_{\mathfrak{X}}$ the $k$-fitting ideal of $\mathfrak{F}$. These ideals measure the obstruction for $\mathfrak{F}$ to be locally generated by $k$ elements. In fact $\mathfrak{F}$ is locally generated by $k$ elements if and only if $\mathrm{Fitt}_k(\mathfrak{F})=\sheaf_{\mathfrak{X}}$. Moreover, fitting ideals commute with base change and completion~\cite[Proposition 20.6]{Eis95}.
\end{notation}

Next we define the notion of a formal smoothing.

\begin{definition}
Let $X$ be a proper equidimensional scheme of finite type over a separable field $k$. Then a formal deformation $\mathfrak{f} \colon \mathfrak{X} \la \mathfrak{S}$, where $\mathfrak{S}=\mathrm{Specf} k[[t]]$ is called a formal smoothing of $X$, if and only if there is a $k \in \mathbb{Z}_{>0}$ such that $\mathfrak{I}^k \subset \mathrm{Fitt}_n(\Omega_{\mathfrak{X}/\mathfrak{S}})$, where $\mathfrak{I} \subset \sheaf_{\mathfrak{X}}$ is an ideal of definition of $\mathfrak{X}$ and $n = \dim X$.
\end{definition}

\begin{remark}
In the previous definition we required that $X$ is equidimensional in order to control the dimension of the components of the general fiber. However it is not a very restrictive condition since almost all singularities of interest in applications such as moduli of canonically polarized varieties and the minimal model program are Cohen-Macauley and hence equidimensional.
\end{remark}

The next proposition shows that formal smoothness implies smoothness in the case of algebraic deformations.

\begin{proposition}
Let $X$ be a proper equidimensional scheme of dimension $n$, of finite type over a separable field $k$. Let $f \colon \mathcal{X} \la S$ be a deformation of $X$ over the spectrum of a discrete valuation ring $(A,m)$, and let $\mathfrak{f}\colon \mathfrak{X} \la \mathfrak{S}$ the associated formal deformation. Then $f$ is a smoothing of $X$ if and only if $\mathfrak{f}$ is a formal smoothing of $X$.
\end{proposition}

\begin{proof}
Since $f$ is proper, it follows that the general fiber $\mathcal{X}_g = \mathcal{X} \times_S \mathrm{Spec} K(A)$, is equidimensional of dimension $n$. Assume that $\mathfrak{f}$ is formally smooth. Then, since $\mathrm{Fitt}_n(\Omega_{\mathfrak{X}/\mathfrak{S}})=\mathrm{Fitt}_n(\Omega_{\mathcal{X}/\mathcal{S}})^{\wedge}$, the formal completion of $\mathrm{Fitt}_n(\Omega_{\mathcal{X}/\mathcal{S}})$ along $X$,
the assumption implies that $\sheaf_{\mathcal{X}} / \mathrm{Fitt}_n(\Omega_{\mathcal{X}/\mathcal{S}})$ is supported on the central fiber. Therefore, $\mathrm{Fitt}_n(\Omega_{\mathcal{X}_g})=\sheaf_{\mathcal{X}_g}$ and hence $\Omega_{\mathcal{X}_g/K(A)}$ is locally generated by $n$ elements. Let $\mathcal{X}^n_g$ be an irreducible component of $\mathcal{X}_g$ and let $P \in \mathcal{X}_g$ be a closed point. Then since $\mathcal{X}_g$ is Noetherian, $\dim \sheaf_{\mathcal{X}_g,P} = n$. Let $m_P \subset \sheaf_{\mathcal{X}_g,P} = n$ be the maximal ideal. Then there is an exact sequence
\[
m_P/m_P^2 \la \Omega_{\mathcal{X}_g/K(A)}\otimes (\sheaf_{\mathcal{X}_g,P}/m_P) \la \Omega_{K(\sheaf_{\mathcal{X}_g,P})/K(A)} \la 0
\]
which is exact on the left too since $k$ is separable. Therefore, $\dim (m_P/m_P^2)=\dim \sheaf_{\mathcal{X}_g,P}$ and hence $\sheaf_{\mathcal{X}_g,P}$ is regular. in fact the proof shows that it is geometrically regular and therefore $\sheaf_{\mathcal{X}_g,P}$ is smooth. Hence $\mathcal{X}_g$ is smooth and irreducible. The converse is similar.
\end{proof}

If $X$ has either complete intersection singularities or it is $\mathbb{Q}$-Gorenstein and the index 1 cover of any of its singular points has complete intersection singularities, it is possible to give simpler criteria which we will use in section~\ref{smoothings-section}.

We will need the next easy lemma.
\begin{lemma}\label{ci}
Let $X$ be a local complete intersection scheme of finite type over a field $k$. Then if $\mathcal{E}xt^1_X(\Omega_X, \sheaf_X)=0$, $X$ is smooth.
\end{lemma}
\begin{proof}
We may assume that $X$ is affine. Then since it is complete intersection, there is an exact sequence
\[
0 \la \sheaf_X^k \la \sheaf_X^m \la \Omega_X \la 0
\]
such that $m-k = \dim X$. Since $\mathcal{E}xt^1_X(\Omega_X, \sheaf_X)=0$, it follows that the previous sequence is split exact. Hence
\[
\sheaf_X^m = \sheaf_X^k \oplus \Omega_X
\]
and therefore $\Omega_X$ is free of rank equal to the dimension of $X$. Hence $X$ is smooth.
\end{proof}

\begin{proposition}\label{formal-smoothing}
Let $X$ be a local complete intersection scheme and $f \colon \mathcal{X} \la S$ a deformation of $X$ over the spectrum of a discrete valuation ring $(A,m_A)$. Let $\mathfrak{f} \colon \mathfrak{X} \la \mathfrak{S}$ be the corresponding formal deformation. Assume that $f$ is proper and of finite type. Let $\mathfrak{I} \subset \sheaf_{\mathfrak{X}}$ be an ideal of definition of $\mathfrak{X}$. Then the following are equivalent.
\begin{enumerate}
\item The family $f \colon \mathcal{X} \la S$ is a smoothing of $X$;
\item There is $m \in \mathbb{N}$ such that $\mathfrak{I}^mT^1(\mathfrak{X}/\mathfrak{S})=0$;
\item There is $k \in \mathbb{N}$ such that for all $n \geq k$,
\[
T^1(X_{n+1}/A_{n+1})=T^1(X_n/A_n)
\]
where $X_n=\mathcal{X} \times_{S} S_n$, $S_n = \mathrm{Spec} A_n$, $A_n=A/m_A^{n+1}$.
\end{enumerate}
\end{proposition}
\begin{proof}
First we show that $(1)$ implies $(2)$. In this case, $\mathfrak{X}=\widehat{\mathcal{X}}$ is the completion of $\mathcal{X}$ along $X$. Then $\Omega_{\mathfrak{X}/\mathfrak{S}}=\widehat{\Omega}_{\mathcal{X}/S}$~\cite{TaLoRo07} and hence
\[
T^1(\mathfrak{X}/\mathfrak{S})=\mathcal{E}xt^1_{\mathfrak{X}}(\Omega_{\mathfrak{X}/\mathfrak{S}}, \sheaf_{\mathfrak{X}})=\mathcal{E}xt^1_{\mathcal{X}}(\Omega_{\mathcal{X}/S}, \sheaf_{\mathcal{X}})^{\wedge}=T^1(\mathcal{X}/S)^{\wedge}
\]
Now by Lemma~\ref{ci}, $\mathcal{X} \la S$ is a smoothing if and only if $T^1(\mathcal{X}/S)$ is supported on $X$. Since $T^1(\mathcal{X}/S)$ is a coherent $\sheaf_{\mathcal{X}}$-module, this is equivalent to say that there is $m \in \mathbb{N}$ such that $I^mT^1(\mathcal{X}/S)=0$, where $I$ is the ideal sheaf of $X$ in $\mathcal{X}$. Hence $\mathfrak{I}^mT^1(\mathfrak{X}/\mathfrak{S})=0$, where $\mathfrak{I}=\widehat{I}$. Conversely, if $\mathfrak{I}^mT^1(\mathfrak{X}/\mathfrak{S})=0$, for some $m$ and some ideal of definition $\mathfrak{I}$, it also holds for all ideals of definition and in particular for $\mathfrak{I}=\widehat{I}$. Hence $(I^mT^1(\mathcal{X}/S))^{\wedge}=0$ and therefore there is a $X \subset \mathcal{U} \subset \mathcal{X} $ an open neighborhood of $X$ in $\mathcal{X}$ such that $I^mT^1(\mathcal{X}/S)|_{\mathcal{U}}=0$ and hence since $f$ is proper and $S$ is local, $I^mT^1(\mathcal{X}/S)=0$. Hence $T^1(\mathcal{X}/S)$ is supported on $\mathcal{X}$ and therefore $f$ is a smoothing.

Next we show that $(1)$ is equivalent to $(3)$. Let $t$ be a generator of the maximal
ideal of $R$. Then the exact sequence
\[
0 \la \sheaf_{\mathcal{X}} \stackrel{t^{n+1}}{\la} \sheaf_{\mathcal{X}} \la \sheaf_{X_n} \la 0
\]
gives the exact sequence
\[
0 \la T_{\mathcal{X}/\Delta} \stackrel{t^{n+1}}{\la} T_{\mathcal{X}/\Delta} \la T_{X_n/A_n} \la
T^1(\mathcal{X}/\Delta) \stackrel{t^{n+1}}{\la}T^1(\mathcal{X}/\Delta) \la T^1(X_n/A_n)
\la 0
\]
Then $f$ is a smoothing if and only if $T^1(\mathcal{X}/S)$ is supported on $X$ and hence if and only if there is $k \in \mathbb{N}$ such that $t^k T^1(\mathcal{X}/S)=0$. Now from the previous exact sequence it follows that this is equivalent to say that $T^1(X_{n+1}/A_{n+1})=T^1(X_n/A_n)$, for all $n \geq k$.
\end{proof}
\begin{proposition}\label{Q-formal-smoothing}
Let $X$ be a $\mathbb{Q}$-Gorenstein scheme such that the index 1 cover of its singular points has complete intersection singularities only. Let $f \colon \mathcal{X} \la S$ be a $\mathbb{Q}$-Gorenstein deformation of $X$ over the spectrum of a discrete valuation ring $(A,m_A)$. Let $\mathfrak{f} \colon \mathfrak{X} \la \mathfrak{S}$ be the corresponding formal deformation. Assume that $f$ is proper and of finite type. Let $\mathfrak{I} \subset \sheaf_{\mathfrak{X}}$ be an ideal of definition of $\mathfrak{X}$. Then the following are equivalent.
\begin{enumerate}
\item The family $f \colon \mathcal{X} \la S$ is a smoothing of $X$;
\item There is $m \in \mathbb{N}$ such that $\mathfrak{I}^mT_{qG}^1(\mathfrak{X}/\mathfrak{S})=0$ and $\mathfrak{I}^m \subset \mathrm{Fitt}_1(\omega_{\mathfrak{X}/\mathfrak{S}})=0$;
\item There is $k \in \mathbb{N}$ such that for all $n \geq k$, $\mathfrak{I}^m \subset \mathrm{Fitt}_1(\omega_{\mathfrak{X}/\mathfrak{S}})=0$ and
\[
T_{qG}^1(X_{n+1}/A_{n+1})=T_{qG}^1(X_n/A_n)
\]
where $X_n=\mathcal{X} \times_{S} S_n$, $S_n = \mathrm{Spec} A_n$, $A_n=A/m_A^{n+1}$.
\end{enumerate}
\end{proposition}
\begin{proof}
The proof goes along the lines of the proof of Proposition~\ref{formal-smoothing} with a few differences that we explain next.
The condition $\mathfrak{I}^m \subset \mathrm{Fitt}_1(\omega_{\mathfrak{X}/\mathfrak{S}})=0$ means that generically over $S$, $\omega_{\mathcal{X}/S}$ is generated by one element and hence it is a line bundle. Therefore the general fiber of $f$ is Gorenstein. Hence the index 1 cover of any singularity of $\mathcal{X}$ is \'etale away from the central fiber. Now since the index 1 cover of any singular point of $X$ is assumed to be complete intersection, it follows that the general fiber of $f$ is also complete intersection. Then applying the arguments of the proof of Proposition~\ref{formal-smoothing}, we get the claimed result.
\end{proof}

\section{Smoothing criteria.}\label{smoothings-section}
Let $X$ be a proper pure and reduced scheme of finite type over a field $k$. Moreover, assume that the singular points of $X$ are either complete intersection or $\mathbb{Q}$-Gorenstein with complete intersection index 1 covers. In this section we give some smoothing and non-smoothing criteria for such schemes $X$. Following the methodology of this section and the methods developed in previous sections, one could also give similar criteria for algebraic germs $Y \subset X$. However, for the sake of simplicity we will only consider the case $X=Y$.

In what follows we denote by $D$ either $Def(X)$ or $Def^{qG}(X)$ and by $T^i_D(X)$ either $T^i(X)$ or $T^i_{qG}(X)$.

The sheaves $T^i_D(X)$ are fundamental in the study of the deformation theory of $X$. However they can be extremely complicated. The reduced part of their support is contained in the singular locus of $X$ but it may have embedded components. This happens even in the
simplest cases. For example if $X$ is the pinch point given by $x^2-y^2z=0$, then $T^1(X)=k[x,y,z]/(x,y^2,yz)$ and it has an embedded point over
 the pinch point. This makes any calculation involving $T^i_D(X)$ very difficult. So it is better to consider the pure part of $T^i_D(X)$, instead which we define next. It is just a generalization of
the notion of torsion free.
\begin{definition}
Let $X$ be a pure and reduced scheme and $\mathcal{F}$ a coherent sheaf on $X$ of dimension $d$. Let $\mathcal{F}_{d-1}\subset \mathcal{F}$ be
the maximal subsheaf of $\mathcal{F}$ of dimension at most $d-1$. Then we define,
\begin{enumerate}
\item The support of the torsion part of $\mathcal{F}$ to be the support of $\mathcal{F}_{d-1}$.
\item The rank of $\mathcal{F}$,$rk(\mathcal{F})$, by
\[
rk(\mathcal{F}) = \text{max}_{\xi} \{ \text{length$(\mathcal{F}_{\xi})$, where $\xi$ is a generic point of the support of $\mathcal{F}$} \}
\]
\item The pure part of $\mathcal{F}$, $p(\mathcal{F})$, to be the quotient $\mathcal{F}/\mathcal{F}_{d-1}$. This is pure of dimension $d$.
\end{enumerate}
\end{definition}
Let $X_n \la \mathrm{Spec}A_n$ be a deformation of $X$ over $A_n$ and let $X_{n-1}=X_n \otimes_{A_{n-1}}A_n$. Then from our discussion in sections~\ref{local-to-global-section},~\ref{Q-section}, it is follows that in order to understand the obstructions to lift $X_n$ to a deformation $X_{n+1}$ over $A_{n+1}$, it is important to study the sheaves $\mathcal{F}_n$ and $T^1_D(X)/\mathcal{F}_n$, where $\mathcal{F}_n \subset T^1_D(X)$ is the cokernel of the natural map $T_{X_n/A_n}\la T_{X_{n-1}/A_{n-1}}$. The next lemma does this in some cases.
\begin{lemma}\label{Fn}
Let $\mathcal{X} \la \Delta=\text{Spec}(R)$ be a deformation of $X$, where $(R,m)$ is a DVR. Let $X_n = \mathcal{X} \otimes_R R/m^{n+1}$ and
(as in Proposition~\ref{local-sequence})
$\mathcal{F}_n \subset T^1_D(X)$ the cokernel of the natural map $T_{X_n/A_n} \la T_{X_{n-1}/A_{n-1}}$, where $A_n =R/m^{n+1}$. Then
\begin{enumerate}
\item There is an injective map
\[
\phi \colon \widehat{T_D^1}(\mathcal{X}/\Delta) \la \lim_{\underset{n}{\leftarrow}}T_D^1(X_n/A_n)
\]
where $\widehat{T_D^1}(\mathcal{X}/\Delta)$ is the $m$-adic completion of $T^1_D(\mathcal{X}/\Delta)$. Moreover, $\phi$ is an isomorphism at any local
complete intersection point of $X$.
\item Suppose that $X$ is unobstructed at any generic point of its singular locus and that $\mathcal{X}$ is a smoothing.
Then there is $n_0 \in \mathbb{Z}$ such that
\begin{enumerate}
\item
\[
rk(T^1_D(X)/\mathcal{F}_n)=0
\]
if $n \geq n_0$, and
\item
\[
0 < rk(T^1_D(X)/\mathcal{F}_n) \leq rk(T^1_D(X))
\]
for all $n < n_0$.
\item Suppose that at any generic point $\xi$ of the singular locus of $X$, $X$ is a hypersurface singularity $(f=0) \subset \mathbb{C}^n$
with $\mu(f)=\tau(f)$, where $\mu(f)$, $\tau(f)$ are the Milnor and Tjurina numbers of $f$. Then if $\mathcal{X}$ is smooth at $\xi$,
$rk(T^1_D(X)/\mathcal{F}_n)=0$, for all $n$.
\end{enumerate}
\end{enumerate}
\end{lemma}
\begin{proof}
Let $t$ be a generator of the maximal
ideal of $R$. Then the exact sequence
\[
0 \la \sheaf_{\mathcal{X}} \stackrel{t^{n+1}}{\la} \sheaf_{\mathcal{X}} \la \sheaf_{X_n} \la 0
\]
gives the exact sequence
\begin{equation}\label{Fn-eq3}
0 \la T_{\mathcal{X}/\Delta} \stackrel{t^{n+1}}{\la} T_{\mathcal{X}/\Delta} \la T_{X_n/A_n} \la
T^1_D(\mathcal{X}/\Delta) \stackrel{t^{n+1}}{\la}T^1_D(\mathcal{X}/\Delta) \la T_D^1(X_n/A_n)
\la T^2_D(\mathcal{X}/\Delta)
\end{equation}
where $T^2_D(\mathcal{X}/\Delta)$ is a sheaf
supported on the non complete intersection singular points of $X$. Then it follows that there are injections
\[
\phi_n \colon T^1_D(\mathcal{X}/\Delta)/ t^{n+1}T^1_D(\mathcal{X}/\Delta) \la T_D^1(X_{n}/A_{n})
\]
Passing to the inverse limits we get the map $\phi$ claimed. Moreover, since $\phi_n$ are isomorphisms at any complete intersection point of $X$,
$\phi$ is an isomorphism too.

Suppose that $\mathcal{X}$ is a smoothing and that at any generic point of its singular locus, $X$ is unobstructed. Then at any generic point $\xi$ of the singular locus of $X$, $T_D^2(\mathcal{X}/\Delta)_{\xi}=0$ and the argument of the proof of proposition~\ref{formal-smoothing} shows that there is an $n_0 \in \mathbb{Z}$ such that $T_D^1(X_n/A_n)_{\xi} = T_D^1(X_{n-1}/A_{n-1})_{\xi}$ for
all $n \geq n_0$. In fact something stronger holds. Suppose that there is $k \in \mathbb{Z}$ such that $T_D^1(X_{k}/A_{k})_{\xi} = T_D^1(X_{k-1}/A_{k-1})_{\xi}$.
Then we will show that
$T_D^1(X_n/A_n)_{\xi} = T_D^1(X_{n-1}/A_{n-1})_{\xi}$, for all $n \geq k$. From~(\ref{Fn-eq3}) it follows that
\[
T_D^1(X_{n}/A_{n})_{\xi}=T^1_D(\mathcal{X}/\Delta)_{\xi}/ t^{n+1}T^1_D(\mathcal{X}/\Delta)_{\xi}
\]
for all $n$ and hence, since  $T_D^1(X_{k}/A_{k})_{\xi} = T_D^1(X_{k-1}/A_{k-1})_{\xi}$,
\[
t^{k+1}T^1_D(\mathcal{X}/\Delta)_{\xi}=t^{k}T^1_D(\mathcal{X}/\Delta)_{\xi}
\]
and consequently
\[
t^{n+1}T^1_D(\mathcal{X}/\Delta)_{\xi}=t^{n}T_D^1(\mathcal{X}/\Delta)_{\xi}
\]
for all $n \geq k$. Hence $T_D^1(X_n/A_n)_{\xi} = T_D^1(X_{n-1}/A_{n-1})_{\xi}$, for all $n \geq k$.

Moreover, by Propositions~\ref{local-sequence},~\ref{Q-local-sequence}, there is an exact sequence
\begin{equation}\label{Fn-eq2}
 0 \la T^1_D(X)/\mathcal{F}_n \la T^1_D(X_n/A_n) \stackrel{\phi_n}{\la} T^1_D(X_{n-1}/A_{n-1})
\end{equation}
and hence it follows that there is a $n_0 \in \mathbb{Z}$ such that generically along the singularities of $X$,
$\phi_n$ is an isomorphism for all $n \geq n_0$, but not if $n< n_0$. Therefore
$ rk(T^1_D(X)/\mathcal{F}_n)=0$, if $n \geq n_0$, and $0 < rk(T^1_D(X)/\mathcal{F}_n) \leq rk(T^1_D(X))$, if $n < n_0$,
as claimed.

Let $\xi \in X$ be a generic point of the singular locus of $X$ and let $K=k(\sheaf_{\mathcal{X},\xi})$.
Suppose that at $\xi$, $X$ is a hypersurface singularity given by $(f=0) \subset \mathbb{C}^n$ and $\mu(f)=\tau(f)$. If $\mathcal{X}$ is smooth
at $\xi$, then $\dim_K T^1_D(\mathcal{X}/\Delta) = \mu(f)$. But since by assumption $\mu(f)=\tau(f)=\dim_K T^1_D(X)$, it follows from~(\ref{Fn-eq3})
that $T^1_D(\mathcal{X}/\Delta) = T^1_D(X)$ and hence $tT^1_D(\mathcal{X}/\Delta)=0$. Therefore,  $t^nT^1_D(\mathcal{X}/\Delta)=0$, for all $n$,
and hence
\[
T_D^1(X_n/A_n)=T_D^1(X_{n-1}/A_{n-1})=T^1_D(\mathcal{X}/\Delta)
\]
for all $n$. Hence from~\ref{Fn-eq2} it follows that $rk(T^1_D(X)/\mathcal{F}_n)=0$, for all $n$, as claimed.
\end{proof}

The next theorem gives some conditions under which $X$ is not smoothable.
\begin{theorem}\label{non-smoothing-2}
Suppose that $H^0(p(T^1_D(X)))=0 $ and that at any generic point of the singular locus of $X$, $X$ is
complete intersection. Let $Z$ be the support of the torsion part of $T^1_D(X)$ and let
$f \colon \mathcal{X} \la \Delta$ be a one parameter deformation of $X$. Then
\begin{enumerate}
\item $X^{sing} \subset \mathcal{X}^{sing}$, where $X^{sing}$ and $\mathcal{X}^{sing}$ are the singular parts of $X$ and $\mathcal{X}$. In particular,
$\mathcal{X}$ is not smooth.
\item Suppose in addition that $H^1_Z(p(T^1_D(X)))=0$ and that at any generic point $\xi$ of the singular
locus of $X$, $X$ is analytically isomorphic to $(x_1^2+\cdots +x_k^2=0)\subset \mathbb{C}^n$. Then there is a proper
closed subset $W$ of the singular locus of $X$, such that $\mathcal{X}-W$ is locally trivial. In particular,
the general fiber $\mathcal{X}_g$ of $f$ is singular and hence $X$ is not smoothable.
\end{enumerate}
\end{theorem}
\begin{corollary}
Suppose that $T^1_D(X)$ is pure and $H^0(T^1_D(X))=0$. Suppose also that
 the general singularity of $X$ is analytically isomorphic to
$(x_1^2+\cdots +x_k^2=0)\subset \mathbb{C}^n$. Then $X$ is not smoothable.
\end{corollary}
In particular the previous corollary applies to schemes with only normal crossing singularities.

\begin{proof}[Proof of Theorem~\ref{non-smoothing-2}]
Let $\mathcal{X}\la \Delta$ be a deformation of $X$ over $\Delta=\mathrm{Spec}(R)$, where $(R,m)$ is a discrete valuation ring.
Suppose that $\mathcal{X}$ is not trivial at any generic point of the
singular locus of $X$.
Let $X_n=\mathcal{X}\times_{\Delta} \mathrm{Spec}(R/m^n)$.
By our assumptions, every section of $T^1_D(X)$ vanishes generically along the singularities of $X$.
The theorem will follow if we show that
\begin{enumerate}
\item $T^1_D(\mathcal{X}/\Delta)$ has a section $s$ that does not vanish generically along the singular locus of $X$.
\item Any section of $T_D^1(X_n/A_n)$ vanishes generically along the singular locus of $X$ for any $n$.
\end{enumerate}
Indeed, if there is a smoothing $\mathcal{X}$, then by $(1)$ there is a section $s$ of $T^1_D(\mathcal{X}/\Delta)$ that does not vanish at
any generic point of the singular locus of $X$. But then by Lemma~\ref{Fn}.1, there is a $n \in \mathbb{Z}$ such that the image $s_{n}$
of $s$ in $T^1_D(X_{n}/A_{n})$ does not vanish at any generic point of the singular locus of $X$. But this is impossible by $(2)$.

Next we show $(1)$.
Since at any generic point of the singular locus $X$, $X$ is complete intersection, it follows that there is an exact sequence
\begin{equation}\label{extension1}
0 \la f^{\ast}\omega_{\Delta}=\sheaf_{\mathcal{X}} \la \Omega_{\mathcal{X}} \la \Omega_{\mathcal{X}/\Delta} \la 0
\end{equation}
This gives a section $s$ of $\mathcal{E}xt^1_{\mathcal{X}}(\Omega_{\mathcal{X}/\Delta},\sheaf_{\mathcal{X}})=T^1(\mathcal{X}/\Delta)$.
If $\mathcal{X}$ is also $\mathbb{Q}$-Gorenstein, then this gives an element of $T^1_{qG}(\mathcal{X}/\Delta)$.
Since $X$ is pure, $X_n$ is also pure and hence there is an exact sequence
\begin{equation}\label{extension2}
0 \la \sheaf_{X_n} \la \Omega_{\mathcal{X}}\otimes_{\sheaf_{\mathcal{X}}} \sheaf_{X_n} \la \Omega_{X_n/A_n} \la 0
\end{equation}
which gives an element of $T^1(X_n/A_n)=\mathcal{E}xt^1_{X_n}(\Omega_{X_n/A_n}, \sheaf_{X_n})$. If $\mathcal{X}$ is also $\mathbb{Q}$-Gorenstein,
then this gives an element of $T^1_{qG}(X_n/A_n)$.
Next we claim that the extension~\ref{extension1} is not split, not even generically split along the singular locus of $X$.

\textbf{Case 1.} Suppose that $\mathcal{X}$ is smooth and that~\ref{extension1} is generically split along the singular locus of $X$.
Then $\Omega_{\mathcal{X}}\cong \Omega_{\mathcal{X}/\Delta} \oplus \sheaf_{\mathcal{X}}$ and hence $\Omega_{\mathcal{X}/\Delta} $ is free
and hence $\Omega_X $ is free which is of course not true. Hence in this case, ~\ref{extension1} is not even generically split.

\textbf{Case2.} Suppose that the general singularity of $X$ is analytically isomorphic to
\begin{equation}\label{singularity}
(x_1^2+\cdots +x_k^2=0)\subset \mathbb{C}^n
\end{equation}
 Then if~\ref{extension1} was generically split, then generically over the singular locus of $X$,
\begin{equation}\label{ext-equality}
\mathrm{Ext}^1_{\mathcal{X}}(\Omega_{\mathcal{X}},\sheaf_{\mathcal{X}}) \cong \mathrm{Ext}^1_{\mathcal{X}}(\Omega_{\mathcal{X}/\Delta},\sheaf_{\mathcal{X}})
\end{equation}
 Around the generic point $\xi$ of the singular locus of $X$ we may assume that $X$ is a the singularity given by~\ref{singularity}.
Hence now all Ext spaces involved are finite dimensional over $K=k(\sheaf_{\mathcal{X},\xi})$. We will now show by direct computation
that that~\ref{ext-equality} is impossible. In suitable local analytic coordinates, $X$ is given by~\ref{singularity} and, by using the
Weierstrass preparation theorem, $\mathcal{X}$ by
\[
x_1^2+\cdots x_k^2 +t^sg(x_{k+1},\ldots, x_n,t) =0
\]
where $g\neq 0$ and $t$ does not divide $g$. Straightforward calculations show that
\[
\mathrm{Ext}^1_{\mathcal{X}}(\Omega_{\mathcal{X}},\sheaf_{\mathcal{X}})=\frac{k[x_1,\cdots , x_n,t]}{(x_1,\ldots,x_k,t^s \partial g/ \partial x_{k+1},
\ldots, t^s \partial g /\partial x_{n}, st^{s-1}g+t^s\partial g/\partial t, t^sg)}
\]
and similarly
\[
\mathrm{Ext}^1_{\mathcal{X}}(\Omega_{\mathcal{X}/\Delta},\sheaf_{\mathcal{X}})=\frac{k[x_1,\ldots,x_n,t]}{(x_1,\ldots,x_k,t^s \partial g/ \partial x_{k+1},
\ldots, t^s\partial g /\partial x_{n}, t^sg)}
\]
If $\mathrm{Ext}^1_{\mathcal{X}}(\Omega_{\mathcal{X}},\sheaf_{\mathcal{X}}) \cong
\mathrm{Ext}^1_{\mathcal{X}}(\Omega_{\mathcal{X}/\Delta},\sheaf_{\mathcal{X}})$, then
\[
st^{s-1}g+t^s\partial g/\partial t \in (x_1,\ldots,x_k,t^s \partial g/ \partial x_{k+1},\ldots, t^s\partial g /\partial x_{n}, t^sg)
\]
and hence there are polynomials $h_i, h \in k[x_1,\ldots, x_n,t]$ such that
\[
st^{s-1}g+t^s\partial g/\partial t = \sum_{i=s+1}^n h_i t^s \partial g/\partial x_i +ht^s g
\]
and therefore $t | g$, which is impossible. This shows part $(1)$ of the claim.

Next we show \textit{(2)}. We proceed by induction on $n$. $n=1$ is true by assumption. By Lemma~\ref{Fn}, there is $n_0 \in \mathbb{Z}$
such that $rk((T^1_D(X)/\mathcal{F}_n)=0$, for all $n \geq n_0$ and, since $rk(T^1_D(X))=1$, $rk((T^1_D(X)/\mathcal{F}_n)=1$,
for all $n < n_0$. Hence $p(T^1_D(X)/\mathcal{F}_n)=p(T^1_D(X))$, for all $n<n_0$.

Suppose that $n <n_0$ and construct the pushout diagram
\[
\xymatrix{
0 \ar[r] & T^1_D(X)/\mathcal{F}_n \ar[r]^{\alpha_n}\ar[d]^{\beta_n} & T^1(X_n/A_n) \ar[r]\ar[d]^{\gamma_n} & T^1(X_{n-1}/A_{n-1})\ar@{=}[d] \\
0 \ar[r] & p(T^1_D(X)/\mathcal{F}_n) \ar[r] & Q_n \ar[r] & T^1(X_{n-1}/A_{n-1})
}
\]
with $\text{Ker}(\beta_n)=\text{Ker}(\gamma_n)$ supported on $Z$. Let $M_n=\text{CoKer}(\alpha_n)$. Then there is a commutative diagram
\[
\xymatrix{
0 \ar[r] & 0=H_Z^0( p(T^1_D(X))) \ar[r]\ar[d]_{f_1} & H_Z^0(Q_n) \ar[r]\ar[d]_{f_2} & H_Z^0(M_n) \ar[r]\ar[d]_{f_3}
& H_Z^1( p(T^1_D(X)))\ar[d]_{f_4}\\
0 \ar[r] & 0=H^0( p(T^1_D(X))) \ar[r] & H^0(Q_n) \ar[r] & H^0(M_n) \ar[r]\ & H^1( p(T^1_D(X)))
}
\]
Now $f_3$ is an isomorphism by induction and $H_Z^1( p(T^1_D(X)))=0$, by assumption. Hence, by the five-lemma, $f_2$ is also an isomorphism
and therefore all sections of $Q_n$ are supported on $Z$. Now there is an exact sequence
\[
0 \la \text{Ker}(\gamma_n) \la  T^1(X_n/A_n) \la Q_n \la 0
\]
and since $\text{Ker}(\beta_n)=\text{Ker}(\gamma_n)$, $\text{Ker}(\gamma_n)$ is supported on $Z$. Let $U=X-Z$. Then $T^1(X_n/A_n)|_U=Q_n|_U$ and hence,
since the sections of $Q_n$
are supported on $Z$, the sections of $T^1(X_n/A_n)$ are also supported on $Z$. Hence for all $n<n_0$, the sections of $T^1(X_n/A_n)$ are supported
on $Z$. If $n \geq n_0$, then
$rk(T^1_D(X)/\mathcal{F}_n)=0$ and hence $Z^{\prime}_n=\text{Supp}(T^1_D(X)/\mathcal{F}_n)$ is a proper subset of $X^{sing}$. By induction, all
sections of $T^1(X_{n-1}/A_{n-1})$ are supported
on a proper subset $Z_{n-1}$ of $X^{sing}$. Let $Z_n=Z_n^{\prime}\cup Z_{n-1}$ and $U_n=X-Z_n$. Then $T^1(X_n/A_n)|_{U_n}=T^1(X_{n-1}/A_{n-1})|_{U_{n}}$
and hence all sections of $T^1(X_n/A_n)$
are supported on $Z_n$. This shows $(2)$.

It remains to show part $(1)$ of the theorem. This is a local result and hence we may assume that $X$ is affine and $\mathcal{X}$ smooth.
Then from Lemma~\ref{Fn} it follows that $rk((T^1_D(X)/\mathcal{F}_n)=0$, for all $n$. Then the previous proof shows that the sections
of $T_D^1(X_n/A_n)$ vanish at any generic point of the singular locus of $X$, for all $n$, and part $(1)$ follows as before.
\end{proof}

Next we present some smoothing criteria.

\begin{theorem}\label{smoothing1}
Let $X$ be a proper pure and reduced scheme of finite type over a field $k$ of characteristic zero.
Let $D$ be either $Def(X)$ or $Def^{qG}(X)$. Moreover assume that
\begin{enumerate}
\item $X$ has complete intersection singularities if $D=Def(X)$, or
\item $X$ is locally smoothable and the index 1 cover of any singularity of $X$ has complete intersection singularities, if $D=Def^{qG}(X)$.
\end{enumerate}
Then if  $T_D^1(X)$ is finitely generated by its global sections and $H^1(T_D^1(X))=H^2(T_X)=0$, $X$ is $D$-formally smoothable.
\end{theorem}

\begin{corollary}\label{effective}
If every deformation of $X$ is effective then $X$ is $D$-smoothable.
\end{corollary}

\begin{remark}
\begin{enumerate}
\item The requirement that $X$ is proper can be replaced by the more general requirement that $Def(X)$ has a hull. 
\item The conditions of the theorem on the vanishing of the obstructions are rather restrictive but there are cases when they are satisfied. We mention 
two of them. The first case is when there is a proper morphism $f \colon X \rightarrow \mathrm{Spec} A$, such that $\dim f^{-1}(s) \leq 1$, for all 
$s \in \mathrm{Spec} A$. Then by the formal functions theorem it follows that $H^2(T_X)=0$. This is for example the case of birational maps 
with at most one-dimensional fibers. The second case is when $X$ is a Fano variety with only 
double point normal crossing singularities such that $T^1(X)$ is finitely generated by its global sections. 
Then $H^1(T^1(X))=H^2(T_X)=0$~\cite{Tzi09a}.
\end{enumerate}
\end{remark}

\begin{proof}[Proof of Theorem~\ref{smoothing1}]
For simplicity I will only do the case when $D=Def(X)$. The $\mathbb{Q}$-Gorenstein case is exactly similar. One has only to lift the following 
argument to the index 1 covers.

The conditions of the theorem imply that $\text{Def}(X)$ exists and is smooth.
Let $s_1, \ldots , s_k \in H^0(T^1(X))$ be sections that generate $T^1(X)$. Since $\text{Def}(X)$ is smooth, the sections $s_1, \ldots , s_k$
lift to a formal deformation $f_n \colon Y_n \la S_n$ of $X$ over $S_n=\mathrm{Spec} (S/m_S^{n+1})$, where $S=k[[t_1, \ldots, t_k]]$
and $m_S$ its maximal ideal. Let $f \colon \mathcal{Y} \la \mathcal{S}$ be the corresponding morphism of formal schemes.
We will show that $\mathcal{Y}$ is smooth over $\mathrm{Specf} K(S)$. Let $U \subset X$ be the smooth locus of $X$. Then $f|_U$ is smooth and hence since
$X$ is pure, it follows that there is an exact sequence~\cite{TaLoRo07}
\begin{equation}\label{formal-seq1}
0 \la f^{\ast} \widehat{\Omega}_{\mathcal{S}} \la \widehat{\Omega}_{\mathcal{Y}} \la \widehat{\Omega}_{\mathcal{Y}/\mathcal{S}} \la 0
\end{equation}
Moreover, $\widehat{\Omega}_{\mathcal{S}}=\widehat{\Omega}_R^{\triangle}\cong \sheaf_{\mathcal{S}}^k$, where $R=k[t_1,\ldots, t_k]$.
Hence $f^{\ast} \widehat{\Omega}_{\mathcal{S}}=\sheaf_{\mathcal{Y}}^k$ and dualizing the previous sequence we get
\begin{equation}\label{formal-seq2}
\mathrm{Hom}_{\mathcal{Y}}(\widehat{\Omega}_{\mathcal{Y}},\sheaf_{\mathcal{Y}}) \la \sheaf_{\mathcal{Y}}^k \stackrel{\phi}{\la} T^1(\mathcal{Y}/\mathcal{S}) \la T^1(\mathcal{Y}) \la 0
\end{equation}
But by construction $\phi$ is surjective and therefore
\[
T^1(\mathcal{Y})=\underline{\text{Ext}}^1_{\mathcal{Y}}(\widehat{\Omega}_{\mathcal{Y}},\sheaf_{\mathcal{Y}})=0
\]
\textbf{Claim:} $\sheaf_{\mathcal{Y}}$ and $\sheaf_{\mathcal{Y}}\otimes_{S}K(S)$ have smooth local rings.

The result is local and hence we may assume that $X$ is affine given by $\sheaf_X = k[x_1, \ldots, x_m]/(f)$, where $(f)=(f_1, \ldots , f_s)$
is a complete intersection. Then $\sheaf_{X_n}=S_n[x_1, \ldots , x_m] /(f_n)$, where $(f_n)$ is a lifting of $(f)$ on $S_n[x_1, \ldots, x_m]$ and hence
\begin{equation}\label{formal-ring}
\sheaf_{\mathcal{Y}}=\underset{\leftarrow}{\lim} \sheaf_{X_n} =
\frac{S[x_1,\ldots , x_m]^{\wedge}}{(\overline{f}_1, \ldots , \overline{f}_s)}
\end{equation}
where $S[x_1,\ldots , x_m]^{\wedge}$ is the $m_S$-adic completion of $S[x_1,\ldots , x_m]$ and $\overline{f}_i = \underset{\leftarrow}{\lim} f_i^{(n)}$.
Let
\[
0 \la \sheaf_X^r \la \sheaf_X^m \la \Omega_X \la 0
\]
be a presentation of $\Omega_X$, where $m-r = \dim \sheaf_X$. Then this exact sequence lifts to compatible exact sequences
\[
0 \la \sheaf_{X_n}^r \la \sheaf_{X_n}^m \la \Omega_{X_n} \la 0
\]
Moreover, $\sheaf_{\mathcal{Y}}= \underset{\leftarrow}{\lim}\sheaf_{X_n}$, and hence taking inverse limits and taking into consideration that
$\widehat{\Omega}_{\mathcal{Y}}= \underset{\leftarrow}{\lim} \Omega_{\sheaf_{X_n}}$~\cite{TaLoRo07} we get an exact sequence
\[
0 \la \sheaf_{\mathcal{Y}}^r \la \sheaf_{\mathcal{Y}}^m \la \widehat{\Omega}_{\mathcal{Y}} \la 0
\]

This extension is trivial since $\underline{\text{Ext}}^1_{\mathcal{Y}}(\widehat{\Omega}_{\mathcal{Y}},\sheaf_{\mathcal{Y}})=0$. Hence
\[
\sheaf_{\mathcal{Y}}^m \cong   \sheaf_{\mathcal{Y}}^r   \oplus  \widehat{\Omega}_{\mathcal{Y}}
\]
and hence $\widehat{\Omega}_{\mathcal{Y}}$ is locally free of the rank claimed. This implies that $\sheaf_{\mathcal{Y}}$ has geometrically
regular local rings. Indeed, let $P \in X$ be a point and $m_P$ the maximal ideal of $\sheaf_{\mathcal{Y},P}$. Then since $k$ is perfect it
follows that~\cite{Eis95}
\[
m_p/m_P^2 \cong \widehat{\Omega}_{\mathcal{Y}} \otimes k(P)
\]
and therefore $\dim_{k(P)} m_P/m_P^2 = \dim \sheaf_{\mathcal{Y},P}$ and hence $\sheaf_{\mathcal{Y},P}$ is geometrically regular and therefore smooth. Since any localization of $\sheaf_{\mathcal{Y}}$ is a localization of $\sheaf_{\mathcal{Y},P}$, for some $P\in X$, it follows that $\sheaf_{\mathcal{Y}}$ has smooth local rings. In particular, since any localization of $\sheaf_{\mathcal{Y}}\otimes_{S}K(S)$ is a localization of $\sheaf_{\mathcal{Y}}$, $\sheaf_{\mathcal{Y}}\otimes_{S}K(S)$ is smooth.

Since $\widehat{\Omega}_{\mathcal{Y}}\cong \sheaf_{\mathcal{Y}}^d$, where $d = \dim X$, the sequence~\ref{formal-seq2} becomes
\[
\sheaf_{\mathcal{Y}}^d \stackrel{\psi}{\la} \sheaf_{\mathcal{Y}}^k \stackrel{\phi}{\la} T^1(\mathcal{Y}/\mathcal{S}) \la T^1(\mathcal{Y}) \la 0
\]
and as in the usual scheme case, $\psi$ is given by the jacobian matrix $J=(\partial \overline{f}_i /x_j)$. Since $\sheaf_{\mathcal{Y}}\otimes_{S}K(S)$ is smooth, it follows that $J$ has maximum rank at all localizations of $\sheaf_{\mathcal{Y}}\otimes_{S}K(S)$. Therefore, $\psi \otimes_S K(S)$ is surjective and hence $T^1(\mathcal{Y}/\mathcal{S})$ is torsion over $\mathcal{S}$. Hence there is a formal arc $\Delta=\mathrm{Specf}k[[t]] \la \mathrm{Spec}S$ such that in the fiber $\mathcal{X}=\mathcal{Y} \times_{\mathrm{Specf}S} \mathrm{Specf}k[[t]]$, $T^1(\mathcal{X}/\Delta)$ is torsion over $k[[t]]$ and hence there is $l \in \mathbb{N}$ such that $t^lT^1(\mathcal{X}/\Delta)=0$ and therefore $\mathcal{X} \la \Delta$ is a formal smoothing of $X$.

\end{proof}

The previous proof shows that
\begin{corollary}
With assumptions as in Corollary~\ref{effective}, suppose that $T_D^1(X)=\sheaf_Z$, where $Z$ is the singular locus of $X$. Then there is
a smoothing $f \colon \mathcal{X} \la \Delta$ of $X$ such that
\begin{enumerate}
\item $ \mathcal{X}$ is smooth if $D=Def(X)$;
\item the singularities of $\mathcal{X}$ are smooth quotients if $D=Def^{qG}(X)$.
\end{enumerate}
\end{corollary}

There is one nice and very simple case when $T^1(X)$ is finitely generated by its global sections.
\begin{corollary}
Let $X$ be a projective local complete intersection field over a field $k$ of characteristic zero. Let $X \subset Y$ be an embedding such that $Y$ is smooth.
Suppose that $\mathcal{N}_{X/Y}$ is finitely generated by its global sections and
$H^1(T^1(X))=H^2(T_X)=0$. Then $X$ is formally smoothable.
\end{corollary}
\begin{proof}
Dualizing the conormal sequence for $X \subset Y$ we get a surjection
\[
\mathcal{N}_{X/Y} \la T^1(X) \la 0
\]
Hence $T^1(X)$ is finitely generated by its global sections too and hence $X$ is formally smoothable.
\end{proof}
Next we give a similar criterion for $\mathbb{Q}$-Gorenstein deformations.
\begin{corollary}
Let $X$ be a projective $\mathbb{Q}$-Gorenstein scheme defined over a field $k$ of characteristic zero. Suppose that its Gorenstein points are complete
intersections and the high index points
are complete intersection quotients. Let $X \subset Y$ be an embedding such that locally around any point $P \in X$, $P \in Y$ is a general
deformation of $P \in X$.  Suppose that $\mathcal{N}_{X/Y}$ is finitely generated by its global sections and
$H^1(T_{qG}^1(X))=H^2(T_X)=0$. Then $X$ has a $\mathbb{Q}$-Gorenstein smoothing.
\end{corollary}

\begin{proof}
Dualizing the conormal sequence for $X \subset Y$ we get a sequence
\[
\mathcal{N}_{X/Y} \stackrel{\phi}{\la} T^1(X) \la \underline{Ext}^1_{X}(\Omega_Y \otimes \sheaf_X , \sheaf_X) \la 0
\]
We claim that $\mathrm{Im}(\phi)=T^1_{qG}(X)$ and hence if $\mathcal{N}_{X/Y}$ is generated by global sections, so is $T^1_{qG}(X)$. The claim
is local at the singularities of $Y$ and hence we may assume that $Y$ is affine. By assumption, $Y$ is smooth at any index 1 point and in
this case we are done. Asssume then that $Y$  index $ r > 1$. Let $ \pi \colon \tilde{Y} \la Y$
be the index 1 cover. Then $\tilde{X}= \pi^{-1}(X)$ is the index 1 cover of $X$.
Moreover, since by assumption $Y$ is the general deformation of $X$, $\tilde{Y}$ is smooth and hence there is a surjection
\[
\mathcal{N}_{\tilde{X}/\tilde{Y}} \la T^1(\tilde{X}) \la 0
\]
Let $G$ be the Galois group of the $\pi$. Then taking invariants we get that
\[
\mathcal{N}_{\tilde{X}/\tilde{Y}}^G = \mathcal{N}_{X/Y} \la T^1_{qG}(X) \la 0
\]
as claimed.

\end{proof}
In general, if $X \subset Y$ and $\mathcal{N}_{X/Y}$ is finitely generated by its global sections, or even better, ample,
then $X$ has nice deformation properties. Considering cases with respect to the singularities of $X$ (like normal crossings) and the shape
of the singular locus of $X$, one can get various kinds of criteria similar to the previous corollary for the smoothability of $X$, without even
refering to $T^1(X)$.
Let $Z$ be the the singular locus of $X$. In general $T^1(X)$ is not a sheaf of $\sheaf_Z$-modules. It usually has an embeded part and in fact sometimes
even $Z$ is an embeded component of its support (this for example happens if $X$ is given by $xy+z^n=0$ in $\mathbb{C}^4$, $n \geq 3$). So it is
rather difficult to describe $T^1(X)$ directly and check if it is generated by its global sections. However, if the singular locus of $X$ is
1-dimensional, then it is possible to give criteria for the finite generation of $T^1(X)$ without any reference to its embeded part.

\begin{theorem}\label{smoothing2}
Let $X$ be projective scheme with singularities as in Theorem~\ref{smoothing1} and let $Z$ be its reduced singular locus.
Suppose that $\dim Z =1$ and that
\begin{enumerate}
\item $p(I_Z^k T_D^1(X)/I_Z^{k+1}T_D^1(X))$ is generated by its global sections, for all $k \geq 0$.
\item
\[
H^1(p(I_Z^k T_D^1(X)/I_Z^{k+1}T_D^1(X)))=0
\]
for all $k \geq 0$.
\item \[
H^1(p(T_D^1(X) \otimes \sheaf_Z))=H^2(T_X)=0
\]
\end{enumerate}
Then $X$ is $D$-formally smoothable.
\end{theorem}
\begin{proof}
Let $\mathcal{F}_k \subset I_Z^k T_D^1(X)/I_Z^{k+1}T_D^1(X)$ be the maximal zero-dimensional subsheaf of $I_Z^k T_D^1(X)/I_Z^{k+1}T_D^1(X)$.
Then there is an exact sequence
\[
0 \la H^0(\mathcal{F}_k) \la H^0(I_Z^k T_D^1(X)/I_Z^{k+1}T_D^1(X)) \la H^0(p(I_Z^k T_D^1(X)/I_Z^{k+1}T_D^1(X))) \la 0
\]
Hence if $p(I_Z^k T_D^1(X)/I_Z^{k+1}T_D^1(X))$ is generated by its global sections, so is $I_Z^k T_D^1(X)/I_Z^{k+1}T_D^1(X)$. There are also exact sequences
\[
0 \la I_Z^k T_D^1(X)/I_Z^{k+1}T_D^1(X) \la T_D^1(X) / I_Z^{k+1}T_D^1(X) \la T_D^1(X) / I_Z^{k}T_D^1(X)\la 0
\]
for all $k \geq 0$. Hence by induction, $T_D^1(X) / I_Z^{k}T_D^1(X)$ is finitely generated by its global sections, for all $k$. But since $T_D^1(X)$
is supported on $Z$, $I_Z^m T_D^1(X)=0$, for $m$ sufficiently large. Hence $T_D^1(X)$ is finitely generated by its global sections and hence
by Theorem~\ref{smoothing1} $X$ is D-smoothable.
\end{proof}
If $X$ is normal crossings at any generic point of its singular locus, then $p(T^1(X))$ is an $\sheaf_Z$-module and hence one needs only to take $k=0$
in the conditions of the theorem.
\begin{corollary}\label{smoothing3}
With assumptions as in the previous theorem, suppose in addition that $X$ is normal crossings at any generic point of its singular locus,
$p(T_D^1(X))$ is finitely generated by its global sections and that $H^1(p(T_D^1(X)))=H^2(T_X)=0$. Then $X$ is smoothable.
\end{corollary}

\section{Examples}\label{examples}
In this section we apply the theory developed in the previous parts of the paper to give some examples from the theory of moduli spaces of stable surfaces and the three dimensional minimal model program.

\textbf{1.}  In this example we construct a few classes of locally but not globally smoothable stable surfaces with normal crossing singularities. 
Hence the 
irreducible components of the moduli space of stable surfaces that they belong to, do not contain any smooth surfaces of general type. Therefore these 
are extra components that appear after the moduli space of surfaces of general type is compactified by adding the stable surfaces.

\textbf{1.1.} Let $X$ be a projective surface with exactly one singular point $P$ such that
\renewcommand{\labelenumi}{\alph{enumi}.}
\begin{enumerate}
\item $K_X=kA$, where $A$ is very ample and $k \geq 2$ is an integer.
\item $P \in X$ is analytically isomorphic to the cone over a smooth projective plane curve of degree 4.
\end{enumerate}
Note that such surfaces do exist. Take for example $X \subset \mathbb{P}^3$ given by $(x_0^2+x_3^2)x_0^4+(x_1^2+x_3^2)x_1^4+(x_2^2+x_3^2)x_2^4=0$.

Let $f \colon Y \la X$ be the blow up of $X$ along $P$. Then $Y$ is smooth and the $f$-exceptional divisor is a smooth curve $E \subset \mathbb{P}^2$
of degree 4 such that $\mathcal{N}_{E/Y}=\mathcal{O}_E(-1)$ and hence $E^2=-4$. Moreover, a straightforward calculation shows that
\[
K_Y=f^{\ast}K_X -2E.
\]
Let $Z$ be obtained by glueing two copies of $Y$ along $E$. This is a surface with normal crossing singularities and I claim that $K_Z$ is ample
and $Z$ is not smoothable.

By~\cite{Fr83} or~\cite{Tzi08}, $T^1(Z)=\mathcal{N}_{E/Y} \otimes \mathcal{N}_{E/Y} = \sheaf_E(-2)$ and hence $H^0(T^1(X))=0$.
Hence by Theorem~\ref{non-smoothing-2}, $Z$ is not smoothable.

Next we show that $K_Z$ is ample. For this it suffices to show that $K_Z|_{Z_i}$, $i=1,2$, is ample, where
$Z_i \cong Y$ are the irreducible components of $Z$. It is not difficult to see that
\[
K_Z|_{Z_i}=K_Y+E=f^{\ast}K_X -E
\]
This is ample if and only if $(f^{\ast}K_X-E)^2 > 0 $ and $(f^{\ast}K_X-E)\cdot D >0$, for any irreducible curve $D \subset Y$.
Now
\[
(f^{\ast}K_X-E)^2=K_X^2-4=k^2A^2-4 >0
\]
since $k \geq 2$ and $A$ is very ample, hence $A^2 >1$. Let $D \subset Y$ be an irreducible curve and $C = f_{\ast} D$. Then
\[
(f^{\ast}K_X-E)\cdot D=K_X \cdot C - E \cdot D = k A \cdot C - m_P(C) = k \deg (C) -m_P(C)
\]
where $m_P(C)$ is the multiplicity of $C$ at $P$ and $\deg C$ is the degree of $C$ with respect to the embedding defined by $A$.
Then $\deg(C) \geq m_P(C)$ and hence since $ k \geq 2$ it follows that
\[
(f^{\ast}K_X-E)\cdot D > 0
\]
 and therefore $K_Z$ is ample as claimed.

\textbf{1.2} Let $X$ be a smooth projective surface with $K_X$ ample. Suppose that $X$ contains a smooth curve $C$ with $p_a(C) \geq 2$
and $K_X \cdot C >2(p_a(C)-1)$.

Such surfaces do exist. For example, let $C \subset \mathbb{P}^3$ be a smooth plane curve of degree $k \geq 4$ given by $f_k(x,y,z)=t=0$,
where $f_k(x,y,z)$ is a homogeneous polynomial of degree $k \geq 4$, and $X \subset \mathbb{P}^3$ be the hypersurface of degree $d > k+1$ given by
\[
g_{d-k}(x,y,z,t)f_k(x,y,z)+th_{d-1}(x,y,z,t)=0,
\]
where $g_{d-k}(x,y,z,t)$ and $h_{d-1}(x,y,z,t)$ are homogeneous polynomials of degrees $d-k$ and $d-1$. For general choice of $g_{d-k}(x,y,z,t)$ and
$h_{d-1}(x,y,z,t)$, $X$ is a smooth surface containing $C$. Moreover, $\mathcal{O}_X(K_X)=\mathcal{O}_X(d-4)$ and hence
\[
K_X \cdot C = \deg \mathcal{O}_C(d-4)=(d-4)k>k^2-3k=2(p_a(C)-1),
\]
since $d >k+1$ and $k \geq 4$. Moreover, $K_X$ is ample.

Let $Z$ be obtained by glueing two copies of $X$ along $C$.  This is a surface with normal crossing singularities and I claim that $K_Z$ is ample
and $Z$ is not smoothable.

Let $Z_1$, $Z_2$ be the two irreducible components of $Z$. By construction, $Z_i \cong Z$, $i=1,2$. Then $K_Z$ is ample if and only if $K_Z|_{Z_i}$
is ample, $i=1,2$.
As in the previous example, $K_Z|_{Z_i}=K_X+C$. By construction, $K_Z+C$ is ample and hence $K_Z$ is ample.

By adjunction, $\mathcal{N}_{C/X}=\omega_C \otimes \omega_X^{-1}$ and therefore
\[
\deg \mathcal{N}_{C/X}=2p_a(C)-2-K_X \cdot C <0,
\]
by assumption. As in the previous example, $T^1(X)=\mathcal{N}_{C/X} \otimes \mathcal{N}_{C/X}$. Hence $T^1(X)$ is a line bundle on $C$ of negative
degree. Therefore $H^0(T^1(X))=0$ and $X$ is not smoothable by Theorem~\ref{non-smoothing-2}.

\textbf{2.} In this example we construct a terminal 3-fold divisorial extremal neighborhood $f \colon Y \la X$ such that the general member of $|\mathcal{O}_Y|$
is not normal.

Let $U$ be the germ of a smooth surface around the configuration of rational curves\[
\stackrel{-2}{\circ} \mbox{\noindent ---} \stackrel{-2}{\circ} \mbox{\noindent ---}
\stackrel{-2}{\circ} \mbox{\noindent ---} \stackrel{-3}{\circ} \mbox{\noindent ---}
\stackrel{-2}{\bullet} \mbox{\noindent ---} \stackrel{-3}{\circ} \mbox{\noindent ---}
\stackrel{-1}{\bullet} \mbox{\noindent ---}
\stackrel{-2}{\circ} \mbox{\noindent ---} \stackrel{-5}{\circ}
\]
Let $h \colon U \la \tilde{Z}$ be the contraction of all the curves except the ones marked by a solid circle.
Then we get a map $\tilde{Z} \stackrel{\tilde{f}}{\la} T$ contracting two smooth rational curves $C_1$ and $C_2$ to a point $0 \in T$ such
 that $0 \in T$ is
an $A_5$ singularity, and $\tilde{Z}$ has exactly three singular points $P_1 \in C_1$, $P_2\in C_2$ and $Q = C_1 \cap C_2$.
It is easy to see that $(P_1 \in \tilde{Z}) \cong 1/9 (1,5)$, $(P_2 \in \tilde{Z}) \cong 1/9(1,-5)$,
and $(Q \in \overline{Z}) \cong 1/3(1,1)$. Let $Z$ be obtained from $\tilde{Z}$ by identifying $C_1$ and $C_2$ via an involution of $C_1+C_2$ taking
$P_1$ to $P_2$ and leaving $Q$ fixed. Let $\pi \colon \tilde{Z} \la Z$ be the quotient map. Then the singular locus of $Z$ is a smooth rational
curve $C$, $\pi^{-1}(C) = C_1+C_2$ and $Z$
 has one singularity analytically isomorphic to $(xy=0)/\mathbb{Z}(5,-5,1)$, one degenerate cusp analytically isomorphic to $x^3+y^3+xyz=0$ and is
normal crossings at all other singular points.
Moreover, $\tilde{Z}$ is the normalization of $Z$ and there is a natural morphism $ f \colon Z \la T$ contracting $C$ to $0 \in T$.

Straightforward calculations show that $K_Z \cdot C =-1/9 <0$. Moreover since $U$ is the minimal log-resolution of $C \subset Z$,
$\deg p(T^1_{qG}(Z)) = -2-1+1+3=1$~\cite{Tzi08} and hence $p(T^1_{qG}(Z))=\mathcal{O}_{\mathbb{P}^1}(1)$. Hence by Corollary~\ref{smoothing3}, there
exists a
$\mathbb{Q}$-Gorenstein smoothing $Y \la \Delta$ of $Z$. Now $f$ extends to a morphism $g \colon Y \la X$ over $\Delta$, where $X$ is a deformation
of $T$~\cite{Ko-Mo92}. Now $g \colon Y \la X$ is
a 3-fold extremal neighborhood and $Z \in |\mathcal{O}_Y|$ is the general member. Moreover, the neighborhood is divisorial since $X$ is Gorenstein.

Finally I would like to remark that the method of producing 3-fold extremal neighborhoods by deforming birational surface morphisms $f \colon Z \la T$
is fundamental in the classification of
flips by Koll\'ar and Mori~\cite{Ko-Mo92} and in principle it could be used in higher dimensions in order to understand higher dimensional flips and
divisorial contractions.

\end{document}